\documentclass[english,11pt,leqno]{article}

\usepackage{tikz, tikz-3dplot}
\usetikzlibrary{decorations.markings}
\usetikzlibrary{decorations.pathreplacing}
\usetikzlibrary{calc}

\usepackage{bbm}

\usepackage{amsxtra}
\usepackage{amsmath}
\usepackage{amssymb}
\usepackage{amsfonts}
\usepackage[all,2cell]{xy}
\UseAllTwocells
\usepackage{mathrsfs}
\usepackage{amsthm}
\usepackage{enumitem}
\usepackage{hyperref}
\usepackage{subfigure}

\setlist[enumerate]{leftmargin=*,labelindent=.5pc}
\def\CC{\mathbb{C}}

\def\GG{\mathbb{G}}
\def\HH{\mathbb{H}}

\def\PP{\mathbb{P}}
\def\RR{\mathbb{R}}
\def\ZZ{\mathbb{Z}}
\def\QQ{\mathbb{Q}}
\def\TT{{\mathbb{T}}}

\def\scrF{\mathscr{F}}
\def\F{\scrF}

\def\Gen{\mathfrak{G}}

\def\Ac{\mathcal{A}}

\def\Gc{\mathcal{G}}

\def\Lc{\mathcal{L}}
\def\Mc{\mathcal{M}}
\def\Nc{\mathcal{N}}

\def\Oc{\mathcal{O}}
\def\Pc{\mathcal{P}}

\def\Tc{\mathcal{T}}
\def\Uc{\mathcal{U}}
\def\Xc{\mathcal{X}}
\def\Yc{\mathcal{Y}}

\def\Ab{\mathbf{A}}

\def\Cb{\mathbf{C}}

 \def\Aut{\operatorname{Aut}\nolimits}

 \def\conf{{\on{conf}}}
  \def\Conf{{\on{Conf}}}

\def\Ho{\operatorname{Ho}}
 \def\Hom{\operatorname{Hom}}
 
 \def\id{\operatorname{id}}

 \def\Id{\operatorname{Id}\nolimits}

 \def\Ker{\operatorname{Ker}\nolimits}

 \def\Mod{\operatorname{Mod}}
 
\def\NC{\operatorname{NC}}
\def\Ob{\operatorname{Ob}\nolimits}
\def\on{\operatorname}

\def\pt{\operatorname{pt}\nolimits}
\def\pure{{\on{pure}}}

\def\Res{\operatorname{Res}\nolimits}

\def\Vect{\operatorname{Vect}}

\def\3x3{\operatorname{3}\times\operatorname{3}}
\def\1{{\bf 1}}
\def\lra{\longrightarrow}
\def\lla{\longleftarrow}
\def\(({(\hskip -1mm (}
\def\)){)\hskip -1mm )}
\def\((({(\hskip -1mm (}
\def\))){)\hskip -1mm )}

\def\bbs{{\backslash\hskip -1mm \backslash}}
 
\def\Bbs {{\bigl\backslash \hskip -.1cm \bigl\backslash}}

\def\be{\begin{equation}}
\def\ee{\end{equation}}
\def\ed{\end{document}}

\def\ind{\varinjlim{}}
\def\pro{\varprojlim{}}

\def\hopro{\operatorname{ho}\!\varprojlim{}}
 \def\k{\mathbf{k}}

\usepackage{MnSymbol}
\usepackage{euscript}

\def\bbs{\backslash \hskip -.1cm \backslash}
\def\Bbs {\bigl\backslash \hskip -.1cm \bigl\backslash}

\def\op{{\operatorname{op}}}

\def\Fr{\on{Fr}}
\def\Fun{\operatorname{Fun}}

\def\Gstr{\GG\on{Str}}
\def\Gcstr{G_\conf\on{Str}}

\def\ho{\operatorname{ho}\!}

\def\hopro{\ho \underleftarrow{\lim}{}}
\def\Hom{\operatorname{Hom}}

\def\Klein{\on{Klein}}

\def\N{\operatorname{N}}

\def\O{{\mathcal O}}

\def\orr{{\on{or}}}

\def\p{\mathbbm{p}}

\def\Riem{\on{Riem}}

\def\Cat{{\mathcal Cat}}
\def\Groupoids{{\bf Grpds}}

\def\cm{\langle m \rangle}
\def\cn{\langle n \rangle}

\def\Top{ {\mathcal Top}}

\def\Teich{\on{Teich}}
\def\Set{ {\mathcal Set}}

\def\FC{\operatorname{FC}}

\def\B{ {\EuScript B}}
\def\C{{\EuScript C}}
\def\D{{\EuScript D}}

\def\Top { {\mathcal Top}}

\def\Set{ {\mathcal Set}}

\def\<<{\langle {}\hskip -.1cm {}\langle}
\def\>>{\rangle \hskip -.1cm \rangle}
\def\Yc{{\mathcal Y}}

\def\FC{\on{FC}}

\def\Tr{\on{Tr}}

\def\centerarc[#1](#2)(#3:#4:#5)  { \draw[#1] ($(#2)+({#5*cos(#3)},{#5*sin(#3)})$) arc (#3:#4:#5); }

\def\Zt{\mathbb{Z}/2\mathbb{Z}}
\def\BZt{{\mathcal B}\mathbb{Z}/2\mathbb{Z}}
\def\B{{\mathcal B}}
\def\Cat{{\EuScript Cat}}

\def\dgcat{\on{dgcat}}
\def\dgcatt{\dgcat^{(2)}}

\def\MF{\on{MF}}

\def\cm{\langle m \rangle}
\def\cn{\langle n \rangle}
\def\dm{\overline{\langle m \rangle}}
\def\dn{\overline{\langle n \rangle}}
\def\pn{\widetilde{n}}
\def\ppm{\widetilde{m}}
\def\kcn{\widetilde{n}_N}
\def\kcm{\widetilde{m}_N}

\def\G{ {\mathcal G}}

\def\DGgraph{\DG-\on{Graph}}
\def\Diff{\on{Diff}}
\def\Isom{\on{Isom}}

\def\DG{ {\Delta \Gen}}

\def\DW{ {\Delta {\mathfrak W}}}
\def\CG{{ {\EuScript C}_G}}

\def\Homeo{{ \on{Homeo}}}
\def\FSet{{\bf FSet}}

 \def\Tess{\on{Tess}}
  \def\Bar{{\on{Bar}}}
  \def\bar{{\on{bar}}}

\setlist[enumerate,1]{label=(\arabic{*})}
\setlist[enumerate,2]{label=(\alph{*})}
\setlist[enumerate,3]{label=(\roman{*})}

\newtheorem{thm}{Theorem}[section]
\newtheorem{lem}[thm]{Lemma}
\newtheorem{cor}[thm]{Corollary}
\newtheorem{prop}[thm]{Proposition}

\theoremstyle{definition}
\newtheorem{defi}[thm]{Definition}
\newtheorem{exa}[thm]{Example}
\newtheorem{exas}[thm]{Examples}

\theoremstyle{remark}
\newtheorem{rem}[thm]{Remark}
\newtheorem{refo}[thm]{Reformulation}
\newtheorem{rems}[thm]{Remarks}

\numberwithin{equation}{section}

\title{Crossed simplicial groups and structured surfaces}
\author{T. Dyckerhoff \footnote{Hausdorff Center for Mathematics, Bonn, email:{\tt
dyckerho@math.uni-bonn.de}}, M. Kapranov \footnote{Kavli IPMU, Japan, email:{\tt mikhail.kapranov@ipmu.jp}}}

    \newenvironment{dedication}
        {\vspace{6ex}\begin{quotation}\begin{center}\begin{em}}
        {\par\end{em}\end{center}\end{quotation}\vspace{10ex}}

\renewcommand{\thesection}{\Roman{section}}

\begin{document}
\maketitle

\begin{dedication}Dedicated to the memory of Jean-Louis Loday
\end{dedication}

\begin{abstract}
	We propose a generalization of the concept of a ribbon graph suitable to provide
	combinatorial models for marked surfaces equipped with a G-structure. Our main insight is 
	that the necessary combinatorics is neatly captured in the concept of a crossed simplicial group 
	as introduced, independently, by Krasauskas and Fiedorowicz-Loday. In this context, Connes'
	cyclic category leads to ribbon graphs while other crossed simplicial groups naturally yield
	different notions of structured graphs which model unoriented, N-spin, framed, etc,
	surfaces.
	Our main result is that structured graphs provide orbicell decompositions of the respective
	G-structured moduli spaces. As an application, we show how, building on our theory of
	2-Segal spaces, the resulting theory can be used to construct categorified state sum
	invariants of G-structured surfaces.
\end{abstract}

\newpage

\tableofcontents
\numberwithin{equation}{section}

\addcontentsline{toc}{section}{Introduction}

\section*{Introduction}

Ribbon graphs form a fundamental tool in the combinatorial study of moduli spaces of Riemann
surfaces and of the associated mapping class groups \cite{kontsevich-feynman, penner:book}.
Similarly, they appear in string theory and in perturbative expansions of matrix integrals
\cite{marino}.\\

The first goal of this paper is to propose a generalization of ribbon graphs which governs, in an
analogous way, the geometry and topology of {\em structured surfaces}. By these we mean
$C^\infty$-surfaces $S$, possibly with boundary, equipped with a nonempty set $M \subset S$ of
marked points together with a reduction of the structure group of the tangent bundle
$T_{S\setminus M}$ along a fixed Lie group homomorphism
\begin{equation}
  \label{eq:GG} { \p}: \GG \lra GL(2,\RR).
\end{equation}
We assume that $\mathbbm p$ is a {\em connective covering}: a not necessarily surjective unramified covering
such that the preimage of the component of identity $GL^+(2,\RR) \subset GL(2,\RR)$ is connected. 
This implies that a $\GG$-structure on $S$ is a discrete datum. A $GL^+(2,\RR)$-structure is an orientation; 
in the case when $\GG$ is an $N$-fold covering of $GL^+(2,\RR)$, a $\GG$-structure is known
as an $N$-{\em spin structure}, etc. Fixing a topological type of a $\GG$-structured marked surface
$(S,M)$, we then have the structured mapping class group $\Mod^\GG(S,M)$. 

For a more analytic point of view, consider the subgroup of conformal linear transformations
\[
\on{Conf}(2) \,\,=\,\, (\ZZ/2)\ltimes \CC^* \,\,\subset \,\, GL(2,\RR)
\]
which is homotopy equivalent to $GL(2,\RR)$, so that connective coverings of the two groups are in bijection.
Denoting by $G_\conf$ the preimage of $\on{Conf}(2)$ in $\GG$, we can consider surfaces with
$G_\conf$-structure which are essentially algebro-geometric objects: if $\GG$
preserves orientation, we obtain Riemann surfaces, otherwise Klein surfaces \cite{alling-greenleaf}.
Such objects have moduli spaces $\Mc^G$ (more precisely, stacks) which are algebro-geometric
counterparts of the groups $\Mod^\GG(S)$. 

These groups and moduli spaces are best known for the ``standard" case $\GG=GL^+(2,\RR)$ (oriented
surfaces, Riemann surfaces), see, e.g., \cite{farb-margalit}. Other cases are attracting
increasingly more interest in recent years. For example, in the unoriented case $\GG=GL(2,\RR)$,
the orbifold $\Mc^G$ is well known to be  the real locus of the moduli stack of algebraic curves
\cite{natanzon, seppala-sorvali, liu}. 
The unoriented mapping class groups, although classical \cite {mangler, singerman, birman-chillinworth},
have some of their important properties established only recently \cite {randal-williams:nonorient, wahl}. 
The situation is similar for $N$-spin mapping class groups \cite{randal-williams:spin}; the corresponding moduli
spaces $\Mc^G$ of $N$-spin Riemann surfaces \cite{jarvis} provide important examples of integrable hierarchies
and cohomological field theories \cite{jarvis-kimura-vaintrob}.\\ 

All this makes it desirable to have a flexible combinatorial formalism extending that of ribbon
graphs to the case of an arbitrary $\GG$ as above.  Our main observation is that the ingredients of
such a formalism can be found in the concept of a {\em crossed simplicial group} (due to
Fiedorowicz-Loday \cite{fiedorowicz-loday} and Krasauskas \cite{krasauskas}). A crossed simplicial
group is a certain category $\Delta\Gen$ with objects $[n]$, $n\geq 0$, containing the simplicial
category $\Delta$. It turns out that each connective covering $\GG$ as in \eqref{eq:GG} has its
associated crossed simplicial group $\Delta\Gen$. The prime example is the {\em cyclic category}
$\Lambda$ introduced by A. Connes \cite{connes} as the foundation of cyclic homology. For this
category $\Gen_n = \Aut([n]) = \ZZ/(n+1)$ is the cyclic group. This matches the data of a cyclic ordering on the
set of halfedges incident to a vertex of a ribbon graph. Thus $\Lambda$ can be said to ``govern"
the world of oriented surfaces. 

More generally, for each $\GG$ as above, with associated crossed simplicial group $\Delta\Gen$, we
introduce the concept of a $\Delta\Gen$-structured graph. We show that any embedding of a graph
$\Gamma$ into a $\GG$-structured surface induces a $\DG$-structure on $\Gamma$ (Proposition
\ref{prop:emb-graph-str}).
We further prove (Theorem \ref{thm:g-graph}) that the nerve of the category formed by
$\DG$-structured graphs and their contractions, is homotopy equivalent to the union of the
classifying spaces of the groups $\Mod^\GG(S,M)$ for all topological types of stable marked surfaces
$(S,M)$. 

For example, for unoriented surfaces, the cyclic category $\Lambda$ is replaced by the {\em dihedral
category} $\Xi$, see \cite{loday}. Applying our formalism to $\Xi$, we get a concept known as a {\em
M\"obius graph} \cite {braun, mulase-waldron, mulase-yu} but formulated in a somewhat more
conceptual way.\\

The same way ribbon graphs can be utilized to construct invariants of oriented surfaces,
$\DG$-structured graphs provide means to construct invariants of $\GG$-structured surfaces. While this
includes generalizations of $2$-dimensional oriented topological field theories constructed from
Frobenius algebras, we are mainly interested in a ``categorified'' variant of this construction:
A functor $X:\DG^{\op} \to \C$ can be seen as a simplicial object $X$ in $\C$ together with extra structure
given by an action of $\Gen_n = \Aut([n])$ on $X_n$ for every $n$. 
We can evaluate such $X$ on any $\DG$-structured graph $\Gamma$ to obtain an object $X(\Gamma)$ in $\C$. Assuming that $X$
satisfies a certain combinatorial descent condition ($2$-Segal condition) introduced in
\cite{HSS1}, we can think of $X(\Gamma)$ as the global sections of a combinatorial sheaf on the
surface $(S,M)$ modelled by $\Gamma$, so that $X(\Gamma) \cong X(S,M)$ is independent on the chosen
graph $\Gamma$. Further, if $\C$ carries a model structure, then we have a derived
variant of this construction which generalizes the invariants of \cite{HSS-triangulated} obtained
from $2$-Segal cyclic objects.

For example, a $2$-Segal dihedral object $X$ associates to every stable
marked unoriented surface $(S,M)$, an object $X(S,M)$ with a coherent 
action of the unoriented mapping class group of $(S,M)$.

A cyclic 2-Segal object $X$ in $\C$ can be seen as a nonlinear, categorical analog of a
Frobenius algebra $A$. If $\C=\Set$, then 1-simplices, i.e., elements of $X_1$ are
analogous to elements of $A$, and the number of 2-simplices $\sigma\in X_2$ with three given
boundary 1-simplices $a,b,c$ corresponds  to the cyclically invariant scalar product $(ab,c)$. The
construction of \cite{HSS-triangulated} is thus a ``categorification" of the celebrated fact that
Frobenius algebras (and, more general, Calabi-Yau dg-categories) give rise to invariants of oriented
marked surfaces.  In other words, it fits into  Table \ref{table:invariants} summarizing various
types of ``invariants" and their meaning. 

\begin{table}
  \renewcommand{\arraystretch}{2}
  \begin{tabular}{ p{5cm} | p{5cm}| p{5cm}  }
  Source of invariants & Type of invariants & Precise meaning
  \\
  \hline\hline
  Usual Frobenius algebras & Numerical invariants & Elements of $H^0$ of moduli spaces
    \\ 
    \hline
    Calabi-Yau algebras and categories & Cohomological invariants &  Higher cohomology classes on
    moduli
    spaces  \cite{kontsevich-feynman}\cite{costello-TFT}
 \\ 
    \hline 
    Modular categories & Vector space-valued invariants (fusion data)
    & Local systems of vector spaces on moduli spaces
    \cite{bakalov-kirillov}
      \\ 
    \hline
    Cyclic 2-Segal objects & Categorical invariants & Local systems of objects of $\C$ on moduli spaces
    \\
    \hline
    \end{tabular}
       \caption{ 
       Various ``invariants" of oriented surfaces.}\label{table:invariants}
\end{table}

It is natural therefore to expect that our approach can be developed to include  ``structured"
analogs of all the other rows in Table \ref{table:invariants}. In this paper we discuss only the
structured analog of the concept of a Frobenius algebra, leaving other contexts for future work.
This analog is based on the following concept.

Let $H$ be a group equipped with a {\em parity}, by which we mean a homomorphism
\[
\rho: H \lra \ZZ/2.
\]
Let us write $H_{0}$ and $H_{1}$ for the preimages of $0$ and $1$ whose elements are called even and
odd, respectively. One can then introduce the concept of a {\em twisted action} of $H$ on an
associative algebra so that even elements of $H$ act by automoprhisms while odd elements of $H$ act
by anti-automorphisms. 
One can similarly speak about twisted actions of $H$ on a category: even group elements of 
act by covariant functors while odd elements act by contravariant functors. 
In the situation of a group $\GG$ and a crossed simplicial group $\Delta\Gen$ as above, the group
$\Gen_0$ comes with a natural parity. It turns out
that considering algebras and categories with twisted $\Gen_0$-action allows us to extend many
known classical results to the $\GG$-structured situation. In particular:

\begin{enumerate}
\item[(1)] For a category $\C$ with a twisted $\Gen_0$-action we have a natural ``nerve" $\N^\GG(\C)$ which
is a $\Delta\Gen$-set. 

\item[(2)] For an algebra $A$ with a twisted $\Gen_0$-action, the Hochschild complex $C_\bullet^{\on{Hoch}}(A)$
has a natural structure of a $\Delta\Gen$-vector space.

\end{enumerate}

We expect that a Frobenius algebra (resp. a  Calabi-Yau category) with a twisted $\Gen_0$-action
gives rise to a numerical (resp. cohomological, in the sense of Table \ref{table:invariants})
invariant of $\GG$-structured marked surfaces. The results of \cite{alex-natanzon, novak-runkel} for
$\DG=\Xi, \Lambda_2$ (unoriented surfaces, 2-spin surfaces) support this expectation.\\
 

\noindent
{\bf Acknowledgements.} We thank David Ayala, Jeffrey Giansiracusa, Ralph Kaufmann, Maxim
Kontsevich, Jacob Lurie, Martin Markl, Sergei Natanzon, Sebastian Novak, Ingo Runkel, and Bertrand To\"en for inspiring
discussions. This work was carried out while T.D. was a Titchmarsh Fellow at the University
of Oxford.

\newpage

\section{Crossed simplicial groups and planar Lie groups}

\numberwithin{equation}{subsection}

\subsection{Basic definitions}

At the basis of combinatorial topology lies the {\em simplex category} $\Delta$ whose objects are
the finite ordinals $[n]=\{0,1,\dots, n\}$, $n \ge 0$, with morphisms given by monotone maps.
{\em Simplicial objects} in a category $\Cb$ are functors $\Delta^\op\to\Cb$. In his axiomatization
of cyclic homology, A. Connes introduced a category $\Lambda$ which can be thought of as a hybrid of
$\Delta$ and the family $\{ \ZZ /(n+1) \}$ of finite cyclic groups which appear as the automorphism
groups of the objects of $\Lambda$, see \cite{connes, loday}. The formal interplay between the two
is captured in the notion of a crossed simplicial group introduced in \cite{fiedorowicz-loday,
krasauskas}.

\begin{defi} \label{defi:csg} A {\em crossed simplicial group} is a category $\Delta \Gen$ equipped with an embedding
	$i:\Delta \to \Delta \Gen$ such that:
	\begin{enumerate}
		\item the functor $i$ is bijective on objects,
		\item any morphism $u: i[m] \to i[n]$ in $\Delta \Gen$ can be uniquely expressed as 
			a composition $i(\phi) \circ g$ where $\phi: [m] \to [n]$ is a morphism in
			$\Delta$ and $g$ is an automorphism of $i[m]$ in $\Delta \Gen$.
	\end{enumerate}
\end{defi}

We will refer to the representation $u=i(\phi) \circ g$ in (2) as the {\em canonical factorization}
of $u$. To keep the notation light, we will usually leave the embedding $i$ implicit, 
referring to the objects of $\Delta \Gen$ as $[n]$, $n\geq 0$.
To every crossed simplicial group $\Delta \Gen$, we can associate a sequence of groups
\begin{equation}\label{eq:crossed}
	\Gen_n = \Aut_{\Delta\Gen}([n]). 
\end{equation}
Further, by Property (2), any diagram
\[
	\xymatrix{ & [n]\ar[d]^{g} \\ [m] \ar[r]_{\phi} & [n]}
\]
where $\phi$ in $\Delta$ and $g \in \Gen_n$, can be uniquely completed to a commutative diagram in
$\DG$
\begin{equation}
	\label{eq:can-square}
	\xymatrix{ [m]\ar[d]_{\phi^*g} \ar[r]^{g^*\phi} & [n]\ar[d]^{g} \\ [m] \ar[r]_{\phi} & [n]}
\end{equation}
with $g^*\phi$ in $\Delta$ and $\phi^* g \in \Gen_n$. These data satisfy the following
compatibilities.

\begin{prop}\label{prop:skeleton}
	\begin{enumerate}[label=(\alph*)]
		\item For every morphism $\phi: [m] \to [n]$ in $\Delta$, the association
		$g\mapsto \phi^* g$ defines a
		 map of sets
			\[
				\phi^*: \Gen_n \lra \Gen_m,\; g \mapsto \phi^*g
			\]
		The association $\phi \mapsto \phi^*$ makes the family $\Gen=(\Gen_n)_{n\geq 0}$ 
		into a simplicial set. The maps $\phi^*$ preserve unit elements but not necessarily the group structure.

		\item For objects $[m],[n]$, the association
			\[
				(\phi,g) \mapsto g^*\phi
			\]
			determines a right action of $\Gen_n$ on the set $\Hom_{\Delta}([m],[n])$. 
			In the case $m=n$, this action preserves the identity morphism.

		\item In addition, we have the identities
			\begin{align*}
			\phi^*(g \circ h) & = \phi^*g \circ (g^*\phi)^*h,\\
			g^*(\phi \circ \psi) & = g^*\phi \circ (\phi^*g)^*\psi.
			\end{align*}
			
			\item Conversely, any sequence of groups $\Gen=(\Gen_n)$
	with operations
	\[
	(\phi,g) \mapsto (\phi^*(g),g^*(\phi))
	\]
	satisfying the compatibilities  (a) - (c),  uniquely determines a
	crossed simplicial group.

	\end{enumerate}
\end{prop}	
\begin{proof}\cite{fiedorowicz-loday}, Proposition 1.6.\end{proof}

	\begin{exa}\label{exa:deviation}
Note that a simplicial group is a particular example of a crossed simplicial group, corresponding to
trivial actions of $\Gen_n$ on $\Hom_\Delta([m], [n])$. Therefore, in general, these actions and the identities
in Proposition \ref{prop:skeleton}(c), describe the ``deviation" of  $\Gen$  from being a simplicial
group. 
\end{exa}

\begin{exa}\label{ex:extreme} 
Let us point out the following special cases of canonical factorization in a crossed simplicial
group $\DG$:
\[
\Hom_\DG([n], [0]) \cong \Gen_n, \quad \Hom_\DG([0], [n]) \cong \{0,1,\dots, n\} \times \Gen_0. 
\]
The first identification follows from the fact that $\Hom_\Delta([n], [0]) = \pt$. Note that the
simplicial set structure on $\{\Gen_n\}$ can be deduced immediately from this identification:
$\Hom_\DG(-, [0])$ is a contravariant functor on $\DG$ and, by restriction, on $\Delta$. 
The second identification follows from the canonical identification of $\Hom_\Delta([0], [n])$ with
$\{0,1,\cdots, n\}$ given by evaluation.
\end{exa}

The following proposition implies that any crossed simplicial group $\DG$ has a natural forgetful
functor into the category of sets, so that the objects of $\DG$ can be interpreted as sets equipped with extra
structure. This point of view will be elaborated in Chapter \ref{sec:order}. 

\begin{prop}\label{prop:forget}
Let $\DG$ be a crossed simplicial group. Then we have a functor
\[
	\lambda: \DG \lra \Set, \; [n] \mapsto \Hom_{\DG}([0], [n])/\Gen_0
\]
where $\lambda([n])$ can be canonically identified with the set $\Hom_{\Delta}([0],[n]) \cong \{0,1,\dots,n\}$.
In particular, we obtain, for every object $[n]$ of $\DG$, a canonical group homomorphism
\[
	\lambda_{n}: \Gen_n \lra S_{n+1}.
\]
\end{prop}
\begin{proof}
	 Follows from the second identification in Example \ref{ex:extreme}. 
\end{proof}

\subsection{The Weyl crossed simplicial group}
\label{sub:weyl}

Let $J$ be a finite set. A {\em signed linear order} on $J$ consists of
\begin{enumerate}
	\item a linear order on $J$,
	\item a map of sets $\varepsilon: J \to \Zt$.
\end{enumerate}

We introduce a category $\DW$ with objects given by the sets $\{0,1,\dots,n\}$, $n \ge 0$. A
morphism $f: I \to J$ is given by a map of underlying sets together with the choice of a signed
linear order on each fiber $f^{-1}(j)$. Composition of morphisms is obtained by forming {\em
lexicographic signed linear orders}: Given $f: I \to J$ and $g: J \to K$, we have, for $k \in K$, a
linear order on $(gf)^{-1}(k)$ obtained by declaring $i_1 \le i_2$ if 
\begin{enumerate}
	\item either $f(i_1) = f(i_2)$ and $i_1 \le i_2$ with respect to the linear order on the fiber of $f$,
	\item or $f(i_1) \ne f(i_2)$ and $f(i_1) \le f(i_2)$ with respect to the linear order on the
		fiber of $g$.
\end{enumerate}
The sign of $i \in (gf)^{-1}(k)$ is obtained by setting $\varepsilon_{gf}(i) = \varepsilon_f(i) +
\varepsilon_g(f(i))$.

\begin{prop} The category $\DW$ is a crossed simplicial group with 
	\[
		\Gen_n = \Aut_{\DW}([n]) \cong W_{n+1}
	\]
	where $W_{n+1}$ denotes the signed permutation group of $\{0,1,\dots,n\}$ also known as the
	wreath product $\Zt\wreath S_{n+1}$.
\end{prop}
\begin{proof} We have to verify the unique factorization property which is a direct consequence of the
	definition of $\DW$.
\end{proof}

Following Krasauskas, we call $\DW$ the {\em Weyl crossed simplicial group}, since the group $W_{n}$
is the Weyl group of the root system $B_n$ (or $C_n$). Its fundamental importance stems
from the following result (\cite{krasauskas,fiedorowicz-loday}).

\begin{thm}\label{thm:classification} Let $\DG$ be a crossed simplicial group.
\begin{enumerate}
	\item There is a canonical functor $\pi: \DG \to \DW$.
	\item For every $n \ge 0$, we have an induced short exact sequence of groups
		\[
			1 \lra \Gen_n' \lra \Gen_n \lra \Gen_n'' \lra 1
		\]
	      where $\Gen_n'$ and $\Gen_n''$ denote kernel and image, respectively, of the induced
	      homomorphism $\pi_n: \Gen_n \to W_{n+1}$. The short exact sequences assemble to a
	      sequence of functors
	      \[
			\DG' \lra \DG \lra \DG'' 
	      \]
	      where $\DG'$ is a simplicial group (Example \ref{exa:deviation}) and $\DG'' \subset \DW$ is a crossed
	      simplicial subgroup of $\DW$.
\end{enumerate}
\end{thm}
\begin{proof} 
	Consider the functor
		\[
			\lambda: \DG \lra \Set, \; [n] \mapsto \Hom_{\DG}([0], [n])/\Gen_0
		\]
	from Proposition \ref{prop:forget}. 
	We claim that $\lambda$ admits a canonical factorization over the forgetful functor $\DW \to \Set$. 
	Note that, via $\lambda$, the group $\Gen_n$ acts canonically on the set
	$\Hom_{\Delta}([0],[n]) = \{0,1,\dots,n\}$ of vertices of the combinatorial $n$-simplex. 
	A vertex $i \in \{0,1,\dots,n\}$ can be canonically identified with the corresponding degeneracy map
	\[
		s_i: [n+1] \lra [n],\; j \mapsto \begin{cases} j & \text{for $j \le i$,}\\
			j-1 & \text{for $j > i$.} \end{cases}
	\] 
	Explicitly, given $g \in \Gen_n$ and $i \in \{0,1,\dots,n\}$, we have a commutative square
	\[
		\xymatrix{
			[n+1] \ar[d]_{s_j^*g} \ar[r]^{s_i} & [n] \ar[d]^g\\
			[n+1] \ar[r]^{s_j} & [n]
		}
	\]
	where $j = \lambda(g)(i)$. Clearly, the map $\lambda(s_j^*g)$ induces a map from the fiber
	$s_i^{-1}(i) = \{i,i+1\}$ to the fiber $s_j^{-1}(j) = \{j,j+1\}$ which is either order
	preserving or order reversing. Therefore, we can lift $\lambda(g) \in S_{n+1}$ to a signed permutation 
	\[
		\widetilde{\lambda}(g) = (\varepsilon_0, \varepsilon_1, \dots, \varepsilon_n ;
		\lambda(g)) \in \Zt \wreath S_{n+1}
	\]
	where 
	\[
		\varepsilon_i = \begin{cases} 0 & \text{if $\lambda(s_j^*g): \{i,i+1\} \to \{j,j+1\}$ is order preserving,}\\
						1 & \text{if $\lambda(s_j^*g): \{i,i+1\} \to \{j,j+1\}$ is order reversing.}
				\end{cases}
	\]
	Using the unique factorization property in $\DG$, it is straightforward to verify 
	that the association $g \mapsto \widetilde{\lambda}(g)$ extends to provide a functor
	\[
		\widetilde{\lambda}: \DG \lra \DW
	\]
	commuting with the forgetful functors to $\Set$, which proves (1). The remaining statements
	are easy to verify, the only nonobvious point being that $\DG'$ is in fact noncrossed
	simplicial. But this follows from the construction of $\widetilde{\lambda}$: For $g \in
	\Gen_n$, the condition $\lambda(g) = \id \in S_{n+1}$ implies that $g^*$ fixes all face maps
	$[n-1] \to [n]$. The condition that all signs $\varepsilon_i$ of the lift
	$\widetilde{\lambda(g)}$ are $0$ implies that $g^*$ fixes all degeneracy maps $[n+1] \to
	[n]$. But this implies that $\DG''$ is a simplicial group (cf. Example \ref{exa:deviation}).
\end{proof}

According to Theorem \ref{thm:classification}, the classification of crossed simplicial groups up
to extensions by simplicial groups therefore reduces to the classification of crossed simplicial
subgroups of the Weyl crossed simplicial group $\DW$. There are precisely $7$ such subgroups called
{\em fundamental crossed simplicial groups}. We provide a list in Table \ref{table:types}, following the terminology
introduced in \cite{krasauskas}.
Therefore, every crossed simplicial group has a {\em type} given by its image in $\DW$. For example
a crossed simplicial group of {\em trivial} type is a simplicial group. Note that the
$7$ subgroups of $\DW$ can be further distinguished according to their {\em growth rate}:
\begin{enumerate}
	\item {\em Constant:} For the trivial and reflexive groups, the cardinality of $\Gen_n$ is constant.

	\item {\em Tame:} For the cyclic and dihedral groups, the size of $\Gen_n$ grows 
 		linearly with $n$.

	\item {\em Wild:} For the remaining subgroups, the size of $\Gen_n$ grows exponentially.  
\end{enumerate}

\begin{table}[h]
\begin{center}
\begin{tabular}{l|l}
	Type & Subgroup of $\DW$\\
	\hline
	Trivial & $ \{1\}$\\
	Reflexive & $ \{\Zt\}$\\
	Cyclic & $ \{\ZZ/(n+1)\ZZ \}$\\
	Dihedral & $ \{D_{n+1}\}$\\
	Symmetric & $ \{S_{n+1}\}$\\
	Reflexosymmetric & $ \{\Zt \ltimes S_{n+1}\}$\\
	Weyl & $ \{\Zt \wreath S_{n+1} \}$
\end{tabular}
\end{center}
\caption{The $7$ types of crossed simplicial groups.}
\label{table:types}
\end{table}

\subsection{Semiconstant crossed simplicial groups and twisted group actions}
\label{section:semiconstant}

In this section, we will study a particular class of crossed simplicial groups which are of reflexive type in
the sense of Table \ref{table:types}. 

Let $\omega_n: [n]\to [0]$ be the unique morphism from $[n]$ to $[0]$ in $\Delta$.  
Let $\DG$ be a crossed simplicial group.  The first equality in
Proposition \ref{prop:skeleton}(c)
  together with
the uniqueness of $\omega_n$ implies that the pullback map
\[
	\omega_n^*: \Gen_0 \lra \Gen_n
\]
is a group homomorphism. 

\begin{defi}
 A crossed simplicial group $\DG$  is called {\em
semiconstant} if, for every $n \ge 0$, the homomorphism 
$
	\omega_n^*: \Gen_0 \lra \Gen_n,
$
  is an
isomorphism. 

\end{defi}

Note that this condition implies that, for every map $\phi: [m] \to [n]$ in $\Delta$,
the corresponging map $\phi^*: \Gen_n \to \Gen_m$ is a group isomorphism. We may use $\omega_n^*$
to identify $\Gen_n$ with $\Gen_0$. Via this identification we have, for every $g \in \Gen_0$ and
every morphism $\phi: [m] \to [n]$ in $\Delta$, a commutative square
\begin{equation}\label{eq:semiconstant}
	\xymatrix{ [m] \ar[d]_{\phi^*g = g}\ar[r]^{g^*\phi} & [n]\ar[d]^g\\
		[m] \ar[r]^{\phi} & [n]}
\end{equation}
in $\DG$. Therefore, the simplicial set $\Gen_{\bullet}$ corresponding to a semiconstant crossed
simplicial group is a constant simplicial group. However, the action $\phi \mapsto g^*\phi$ may 
be nontrivial which is an additional datum.

\begin{exa}\label{ex:semiconst-emb}
Any crossed simplicial group $\DG$ contains the semiconstant crossed simplicial group
	$\Delta \{\Gen_0\}$ generated by $\Gen_0$: we restrict ourselves to those automorphisms of $[n]$
	which are pullbacks of automorphisms of $[0]$ along  $\omega_n$. 

\end{exa}

The goal of this section is to relate semiconstant crossed simplicial groups to twisted group
actions. We will start by introducing some terminology. A (strict) action of a group $G$ on a small
category $\C$, is defined to be a homomorphism $G \to \Aut_{\Cat}(\C)$ where $\Cat$ denotes the
category of small categories.

\begin{defi}\label{defi:semidirect} Let $\C$ be a category equipped with an action of a group $G$. We define a category $G
	\ltimes \C$ called the {\em semidirect product} of $G$ and $\C$. The objects of $G \ltimes
	\C$ are the objects of $\C$, a morphism from $x$ to $y$ is given by a pair $(g, \phi)$ where
	$g \in G$ and $\phi: g.x \to y$ is a morphism in $\C$. The composition of morphisms $(g,
	\phi): x \to y$ and $(g', \phi'): y \to z$ is the morphism $(g'g, \phi' \circ g'.\phi)$
\end{defi}

\begin{rem} The action of $G$ on $\C$ can be interpreted as a functor $\B G \to \Cat$ where $\B G$
	denotes the groupoid with one object and automorphism group $G$. In this context, the semidirect 
	product $G \ltimes \C$ equipped with its natural functor to $\B G$ is known as the Grothendieck 
	construction. 
\end{rem}

\begin{prop}\label{prop:semi} Let $G$ be a group acting on $\Delta$. Then the corresponding semidirect product $G
	\ltimes \Delta$ is a semiconstant crossed simplicial group. Vice versa, any semiconstant
	crossed simplicial group is isomorphic to a semidirect product $\Gen_0 \ltimes \Delta$.
\end{prop}
\begin{proof} 
	The semidirect product $G \ltimes \Delta$ has the unique factorization property:
	$(g, \phi) = (1, \phi) \circ (g, \id)$ making it a semiconstant crossed simplicial group.
	Given a semiconstant crossed simplicial group $\DG$, it follows from
	\eqref{eq:semiconstant} that $\Gen_0$ acts on $\Delta$. It immediately follows from the
	definition that we may identify $\DG$ with $\Gen_0 \ltimes \Delta$.
\end{proof}

There is an involution $\daleth$ on $\Delta$ which is the identity on objects and maps a morphism 
$\phi:[m] \to [n]$ to the {\em opposite morphism} $\phi^{\op}:[m] \to [n]$ defined by 
$\phi(i) = n+1 - \phi(m+1-i)$. If we replace $\Delta$ by the equivalent larger category
$\mathbf\Delta$ of all nonempty finite ordinals, then $\daleth$ can be defined more naturally as
the functor ${\mathbf\Delta} \to{\mathbf\Delta}$ sending each ordinal $(I, \leq)$ to the opposite
ordinal $(I, \geq)$. 

\begin{prop} We have $\Aut_{\Cat}(\Delta) \cong \Zt$ with generator given by $\daleth$. In
	particular, any action of a group $G$ on $\Delta$ factors via a homomorphism $G \to \Zt$ 
	over the action of $\Zt$ on $\Delta$ given by $\daleth$.
\end{prop}
\begin{proof} Well known.
\end{proof}

\begin{rem} As a corollary, we obtain that any semidirect product $G \ltimes \Delta$ admits a
	canonical functor to $\Zt \ltimes \Delta$. Interpreting this in the general context of
	crossed simplicial groups says that for a semiconstant crossed simplicial group $\DG$, the
	canonical functor into the Weyl crossed simplicial group of Theorem \ref{thm:classification} 
	factors through the semiconstant crossed simplicial group $\Delta \{\Zt\}$.
	Thus, in terms of the classification of crossed simplicial groups, semiconstant crossed
	simplicial groups are extensions of $\Delta \{\Zt\}$ by constant simplicial groups.
\end{rem}

Let $\C$ be a category and let $\BZt$ denote the groupoid with one object and automorphism group
$\Zt$. We define a {\em parity} on $\C$ to be a functor $\C \to \BZt$. Explicitly,
we are given a partition $\Hom(x,y) = \Hom(x,y)_0 \amalg \Hom(x,y)_1$ of every morphism set into
{\em even} and {\em odd} morphisms such that the composite of morphisms of the same parity is even
while the composite of morphisms of opposite parity is odd. Categories with parity naturally form a
category where morphisms are given by parity preserving functors, i.e., commutative diagrams of
\[
	\xymatrix@C=1ex{ \C \ar[rr]\ar[dr] & & \D\ar[dl]\\
	& \BZt. & }
\]

\begin{exas} 
	\begin{enumerate}
		\item A {\em parity} on a group $G$ is defined to be a homomorphism $G \to \Zt$. We
			obtain a category with parity by passing to groupoids $\B G \to
			\BZt$.
		\item Let $\D$ be a category equipped with an involutive functor $\tau: \D \to \D$.
			We may interpret $\tau$ as an action of $\Zt$ on $\D$. The corresponding 
			semidirect product $\Zt \ltimes \D \to \BZt$ is a category with parity.
	\end{enumerate}
\end{exas}

\begin{exas} \label{exas:semi} We give some examples of categories with parities arising as semidirect products with
	$\Zt$.
	\begin{enumerate}
		\item Consider the category $\Cat$ of small categories. The involution $\tau: \C
			\mapsto \C^{\op}$ gives rise to a semidirect product $\Zt \ltimes \Cat$.
			We can consider enriched variants of this construction. For example, given a
			field $\k$, we obtain a semidirect product $\Zt \ltimes \Cat_\k$ where $\Cat_\k$ denotes the
			category of small $\k$-linear categories (categories enriched over $\Vect_\k$).
		\item As we have seen above, the simplex category $\Delta$ has a natural involution
			$\daleth$ and we obtain a corresponding semidirect product $\Zt \ltimes \Delta$.
		\item Let $\C$ be any category. The involution $\daleth$ of $\Delta$ induces an
			involution on the category $\C_{\Delta}$ of simplicial objects in $\C$. We
			obtain a semidirect product $\Zt \ltimes \C_{\Delta}$.
	\end{enumerate}
\end{exas}

We have an adjunction 
\[
		\FC: \Set_{\Delta} \longleftrightarrow \Cat: \N
\]
where $\FC$ associates to a simplicial set $K$ the free category $\FC(K)$ generated by $K$ and $\N$ takes a
small category $\C$ to its nerve $\N(\C)$. Note that this adjunction is compatible with the
involutions $\daleth$ and $\tau$ so that we have the following consequence.

\begin{prop} We have an adjunction
\[
	\xymatrix@C=2ex{
		\FC: \Zt \ltimes \Set_{\Delta} \ar[dr]\ar@{<->}[rr]&  &\ar[dl] \Zt \ltimes \Cat: \N\\
		 & \B \Zt &  
	}
\]
of categories with parity.
\end{prop}

Given a group $G$ and a category $\D$, both equipped with parity, we define a {\em twisted 
action} of $G$ on an object $x$ of $\D$ to be a parity preserving homomorphism $G \to
\Aut_{\D}(x,x)$ or, in other words, a functor
\[
	\xymatrix@C=1ex{ \B G \ar[rr]\ar[dr] & & \D\ar[dl]\\
	& \BZt & }
\]
over $\BZt$.

\begin{exas} \begin{enumerate} 
	\item From Example \ref{exas:semi}(1), we obtain the concept of a twisted group
		actions on a $\k$-linear category. Explicitly, this means that even group elements act as 
		covariant functors while odd group elements act via contravariant functors. As a special case, given
		by a $\k$-linear category with one object, we obtain the concept of a twisted group action on an
		associative $\k$-algebra. 
	\item Example \ref{exas:semi}(3) provides us with the concept of a twisted group action on a
		simplicial set (or more generally simplicial object). Here, even group elements act
		in an the usual orientation preserving way: all face and degeneracy relations are
		respected. Odd group elements act in an orientation reversing way so that all face
		and degeneracy relations are reversed according to the involution $\daleth: \Delta \to
		\Delta$.
	\end{enumerate}
\end{exas}

\begin{prop}\label{prop:main} Let $\DG$ be a semiconstant crossed simplicial group. Then, for any category $\C$, 
	we have a natural equivalence of categories
	\[
		\C_{\DG} \overset{\cong}{\lra} \Gen_0-\C_{\Delta}, \; K \mapsto K_{| \Delta} 
	\]
	where 
	\[
		\Gen_0-\C_{\Delta} = \Fun_{\BZt}(\B \Gen_0, \Zt \ltimes \C_{\Delta})
	\]
	denotes the category of simplicial objects in $\C$ equipped with a twisted
	group action of $\Gen_0 \to \Aut(\Delta) \cong \Zt$.
\end{prop}
\begin{proof} This follows immediately from unravelling the definitions.
\end{proof}

\newpage

\subsection{Planar crossed simplicial groups}\label{sub:cross-examples}
 
In this section, we introduce a certain class of crossed simplicial groups called {\em planar
crossed simplicial groups}. In terms of the type classification of Table \ref{table:types}, planar
crossed simplicial groups can be of cyclic or dihedral type. They correspond, very precisely, to Lie
groups that appear as structure groups of surfaces. We start by listing them all in detail. 

\begin{exa}\label{ex:cyclic}
 	The {\bf cyclic category} $\Lambda$ has objects $\cn$, $n \ge 0$, while the set of
	morphisms from $\cm$ to $\cn$ can be described as follows. Let $C_n$ denote the topological
	space given by the unit circle in $\CC$ equipped with the subset of marked points $\{0, 1,
	\dots, n\}$, embedded via the map $k \mapsto \exp(2\pi i k /(n+1))$. A morphism $f: \cn \to
	\cm$ in $\Lambda$ is a homotopy class of monotone maps $C_m \to C_n$ of degree $1$ such
	that $f(\{0,1,\dots,m\}) \subset \{0,1,\dots,n\}$. The category $\Delta$ is contained in
	$\Lambda$ given by restricting to those morphisms $f: \cn \to \cm$ such that any homotopy
	inverse of $f$, relative to the marked points, maps the oriented arc between $m$ and $0$ on
	$C_m$ into the arc between $n$ and $0$ on $C_n$.
	The family of groups $\Gen$ associated to $\Lambda$ via \eqref{eq:crossed} is the family of cyclic
	groups $\{\Gen_n=\ZZ/(n+1)\}$.
\end{exa}

\begin{exa}\label{ex:dihedral}
 The {\bf dihedral category} $\Xi$ has objects $\dn$, $n \ge 0$. 
	Morphisms from $\dm$ to $\dn$ are homotopy classes of monotone 
	maps $C_m \to C_n$ of degree  $\pm 1$  such that $f(\{0,1,\dots,m\}) \subset
	\{0,1,\dots,n\}$. The category $\Xi$ naturally contains $\Lambda$ and hence also $\Delta$.
	The family of groups $\Gen$ associated to $\Xi$ via \eqref{eq:crossed} is the family of dihedral 
	groups
	  \[
	 \Gen_n = D_{n+1} = \langle \omega, \tau |\, \omega ^2=1, \tau^{n+1}=1, 
	 \omega\tau\omega^{-1} = \tau^{-1}\rangle 
	 \]	
	 with $|D_{n+1}| = 2(n+1)$.
\end{exa}
	  
\begin{exa} The {\bf paracyclic category} $\Lambda_{\infty}$ has objects $\pn$, $n \ge 0$. 
	Morphisms from $\ppm$ to $\pn$  are maps $f: \ZZ \to \ZZ$ which
	preserve the standard linear order and satisfy, for every $l \in \ZZ$, the condition $f(l +
	m+1) = f(l) + n + 1$. The category $\Delta$ can be found in $\Lambda_{\infty}$ by
	considering only morphisms $f:\ppm \to \pn$ such that 
	\[
	 f(\{0,1,\dots, m\}) \subset 
	\{0,1,\dots,n\}.
	\] 
	Construction \eqref{eq:crossed} yields the constant family
	of infinite cyclic groups $\{\Gen_n=\ZZ\}$. While this category was introduced in
	 \cite{fiedorowicz-loday}, we borrow
	the terminology {\em paracyclic} from \cite{getzler-jones}.
\end{exa}

\begin{exa}\label{ex:para-dihedral} 
The {\bf paradihedral category} $\Xi_\infty$ has objects $\widehat n$, $n\geq 0$.
Morphisms from $\widehat m$ to $\widehat n$  are maps $f: \ZZ \to \ZZ$ which either preserve or reverse
	 the standard linear order and satisfy, for every $l \in \ZZ$, the condition 
	 \[
	 f(l +m+1) = \begin{cases} 
	 f(l) + n + 1, \text{ if $f$ preserves the order}, \\
	 f(l)-n-1,\text { if $f$ reverses the order}.
	 \end{cases}
	\]
	Thus $\Xi_\infty\supset \Lambda_\infty\supset\Delta$. 
	Construction \eqref{eq:crossed} yields the constant family
	of infinite dihedral groups 
	  \[
	 \Gen_n = D_{\infty} = \langle \omega, \tau |\, \omega ^2=1, \omega\tau\omega^{-1} = \tau^{-1}\rangle. 
	 \]	
\end{exa}

For $N\geq 1$, we define the functor
\[
\on{sd}_N: \Delta\lra\Delta, \quad [n-1] \mapsto [Nn-1]
\]
which takes a monotone map $\phi:[m] \to [n]$ to the $N$-fold concatenation of $\phi$ with itself.
Given a simplicial set $X$, the 
{\em $N$-fold (edgewise) subdivision} of $X$ is
the simplicial set $\on{sd}^*_N(X)$
obtained by precomposing $X:\Delta^\op\to\Set$ with $\on{sd}_N$.
Geometrically, $\on{sd}^*_N(X)$ is obtained from $X$ by subdividing each simplex into several simplices
such that
\begin{enumerate}
\item Each edge is subdivided into $N$ intervals by introducing $N-1$ intermediate vertices. 

\item No other new vertices  inside simplices of dimension $\geq 2$ are introduced.
\end{enumerate}
In particular, the geometric realizations of $X$ and $\on{sd}_N^*X$ are canonically homeomorphic.
See  \cite[\S 1]{bokstedt-et-al} \cite[E.6.4.3]{loday} for more details.
 
\begin{exa} 
 \label{ex:subdivided} The {\bf $N$-subdivided categories $\Lambda_N, \Xi_N$.}
 Given a crossed simplicial group $\Delta\Gen$ with the associated simplicial set $\Gen$,
 one can form the simplicial set $\on{sd}_N^*\Gen$ and ask whether it can be
 completed to a new crossed simplicial group $\Delta\on{sd}_N^*\Gen$. 
 It was observed  in \cite[Ex.7]{fiedorowicz-loday} that this is indeed so
 for $\Delta\Gen$ equal to $\Lambda$ and $\Xi$.
 The corresponding crossed simplicial groups will be denoted $\Lambda_N, \Xi_N$. 
 They can be defined explicitly as follows.
 
 (a) The {\bf $N$-cyclic category} $\Lambda_{N}$ has objects $\kcn$, $n \ge 0$. A
	morphism from $\kcm$ to $\kcn$ is  an equivalence class of maps 
	$f: \ZZ \to \ZZ$ which preserve the standard linear order and satisfy, for every $l \in \ZZ$, the condition $f(l +
	m+1) = f(l) + n + 1$. A pair of maps $f,g$ are considered equivalent if there exists an integer $r$ such that
	$f - g = rN(n+1)$. The category $\Delta$ can be found in $\Lambda_{N}$ by
	considering only those morphisms $f:\kcm \to \kcn$ such that 
	\[
	f(\{0,1,\dots, m\}) \subset \{0,1,\dots,n\}.
	\]
	Thus $\Gen_n= \{\ZZ/N(n+1)\ZZ\}$.  
	 
	A description of $\Lambda_N$ which is more in line with Example \ref{ex:cyclic} can be obtained
	as follows. Fix an $N$-fold cover $\widetilde{C} \to C$ of the unit circle in $\CC$. Then a
	morphism in $\Lambda_N$ between $\kcm$ and $\kcn$ can be described as a homotopy class of
	monotone maps $C_m$ to $C_n$ of degree $1$, preserving the marked points, together with a lift to $\widetilde{C}$.
	
	(b) The {\bf $N$-dihedral category} $\Xi_N$ is obtained by modifying the definition of $\Lambda_N$, 
	 allowing $f$ to either preserve or
	reverse the standard linear order and imposing the condition on $f(l+m+1)$ as in Example 
	\ref{ex:para-dihedral}. We have a topological description analogous to the one for the
	$N$-cyclic category.
\end{exa}

\begin{exa}\label{ex:quaternionic}
The {\bf quaternionic category} $\nabla$ has objects $\check n$, $n\geq 0$.
To describe morphisms, we first introduce the {\em quaternionic groups}
\[
   Q_n  \,\,=\,\,  \langle w, \tau |\, w^2=\tau^n , \tau^{2n}=w^4=1, w\tau w^{-1} = \tau^{-1}\rangle,
   \quad |Q_n|=4n.
\]
    Here $n\geq 1$ is an integer.  Thus $Q_1=\ZZ/4$. Let $\HH$  be the skew field of quaternions.
    For any $r$, we denote by   $\mu_r \subset \CC^* \subset \HH^*$ the group of $r$th roots of 1.
    Then $Q_n$, $n\geq 2$, can be identified with the subgroup in $\HH^*$ generated by $\mu_{2n}$
    and $j$. In other words,   $Q_n$ is, for $n\geq 2$, a finite subgroup in $SU_2$ of type $D$ in
    the standard ADE classification (and in that context is also sometimes referred to, confusingly,
    as the ``dihedral group").  Another  traditional name for $Q_n$ is the {``dicyclic group"}
    \cite[\S 7.2]{coxeter}.  We have the central extension 
    \begin{equation}\label{eq:p_n} 1\to
	    \ZZ/2\lra Q_n\buildrel \pi_n \over \lra D_n\to 1, \quad \pi_n(\omega) = \omega
    \end{equation} 
    For $n\geq 3$, $D_n$ embeds into the group of automorphisms of $\mu_n\simeq
    \ZZ/n$, while $Q_n\subset\HH^*$ normalizes $\mu_n$, so $\pi_n$ is obtained by looking at the
    conjugation action of $Q_n$.

 The morphisms of $\nabla$ are defined to be generated by the morphisms in $\Delta$
 and by elements of the groups $\Aut_\nabla(\check n):=Q_{n+1}$ subject to the relations
 spelled out in \cite{loday}, Prop. 6.3.4(e) (proof). Thus, for $\nabla$ we have
 \[
 \Gen_n = Q_{n+1}, \quad |\Gen_n|= 4(n+1). 
 \]
 \end{exa}
 
 \begin{exa}\label{ex:N-quat}
  The {\bf $N$-quaternionic category} $\nabla_N$ is the crossed simplicial group
  obtained from $\nabla$ by $N$-fold subdivision, similarly to $\Lambda_N$
  and $\Xi_N$. In other words, the simplicial set associated to $\nabla_N$
  is $(Q_{N(n+1)})_{n\geq 0} = \on{sd}_N^*\bigl( (Q_{n+1})_{n\geq 0}\bigr)$,
  see \cite{fiedorowicz-loday}, Ex. 7.

\end{exa}

\begin{defi}   The crossed simplicial groups from Examples
\ref{ex:cyclic}-\ref{ex:N-quat} will be called the {\em planar crossed simplicial groups}. 
\end{defi}

Let $\Delta \Gen$ be a crossed simplicial group.  A  functor $X: (\Delta \Gen)^{\op} \to \C$
with values in a category $\C$ will be called a $\DG$-{\em object} of $\C$. 
  Via the embedding $i: \Delta \to \Delta \Gen$, we can describe $X$ as
the simplicial object $i^* X$ equipped with additional structure: For every $n \ge 0$, the object
$X_n$ carries an action of the group $(\Gen_n)^{\op}$, compatible with the simplicial structure according to the
relations in $\Delta \Gen$. In the case when $\Delta\Gen$ is one of the planar crossed simplicial
groups $\Lambda$,$\Xi$, $\Lambda_\infty$, etc, we will speak about {\em cyclic, dihedral,
paracyclic, etc,} objects in $\C$. 

For example, the simplicial set $\Gen$ extends naturally to a $\DG$-object, as it follows from
Example \ref{ex:extreme}.

\begin{exa} A paracyclic object $X: (\Lambda_{\infty})^{\op} \to \C$ can be described as a
simplicial object $i^* X$ equipped with automorphisms $t_n: X_n\to X_n$, $n \ge 0$, satisfying the relations 
\begin{align*}
	\partial_i t_n & = \begin{cases}
	\partial_n & \text{for $i = 0$,}\\
	t_{n-1}\partial_{i-1} & \text{for $1 \le i \le n$,}
	\end{cases} & 
	s_i t_n & = \begin{cases}
	t_{n+1}^2 s_n & \text{for $i = 0$,}\\ 
	t_{n+1}s_{i-1} & \text{for $1 \le i \le n$.}
	\end{cases} 
\end{align*}
An $N$-cyclic object admits an identical description with the additional condition that the automorphism
$t_n$, $n \ge 0$, has order $N(n+1)$. In particular, for $N=1$ we obtain an explicit description of cyclic objects.
\end{exa}

\subsection{Relation to planar Lie groups}

\subsubsection{Connective coverings and 2-groups}
In this paper, all topological groups will be assumed to have the homotopy type of a CW-complex. 
Let $G$ be a topological group. We denote by $G_e\subset G$ the connected component of the identity. 
The two structures on $G$ give rise to the following algebraic data:
\begin{enumerate}
\item[(1)] The possibly nonabelian group $P=\pi_0(G)$.

\item[(2)] The abelian group $A=\pi_1(G_e, e)$ equipped with a natural $P$-action induced by conjugation in $G$.

\item[(3)] The cohomology class $\gamma=\gamma_G\in H^3(P,A)$ defined as follows. For every connected
component $a\in P=\pi_0(G)$, we choose a representative $g_a\in G$. For every pair $a,b\in P$,
we choose a path $\xi_{a,b}$ joining $g_ag_b$ and $g_{ab}$. We denote by $\xi_{a,b}^{-1}$ the same path
run in the opposite direction. Then, for every triple $a,b,c\in P$, we have a loop
in the component $abc$:
\[
\xymatrix{
g_a g_b g_c \ar[r]^{\xi_{a,b}\cdot g_c}& g_{ab}g_c
\ar[d]^{\xi_{ab, c}}
\\
g_a g_{bc}\ar[u]^{g_a\cdot \xi_{b,c}^{-1}}&\ar[l]^{\xi_{a,bc}^{-1}} g_{abc}
}
\]
By multiplying this loop pointwise on the left with $g_{abc}^{-1}$, we obtain a loop $L(a,b,c)$ in
$G_e$ based at $e$. We denote by $\gamma(a,b,c) \in A$ the class of $L(a,b,c)$. 
The collection $\{ \gamma(a,b,c) \}$ forms a 3-cocycle of $P$ with coefficients in $A$ 
whose cohomology class is independent on the choices.
\end{enumerate}

\begin{defi}
A morphism $p:\widetilde G\to G$ of topological groups is called a {\em connective covering},
if $p$ is a covering of its image, and $p^{-1}(G_e)$ is connected, that is,
coincides with $\widetilde G_e$. A {\em proper covering} is a surjective connective covering. 
\end{defi}

Note that the condition $p^{-1}(G_e)=\widetilde G_e$ is equivalent to saying that the induced homomorphism
$p_*:\pi_0(\widetilde G)\to\pi_0(G)$ is injective. 

We denote by $\on{Con}(G)$ the category of proper coverings of $G$, with morphisms being morphisms
of topological groups commuting with the projections to $G$. Any $\widetilde G\in\on{Con}(G)$ gives rise to subgroups
\begin{equation}
	\label{eq:tilde-P-A}
	\begin{gathered}
		\widetilde P \,\,=\,\,\pi_0(\widetilde G)\,\, \subset\,\, P\,\,=\,\,\pi_0(G),
		\\
		\widetilde A \,\,=\,\, \pi_1(\widetilde G_e,e)\,\, \subset\,\, A\,\,=\,\,\pi_1(G_e,e),
	\end{gathered}
\end{equation}
with $\widetilde A$ being $P$-invariant. 

The structure of connective coverings of a given $G$ is classically known to depend on the algebraic
data $(P,A,\gamma)$ above. We refer to \cite{taylor, brown-mucuk} for detailed background and formulate
here the following summary result.

\begin{prop}\label{prop:summary-covers}
\begin{itemize}
\item[(a)] Suppose we are given subgroups $\widetilde P\subset P$, $\widetilde A\subset A$ so that $\widetilde A$ 
is $P$-invariant. For the existence of a connective covering $\widetilde G$ realizing $\widetilde P$ and $\widetilde A$
as in \eqref{eq:tilde-P-A}, it is necessary and sufficient that the image of $\gamma_G$ in $H^3(\widetilde P, A/\widetilde A)$
be $0$.

\item[(b)] In particular, $G$ has a ``universal covering" (a proper covering $\widetilde G$ with $\widetilde G_e$
simply connected), if and only if $\gamma_G=0$. 

\item[(c)] If $f: G_1\to G_2$ is a morphism of topological groups which is a homotopy equivalence of topological
spaces, then pullback under $f$ gives an equivalence of categories
\[
f^*: \on{Con}(G_2)\lra\on{Con}(G_1). 
\]
\end{itemize}
\end{prop}

We note that algebraic data $(P,A,\gamma)$ as above appear in the classification of the following related
types of objects:
\begin{enumerate}
	\item {\bf Connected homotopy $2$-types,} i.e., homotopy types of connected CW-complexes $X$
with $\pi_{> 2}(X,x)=0$. In this context, $P=\pi_1(X,x)$, $A=\pi_2(X,x)$. 
A topological group $G$ produces such $X$ by forming the classifying space $BG$ and then 
killing $\pi_{> 2}$. 

\item {\bf 2-groups} (also known as {\em non-symmetric Picard categories}. By definition, a 2-group
is an essentially small monoidal category $({\bf G}, \otimes, {\bf 1})$ in which each object is invertible
(up to isomorphism) with respect to $\otimes$ and each morphism is invertible with respect to composition.
In this case
\[
P\,\,=\,\,\on{Pic}({\bf G}) \,\,=\,\, \bigl( \Ob({\bf G})/\on{iso}, \otimes\bigr), \quad A=\Aut_{\bf G}({\bf 1}). 
\]
A topological group $G$ produces a 2-group ${\bf G}=\Pi_1(G)$, the fundamental groupoid
with all $g\in G$ taken as base points, with $\otimes$ given by the product in $G$.
Conversely, a 2-group $\bf G$ has a realization (loop space of the classifying space) $|{\bf G}|$
which can be realized as a topological group with $\pi_0=G, \pi_1=A$.

\item {\bf Crossed modules,}
i.e.,  data consisting of a morphism of groups
\[
\bigl\{ K^{-1} \buildrel\partial \over\lra K^0\bigr\}
\]
and an action $\epsilon: K^0\to \Aut(K^{-1})$ satisfying the axioms
\[
\begin{gathered}
\epsilon(\partial(k_{-1}))(k'_{-1}) \,\,=\,\, k'_{-1} \cdot k_{-1} \cdot (k'_1)^{-1},\quad k_{-1}, k'_{-1} \in K^{-1},
\\
\partial(k_{-1})\cdot k_0 \cdot \partial(k_{-1})^{-1} \,\,=\,\, \partial (\epsilon(k_0)(k_{-1})),
\quad
k_0\in K^0, k_{-1}\in K^{-1}.
\end{gathered}
\]
These axioms ensure that $\on{Im}(\partial)$ is a normal subgroup in $K^0$, so we have the
group $P=K^0/\on{Im}(\partial)$, and that $A=\Ker(\partial)$ lies in the center of $K^{-1}$,
in particular, $A$ is abelian (and $P$ acts on $A$ via $\epsilon$). 
A topological group $G$ defines a crossed module by taking $K^0=G$ and $K^{-1}=\widetilde{G}_e$
to be the universal covering of $G_e$ with base point $e$, see \cite{brown-mucuk} and references therein.
Note the similarity between the above axioms of a crossed module and the identities holding in
any crossed simplicial group (Proposition \ref{prop:skeleton}(c)). 
\end{enumerate}

 \subsubsection{Planar Lie groups}\label{subsub:plan-lie}
 
 Let $S^1$ be the standard unit circle in $\CC$, and  $\Homeo(S^1)$ be the group of all self-homeomorphisms
 of $S^1$ with the standard (compact-open) topology.
 We have morphisms of topological groups
 \[
 O(2) \,\,\hookrightarrow \,\, GL(2,\RR) \lra \Homeo(S^1)
 \]
 which are homotopy equivalences, Therefore, by Proposition \ref{prop:summary-covers}(c), the categories
 of connective coverings of these three groups are naturally identified. 
 
 Connective coverings of $O(2)$ will be called {\em (thin) planar Lie groups}. We will consider
 them as basic objects notation-wise and will denote them by standard Roman letters such as $G$.
 
 Connective coverings of $GL(2, \RR)$ will be called {\em thick planar Lie groups}. We will
 denote them by $G^\diamond$, where $G$ is a thin planar Lie group. For example,
 $SO(2)^\diamond = GL^+(2,\RR)$ is the subgroup of matrices with positive determinant. 
 
 Connective coverings of $\Homeo(S^1)$ will be called {\em planar homeomorphism groups}.
 We will denote them by $\Homeo^G(S^1)$, where $G$ is as above.

We recall the well-known classification of planar Lie groups which is an analog of the theory of
Pin groups  \cite{ABS} extending $O(d)$ by $\ZZ/2$ for any $d$. Remarkably, the groups in this
classification correspond very precisely to planar crossed simplicial groups.  We start with
describing the examples and formulating the correspondence. 
    
\begin{exas}\label{ex:planar}
    \begin{enumerate}[label=(\alph*)]
	    \item The group $SO(2)$ corresponds to the cyclic category $\Lambda$. 
	    \item The group $O(2)$ corresponds to the dihedral category $\Xi$. 
	    \item We denote by $\widetilde{SO}(2)\to SO(2)$ the universal covering of $SO(2)\simeq S^1$.
		It corresponds to the paracyclic category $\Lambda_\infty$.
	    \item We denote 
		   \[
		   \widetilde O(2) = (\ZZ/2)\ltimes \RR \lra O(2) = (\ZZ/2)\ltimes S^1
		   \]
		   the proper covering  of the semi-direct product induced from the universal covering $\RR\to S^1$.
		   It corresponds to the paradihedral category $\Xi_\infty$. 
	   \item Consider the group
		     \[
		  P \,\,=\,\,\bigl\{ q\in\HH^* \bigl| \,\, |q|=1, q S^1 q^{-1} \subset S^1 \bigr\}.
		  \]
		  Here $S^1$ is considered as the unit circle subgroup in $\CC^*\subset\HH^*$.
		   In other words, $P$ is the normalizer of a maximal torus in $SU(2)$. 
		   Like $O(2)$, the group $P$ has  $\pi_0=\ZZ/2$ and  $P_e= S^1$, but the
		   corresponding
		  extension of $\ZZ/2$ by $S^1$  is different from the one provided by $O(2)$. 
		  Similarly to \eqref{eq:p_n},  looking at the conjugation action gives a central extension
		   \[
		   1\to \ZZ/2\lra P \buildrel \pi \over \lra O(2)\to 1
		   \]
		   with $\pi$ a proper covering. This makes  $P$ into a planar Lie group.
		   It corresponds to the quaternionic category $\nabla$. 

	   \item Let $N\geq 1$. We denote the unique $N$-fold proper covering of $SO(2)$ by 
		   \[
		   1\to \ZZ/N \lra \on{Spin}_N(2)\buildrel \phi_N\over \lra SO(2) \to 1.
		   \]
		  As an abstract Lie group, $\on{Spin}_N(2)$ is identified with $SO(2)$, and $\phi_N$
		  is the $N$th power homomorphism. The group $\on{Spin}_N(2)$ corresponds to the
		  $N$-cyclic category $\Lambda_N$.

	    \item Let us represent elements of $O(2) = (\ZZ/2)\ltimes S^1$
		   as pairs $(\alpha, z)=\alpha\cdot z$ with $\alpha\in\ZZ/2$, $z\in S^1$, so that 
		  $\alpha z\alpha^{-1} = z^{-1}$.   Then the map
		     \[
		  f_N: O(2)\lra O(2), \quad (\alpha, z)\mapsto (\alpha, z^N).
		   \]
		   is a (continuous) homomorphism and, moreover, a proper covering with kernel $\ZZ/N$. 
		   We denote this covering
		   \[
		   1\to \ZZ/N \lra \on{Pin}^+_N(2) \buildrel p_N^+ \over\lra O(2)\to 1
		   \]
		   Thus $ \on{Pin}^+_N(2)=O(2)$ as an abstract Lie group, and $p_N^+=f_N$. 
		   The planar Lie group $\on{Pin}^+_N(2)$ corresponds to the $N$-dihedral category $\Xi_N$.
	   \item If $N=2M$ is even, then 
		    the composition $P\buildrel \pi \over\to O(2) \buildrel f_M\over\to O(2)$
		 is a proper covering with kernel $\ZZ/N$. We denote this covering
		  \[
		   1\to \ZZ/N \lra \on{Pin}^-_N(2) \buildrel p_N^- \over\lra O(2)\to 1
		   \]
		   Thus $ \on{Pin}^-_N(2)=Q_\infty$ as an abstract Lie group, and $p_N^-=f_M\pi$. 
		    The planar Lie group $\on{Pin}^-_{N}(2)$ corresponds to the $M$-quaternionic category
		    $\nabla_M$. 
\end{enumerate}
\end{exas}

      \begin{prop}  
      The groups  from Examples \ref{ex:planar}
      exhaust all planar Lie groups up to isomorphism. 
     \qed
    \end{prop}
    
    \noindent {\sl Proof:} The question reduces to classification of proper coverings of $O(2)$
    for which see, e.g.,  \cite[\S 8]{taylor}. \qed
    
    \vskip .2cm

   \begin{table}
  \renewcommand{\arraystretch}{2}
  \begin{tabular}{ p{3cm} | p{2cm}| p{3cm} | p{4cm} | p{3cm} }
  CS group $\Delta\Gen$ &Sequence of groups
  $\Gen_n$ &Lie group $G=|\Gen|$& Type of $G^\diamond$-structured surfaces& Extra data on associative algebras
  and categories \\
   \hline\hline
   Cyclic category $ \Lambda$   & $\ZZ/(n+1)$ &$SO(2) $ & Oriented surfaces& No data  \\ \hline
    Dihedral category $\Xi$  & $D_{n+1}$&$O(2)$  &Unoriented surfaces & Anti-automorphism
    of order 2
     \\ 
    \hline
    Paracyclic category $\Lambda_\infty$ & $\Gen_n=\ZZ$&$\widetilde{SO}(2)$
     &Framed surfaces & Automorphism
     \\ \hline
     Paradihedral category $\Xi_\infty$& $\Gen_n = D_\infty$ & $\widetilde O(2)$& Surfaces with framing
     on orientation cover & Automorphism $t$ and involution $\omega$ s.t. $\omega t\omega = t^{-1}$.
    \\
    \hline
    $N$-cyclic category $\Lambda_N$&$\ZZ/N(n+1)$ &$\on{Spin}(2)_N$
    &$N$-spin (oriented) surfaces&  Automorphism of order $N$ 
    \\ \hline
    $N$-dihedral category $\Xi_N$& $D_{N(n+1)}$ & $\on{Pin}^+_N(2)$ & Unoriented,  with $\on{Pin}^+_N(2)$-structure&
    Twisted action of $D_N$.
     \\ \hline
    $M$-quaternionic category $\nabla_M$ & $Q_{M(n+1)}$&$\on{Pin}^-_{2M}(2)$ & Unoriented,  with $\on{Pin}^-_{2M}(2)$-structure
    & Twisted action of $Q_M$
       \end{tabular}
       \caption{ 
       Planar crossed simplicial group  and associated structures.}\label{table:corr}
\end{table}

The correspondence between crossed simplicial groups and planar Lie groups are collected in the
first three columns of Table \ref{table:corr}. The fourth column means that a crossed simplicial
group of Lie type ``governs" the combinatorial topology of the corresponding class of structured
surfaces. This will be explained in more details in \S \ref {subsec:structuredgraphs}. The last column will be explained in \S
\ref{subsec:str-nerv} and
\ref {sec:str-hoch}.

We further arrange planar Lie groups and natural homomorphisms between them  on the left in  Table
\ref{table:galois} and observe that these homomorphisms are matched by natural functors between
corresponding crossed simplicial groups. Each homomorphism between Lie groups in the table is either
a surjection, in which case we label the arrow with the kernel of the surjection, or is an inclusion
of a normal subgroup with quotient $\ZZ/2$. Note that the group $\on{Pin}^-_{2M}(2)$, isomorphic to
$P\subset\HH^*$, does not admit a simply-connected proper covering \cite[\S 8]{taylor}. 
    
\vfill\eject
   
\begin{table}[h]
   \[
    \xymatrix{
   &\widetilde{O}(2)
   \ar[dd]_{ \ZZ}
    &
   \\
  \widetilde{SO}(2) \ar[ur]
   \ar[dd]_{ \ZZ}&&
   \\
   &\on{Pin}^+_N(2)
    \ar[dd]_{\ZZ/N}
&\on{Pin}^-_{2M}(2)
\ar[ddl]^{\ZZ/2M}
   \\
   \on{Spin}_N(2)
   \ar[ur]
   \ar[dd]_{\ZZ/N}
  &&
   \\
  & O(2) &\\
   SO(2)
   \ar[ur]&&
   }
\quad\quad\quad
   \xymatrix{
   &\Xi_\infty 
   \ar[dd]_{}\\
  \Lambda_\infty \ar[ur]_{}
   \ar[dd]_{}&&
   \\
   &\Xi_N 
    \ar[dd]_{}
&\nabla_M
\ar[ddl]^{}
   \\
   \Lambda_N
   \ar[ur]_{}
   \ar[dd]_{}
  &&
   \\
  & \Xi &\\
   \Lambda
   \ar[ur]_{}&&
   }
 \]
   \caption{Correspondence between planar crossed simplicial groups and planar Lie groups.}\label{table:galois}
\end{table}
    
The following statement materializes the correspondence and systematizes the results of several authors.
    
    \begin{thm}\label{thm:crossed-groups}
     \begin{enumerate}
    \item[(a)] Let $\Delta\Gen$ be any crossed simplicial group. Then the geometric realization $|\Gen|$
    of the simplicial set $\Gen$ has a natural group structure, making it a topological group.
    Further, the geometric realization $|\N (\Delta\Gen)|$ of the nerve of $\Delta\Gen$
    is homotopy equivalent to the classifying space $B |\Gen|$. 
    
    \item[(b)]  Let $\Delta\Gen$ be a  planar crossed simplicial group. Then:
    \begin{enumerate}
    \item[(b1)]  $|\Gen|=G$ is the
    planar Lie group corresponding to $\Delta\Gen$ in Table \ref{table:corr}. In particular, there
    is a covering $p: G\to O(2)$. 
    
    \item[(b2)] The group $\Gen_n$ is identified with the subgroup $p^{-1}(D_{n+1})$,
    where $D_{n+1}$ is embedded into $O(2)=(\ZZ/2)\ltimes S^1$ as
    $(\ZZ/2)\ltimes\mu_{n+1}$. 
   \end{enumerate}
    
    \item[(c)] Let $X:\Delta\Gen^\op\to\Set$ be any $\Delta\Gen$-set. Then $|i^*X|$, the geometric
    realization of $X$ considered as a simplicial set, has a natural $G$-action. 
    
    \item[(d)] The homotopy category of $G$-spaces can be identified with the localization 
    of  the category $\Fun(\Delta\Gen^\op, \Set)$ along morphisms which give homotopy equivalences
    of the induced simplicial sets. 
     
    \end{enumerate}
    \end{thm}
    
    \begin{proof} (a)  was shown by Fiedorowicz and Loday \cite{fiedorowicz-loday}.  The identification of
   $|\Gen|$ in the examples of (b) is also immediate from their considerations. The remaining parts
   are well known \cite{DHK} \cite[Ch.7]{loday} in the cyclic and paracyclic cases
    $\Delta \Gen=\Lambda,
   \Lambda_\infty$,
   and were extended by Dunn \cite{dunn} to the dihedral and quaternionic cases. 
   The   $N$-subdivided  cases can be deduced from these using the fact that $\on{sd}_N^*$
   does not change the geometric realization. The paradihedral case can be
   deduced from the paracyclic one by identifying $\Xi_\infty$ with the semi-direct product of
   $\ZZ/2$ and $\Lambda_\infty$. We leave the remaining details to the reader.
\end{proof}


\subsection{Structured nerves }\label{subsec:str-nerv}

In this  and the following section, we explain the last column of Table \ref{table:corr} which puts planar crossed
simplicial groups in correspondence with algebras and categories with additional types of extra data
such as involutions, automorphisms, etc.  For this, we systematize and expand the ideas of
Loday \cite[Ch. 5]{loday} about ``variations on cyclic homology".

\subsubsection{Canonical parity and the $\DG$-nerve}
Recall (Example \ref{ex:semiconst-emb}) that 
 any crossed simplicial group $\DG$ contains the semiconstant crossed simplicial group $\Delta
\{\Gen_0\}$ generated by $\Gen_0$. By Proposition \ref{prop:semi}, the group $\Gen_0$ acts on $\Delta$ and is
hence equipped with a {\em canonical parity} $\Gen_0 \to \Aut(\Delta) \cong \Zt$.  We explicitly
describe the canonical parity of $\Gen_0$ for planar crossed simplicial groups.

\begin{prop}\label{prop:orient-G-0}
Let $\DG$ be a planar crossed simplicial group with the corresponding planar Lie group $p: G \to O(2)$. Then:
\begin{enumerate}
\item[(a)] Let $O(1)=\ZZ/2$ be embedded into $O(2)$ in the standard way. Then $\Gen_0$ is identified,
as a group with parity, with the preimage of $O(1)$ in $G$, so we have the Cartesian square
\[
\xymatrix{
\Gen_0\ar[r]^\epsilon \ar[d]^{\hskip -.7cm \rho} & G\ar[d]^p\\
\ZZ/2=O(1) \ar[r]& O(2). 
}
\]
\item[(b)] Explicitly, the canonical parities of $\Gen_0 = \ZZ, (\ZZ/2)\ltimes \ZZ, \ZZ/N, D_N, Q_N$
are given as follows:  
 the groups  $\ZZ$ and $\ZZ/N$  are equipped with the trivial homomorphism to
 $\ZZ/2$, and the group  $D_N$ and $Q_N$ with homomorphisms
\[
 \begin{gathered}
   \rho_D: D_N\lra  \ZZ/2, \quad \rho_D(\omega)=\overline 1, \,\,\rho_D(\tau)=\overline 0, 
 \\
 \rho_Q: Q_N\lra  \ZZ/2, \quad \rho_Q(w)=\overline 1, \,\,\rho_Q(\tau)=\overline 0,
 \end{gathered}
 \]
 and $(\ZZ/2)\ltimes\ZZ$ with the projection to $\ZZ/2$. 
\end{enumerate}
\end{prop}
\begin{proof}We recall that $G$ is identified with $|\Gen|$ by Theorem 
\ref{thm:crossed-groups}(b1). A homomorphism $\epsilon: \Gen_0\to |\Gen|$
is obtained from the compatible system of homomorphisms $\omega_n^*: \Gen_0\to \Gen_n$. 
The verification that $\epsilon$ is an embedding with the required properties, is obtained using case
by case analysis.
\end{proof}

Let $\DG$ be any crossed simplicial group, and let $n\geq 0$. We call the representable $\DG$-set
\[
\DG^n := \Hom_\DG(-, [n]): \DG^\op\lra\Set
\]
the $n$-dimensional {\em $\DG$-simplex}. In particular, for $\Lambda$, $\Xi$, etc, we will speak
about {\em cyclic simplices}, {\em dihedral simplices}, etc. As an immediate corollary of the
results of Section \ref{section:semiconstant}, we have the following statement.

\begin{prop}\label{prop:cosimplex} For any crossed simplicial group $\DG$, we have a natural functor 
	\[
		\DG \lra \Gen_0-\Cat,\; [n] \mapsto \FC(\DG^n|_{\Delta})
	\]
	where 
	\[
		\Gen_0-\Cat = \Fun_{\BZt}(\B \Gen_0, \Zt \ltimes \Cat)
	\]
	denotes the category of small categories equipped with a twisted
	action of $\Gen_0$.
\end{prop}
\begin{proof}  Denoting $\on{Res}$ the restriction from $\DG$ to $\Delta\{\Gen_0\}$,  we obtain
	a sequence of functors
	\[
		\DG \overset{\text{Yoneda}}{\lra} \Set_{\DG} \overset{\on{Res^*}}{\lra}
		\Set_{\Delta \{ \Gen_0 \}} \overset{\ref{prop:main}}{\lra} \Gen_0-\Set_{\Delta}
			\overset{\FC}{\lra}
			\Gen_0-\Cat
	\]
	whose composite is the desired functor.
\end{proof}

As a consequence, we obtain the following result. 

\begin{thm}\label{thm:G-nerve}
Let $\DG$ be any crossed simplicial group. Equip $\Gen_0$ with its canonical parity.
Then to each small category $\C$ with a twisted $\Gen_0$-action, there is an associated
$\DG$-set $\N^\DG(\C)$ called the {\em $\DG$-nerve} of $\C$ such that
\[
	\N^\DG(\C)_n = \Hom_{\Gen_0-\Cat}(\FC(\DG^n_{|\Delta}), \C).
\]
\end{thm}
\begin{proof}
	This follows by postcomposing the functor of Proposition \ref{prop:cosimplex} with
	$\Hom_{\Gen_0-\Cat}(-, \C)$.
\end{proof}

\subsubsection{Examples}
We  now give a more explicit description of $\N^\DG(\C)$ for each planar $\DG$. 

\vskip .2cm

  \noindent {\bf Cyclic case (Connes).} Here $\Gen_0 = \ZZ/1$ is the trivial group. 
  Let $\C$ be a small category. Its {\em cyclic nerve} is the simplicial set $\NC_\bullet\C$
  with $\NC_n\C$ being the set of diagrams
  \be\label{eq:NC}
  x_0\buildrel a_0\over  \lra x_1\buildrel a_1\over \lra \cdots\buildrel a_{n-1}\over \lra x_n
  \buildrel a_n\over \lra x_0. 
  \ee
  As i well-known, the cyclic rotation of such diagrams makes $\NC_\bullet\C$ into a cyclic set.

\vskip .2cm
  
    \noindent {\bf Dihedral case (Loday).} Here $\Gen_0=D_1\cong\ZZ/2$ with the identity homomorphism to
  $\ZZ/2$, so that a twisted $\Gen_0$-action is an involution. 
  
  Let $(\C,*)$ be a small category with  involution,
  which we write  as an isomorphism $*: \C\to\C^\op$, $**=\on{Id}$.  Then $\NC_\bullet\C$ becomes a
  dihedral set, with the generator $\omega_n \in D_{n+1} = \Aut \overline \cn$ sending a diagram
  \eqref{eq:NC} into
  \[
 x_1^*\buildrel a_0^*\over\lra  x_0^*\buildrel a_n^*\over \lra x_{n}^*\buildrel a_{n-1}^*\over  \lra x_{n-1}^*
  \buildrel a_{n-2}^*\over \lra\cdots \buildrel a_2^*\over \lra x_2^* \buildrel a_1^*\over\lra x_1^*. 
  \]

  \vskip .2cm
  
  \noindent {\bf Paracyclic and $N$-cyclic cases.} In the paracyclic case $\Gen_0=\ZZ$ with trivial homomorphism into
  $\ZZ/2$, so a twisted $\Gen_0$-action is an automorphism. In the $N$-cyclic case a twisted $\Gen_0=\ZZ/N$-action
  is an automorphism of order $N$. 
  
  Let $\C$ be a category equipped with a (covariant) automorphism
  $F$. We have then the {\em twisted cyclic nerve} $\NC^F_\bullet\C$, see \cite{HSS1}. This is a simplicial set
  whose $n$-simplices are diagrams
 \be\label{eq:NC^F}
  x_0\buildrel a_0\over  \lra x_1\buildrel a_1\over \lra \cdots 
  \buildrel a_{n-1}\over \lra x_n
  \buildrel a_n\over \lra F(x_0). 
  \ee
  We have the transformation $\tau_n: \NC_n^F\C \to\NC_n^F\C$ sending such a diagram to
  \be\label{eq:tau-twisted}
  F^{-1}(x_n) \buildrel F^{-1}(a_n)\over\lra x_0  
  \buildrel a_0\over  \lra x_1\buildrel a_1\over \lra \cdots
  \buildrel a_{n-2}\over\lra x_{n-1}
  \buildrel a_{n-1}\over \lra x_n =  F(F^{-1}(x_n)). 
  \ee
   One sees directly that the $\tau_n$ make $\NC^F_\bullet \C$ into a paracyclic set. 
   If $F^N=\on{Id} $ for some $N$, then  
   $\tau_n^{N(n+1)}=\on{Id}$ as well, and so $\NC^F_\bullet \C$ is an  $N$-cyclic set. 
   
   The twisted cyclic nerve is the categorical analog of the {\em twisted loop space} $L^F X$
   associated to a topological space $X$ together with a homeomorphism $F: X\to X$. 
   Explicitly, $L^F X$ consists of continuous maps $\gamma: \RR\to X$ such that $\gamma(t+1)= F(\gamma(t))$. 
   
    \vskip .2cm
  
  \noindent {\bf Quaternionic case (Dunn): $\Gen_0=\ZZ/4$.} Quaternionic homology  was originally introduced by
  Loday \cite{loday:HDQ} \cite[\S 5.2]{loday} as a  purely technical modification of dihedral
  homology in the case when $2$ is not invertible in the base ring.
    However, from our perspective, the more natural framework for quaternionic
  homology is provided by
    algebras and categories with a {\em twisted action of $\Gen_0=Q_1=\ZZ/4$},
   i.e., with an {\em anti-automorphism $J$ of order 4}. In this case $J^2$ is an automorphism of
  order 2. This point of view  goes back to Dunn \cite[\S 2]{dunn}. 
  
  \vskip .2cm
  
  More precisely, let $\C$ be a category and $J: \C\to \C$ be a contravariant functor such that $J^4=\on{Id}$. 
  Consider the 2-cyclic set $\NC^{J^2}_\bullet\C$. Define $w=w_n: \NC^{J^2}_n\C\to \NC^{J^2}_n\C$
  by
  \[
  \begin{gathered}
  w_n \bigl\{   x_0\buildrel a_0\over  \lra x_1\buildrel a_1\over \lra \cdots 
  \buildrel a_{n-1}\over \lra x_n
  \buildrel a_n\over \lra J^2 x_0\bigr\} = 
  \\
  \bigl\{ J^{-1}x_1 \buildrel J^{-1}(a_0)\over\lra
  J^{-1}(x_0)\buildrel J(a_n)\over\lra J(x_n)
  \buildrel J(a_{n-1}) \over\lra 
  J(x_{n-1})\buildrel J(a_{n-2})\over\lra\cdots
  \buildrel J(a_1)\over
  \lra J(x_1)= J^2(J^{-1}x_1)   \bigr\} .
    \end{gathered}
  \]
  (Note that the source of $J(a_n)$ is $J(J^2x_0) = J^{-1}x_0$.)
   Then $w_n$ together with $\tau_n$ defined as in \eqref{eq:tau-twisted}, define an action of
  $Q_{n+1}$ on $\NC^{J^2}_n\C$. One verifies that in this way $\NC^{J^2}_\bullet\C$
 becomes a quaternionic set. 
 
  \vskip .2cm
 
 \noindent {\bf $N$-dihedral case: $\Gen_0=D_N$.} Let $\C$ be a category with a twisted action of $D_N$, i.e.,
 with a contravariant functor $\Omega: \C\to \C$ and a covariant functor $T: \C\to\C$ satisfying the relations
 of $D_N$. In particular, $T^N=\on{Id}$,  so we have an $N$-cyclic set $\NC^T_\bullet\C$. 
 For a given $n$, the set $\NC_n^T\C$ is acted upon by the transformation $\tau_n$ as in
 \eqref{eq:tau-twisted} and $\omega_n$ which takes an $n$-simplex 
 \[
  x_0\buildrel a_0\over  \lra x_1\buildrel a_1\over \lra \cdots 
  \buildrel a_{n-1}\over \lra x_n
  \buildrel a_n\over \lra T(x_0) 
 \]
 into 
 \[
 \Omega (x_1) \buildrel \Omega(a_0)\over\lra \Omega (x_0) 
 \buildrel T\Omega (a_n)\over\lra T\Omega(x_n) 
 \buildrel T\Omega (a_{n-1})\over\lra T\Omega(x_{n-1}) 
 \buildrel T\Omega (a_{n-1})\over\lra \cdots
 \buildrel T\Omega (a_1)\over\lra T\Omega(x_1).
  \]
  Note that by the  dihedral relations
  we can write the source of $\Omega(a_n)$ as $\Omega T x_0 = T^{-1}\Omega x_0$, so
  $T\Omega (a_n): \Omega x_0\to T\Omega x_n$.
  The maps $\tau_n$ and $\omega_n$ satisfy the relations of $D_{N(n+1)}$
  and one checks directly that in this way $\NC^T_\bullet\C$ becomes an $N$-dihedral
  set.

 \vskip .2cm
  
  \noindent {\bf $N$-quaternionic case: $\Gen_0=Q_N$.} Let $\C$ be a category with a twisted action of $Q_N$, i.e.,
 with a contravariant functor $W: \C\to \C$ and a covariant functor $T:\C\to\C$,
    satisfying the relations
 of $Q_N$. Then $T^{2N}=\on{Id}$,  so we have
 the $2N$-cyclic set $\NC^{T}_\bullet\C$. We define the action of $Q_{N(n+1)}$
 on $\NC^{T}_n\C$ by transformations $\tau_n$, defined  as in \eqref{eq:tau-twisted},
 and $w_n$ which takes a simplex
 \[
  x_0\buildrel a_0\over  \lra x_1\buildrel a_1\over \lra \cdots 
  \buildrel a_{n-1}\over \lra x_n
  \buildrel a_n\over \lra T(x_0) 
 \]
 into  
 \[
 W(x_1)\buildrel W(a_0))\over\lra W(x_0) \buildrel TW(a_n)\over\lra TW (x_n)\buildrel TW(a_{n-1})\over\lra
 TW(x_{n-1})\lra\cdots \buildrel TW(a_1)\over\lra TW( x_1). 
 \]
As before, the identity $WTW^{-1}=T^{-1}$ implies  that $TW(a_n)$ acts from $W(x_0)$ to $TW(x_n)$. 
The relations of the quaternionic group $Q_{N(n+1)}$ and the compatibility of the actions of the $Q_{N(n+1)}$, $n\geq 0$
and of $\Delta$ are verified directly.

\subsection{Structured Hochschild complexes and $\DG$-Frobenius algebras}\label{sec:str-hoch}

Let $\k$ be a field. By a {\em vector space} (resp. {\em algebra}) we will always mean a $\k$-vector
space, resp. a $\k$-algebra. We denote by $\Vect$ the category of vector spaces.

\subsubsection{The $\DG$-structure on the Hochschild complex}

We have the following enriched variant of Theorem \ref{thm:G-nerve} which we will only formulate
for algebras (instead of more general $\k$-linear categories).

\begin{thm}\label{thm:hochschild}
Let $\DG$ be a planar crossed simplicial group. Let $A$ be a unital associative algebra with a 
twisted action of $\Gen_0$. Then the collection of vector spaces
  \[
 C_\bullet(A) =   \bigl\{ C_n(A) = A^{\otimes (n+1)}\bigr\}_{n\geq 0},
  \]
given by the terms of the Hochschild complex of $A$, naturally forms a $\Delta\Gen$-vector space. 
In particular, the group $\Gen_n$ acts on $C_n(A)$. 
\end{thm}
\begin{proof}
The cosimplicial structure is obtained by letting face maps act by insertion of $1$ and degeneracy
maps act by multiplication in $A$. Let us construct an action of $\Gen_n$  on $A^{\otimes (n+1)}$.
Consider the wreath product
\[
S_{n+1}\wr \Gen_0 := S_{n+1}\ltimes \Gen_0^{n+1},
\]
where the action of $S_{n+1}$ on $\Gen_0^{n+1}$ is by permuting the Cartesian factors. 
Note that we have a projection map $S_{n+1}\wr\Gen_0 \to S_{n+1}$. We now claim that the
homomorphism $\lambda_n: \Gen_n\to S_{n+1}$ from Proposition \ref{prop:forget} has a natural lifting
to a homomorphism
\[
L_n: \Gen_n \lra S_{n+1}\wr \Gen_0. 
\]
Indeed, by the canonical factorization
\[
\Hom_\DG([0], [n]) = \Gen_0 \times\Hom_\Delta([0], [n]),
\]
and $\Gen_n$ acts on this set on the left by composition of morphisms in $\DG$.  We see, first,
that the action of any $g\in\Gen_n$ takes any fiber of the projection to $\Hom_\Delta([0], [n])$ to
another fiber. The induced permutation of these fibers is precisely $\lambda_n(g)\in S_{n+1}$.
Second, if we identify each fiber with $\Gen_0$, then the identification of one fiber with another
is given by the left multiplication with some element of $\Gen_0$, These elements, together with
$\lambda_n(g)$, give the desired $L_n(g)\in S_{n+1}\wr \Gen_0$. 

Since $\Gen_0$ acts on $A$ as a vector space, we have the natural action of $S_{n+1}\wr\Gen_0$ on
$A^{\otimes(n+1)}$ which combines the permutations of the tensor factors and action of $\Gen_0$ on
each factor. We pull this action back along $L_n$ to get a $\Gen_n$-action on $A^{\otimes(n+1)}$.
The compatibility of these actions with the simplicial structure is verified straightforwardly. 
\end{proof}

\begin{rem} The Hochschild complex of an associative algebra is usually constructed as the
	totalization of a {\em simplicial} $\k$-vector space. This simplicial object can be obtained
	by precomposing the functor $\DG \to \Vect$ from Theorem \ref{thm:hochschild} with
	\[
		\Delta^{\op} \lra \DG^{\op} \overset{D}{\lra} \DG 
	\]
	where $D$ denotes a self-duality of the category $\DG$ (cf. Corollary \ref{cor:duality}
	below). 
\end{rem}

\subsubsection{$\DG$-(co)homology, traces and Frobenius algebras}

We denote by $\DG-\Vect$ the category of $\DG$-vector spaces. Let $\underline\k$ be the constant
$\DG$-vector space with stalk $\k$, i.e., the contravariant functor which takes each object to $\k$
and each morphism to the identity.

Let $A$ be a unital associative algebra with a twisted $\Gen_0$-action. By Theorem
\ref{thm:hochschild}, $C_\bullet(A)$ is an object of $\DG-\Vect$. Following the approach of Connes
and Loday, we define the $\DG$-{\em homology} and {\em cohomology} of $A$ to be
\begin{align*}
H_\bullet^\DG(A) & = \on{Tor}_\bullet^{\DG-\Vect}( \k, C_\bullet(A)),\\
H^\bullet_\DG(A) & = \on{Ext}^\bullet_{\DG-\Vect}(  C_\bullet(A), \k) 
\end{align*}
where $\k$ denotes the constant $\DG$-module.

\begin{exa}
Elements of $H^0_\DG(A) = \on{Hom}_{\DG-\Vect}( C_\bullet(A), \k)$ will be called {\em
$\DG$-traces} on $A$. Thus a $\DG$-trace $\beta$ is a collection of multilinear forms $\beta_n:
A^{\otimes(n+1)}, \k)$, $n\geq 0$. The condition that $\beta$ is a morphism of cosimplicial vector
spaces, means that all the $\beta_n$ are expressible through the linear form $\beta_1$:
\[
\beta_n(a_0, a_1, \cdots, a_n) = \beta_1(a_0a_1\cdots a_n). 
\]
The further conditions on $\beta_1$ express $\Gen_n$-invariance of $\beta_n$ for $n\geq 0$.
Note that these conditions do not necessarily imply that the bilinear form $\beta_2$ is symmetric. 
\end{exa}

\begin{defi}
A {\em $\DG$-Frobenius algebra} is a finite-dimensional unital associative algebra $A$ equipped with
a twisted $\Gen_0$-action and a $\DG$-trace $\beta$ such that the bilinear form $\beta_2$ is
non-degenerate.  
\end{defi}
 
\subsubsection{Planar examples} 
Here we illustrate the above construction for some of the planar $\DG$. 

\vskip .2cm

\noindent {\bf Cyclic case (Connes).} Here $\DG=\Lambda$. 
Let $A$ be an associative algebra. The graded vector space $C_\bullet(A)$ can be seen as a
linearized, 1-object version of the cyclic nerve. The fact that $\NC_\bullet\C$ is cyclic,
corresponds to the fundamental observation of Connes that $C_\bullet (A)$ has a natural structure of
a cyclic vector space. In this structure, the generator $\tau_n$ of $\ZZ/(n+1)=\Aut \cn$ acts by
cyclic rotation
\[
\tau_n(a_0\otimes  \cdots \otimes  a_n) \,\,=\,\,  a_n\otimes a_0\otimes  \cdots \otimes a_{n-1}.
\]
The concept of a $\DG$-trace on $A$ reduces, for $\DG=\Lambda$, to  that of a linear functional
$\beta_1:A\to\k$ such that for each $n$ the multilinear form $\beta_1(a_0\cdots a_n)$ is cyclically
symmetric. As well known, this is equivalent to requiring that $\beta_1(ab)=\beta_1(ba)$, so
$\beta_1$ is the trace in the classical sense and $\beta_2$ is an invariant symmetric bilinear form
on $A$. Further, the concept of a $\Lambda$-Frobenius algebra reduces to the classical one: a not
necessarily commutative, associative unital $\k$-algebra with a trace $\beta_1$ such that
$\beta_1(ab)$ is a nondenerate bilinear form. It is classical that Frobenius algebras give rise to
topological quantum field theories (TQFTs) on oriented surfaces, see \cite{kock} for systematic
exposition. In this correspondence, commutative Frobenius algebras give rise to closed TQFTs defined on
$2$-dimensional oriented cobordisms without any additional data.  Not necessarily commutative
Frobenius algebra give rise to open TQFTs defined on oriented cobordisms together with a choice of marked
points which are combinatorial representatives of punctures, cf. \S \ref{subsec:c-infty}. 
  
\vskip .2cm
  
\noindent {\bf Dihedral case (Loday).} Here $\DG=\Xi$. Let $(A,*)$ be a unital associative algebra
with involution. As observed by Loday \cite[\S 5.2.11]{loday} $C_\bullet (A)$ becomes a dihedral
vector space, with $\omega_n$ acting by 
\[
  \omega_n(a_0\otimes  \cdots \otimes a_n) \,\,=\,\,  a_0^*\otimes a_n^*\otimes a_{n-1}^*\otimes
  \cdots \otimes a_2^* \otimes a_1^*. 
\]
A $\Xi$-trace on $A$ is a linear functional $\beta_1:A\to\k$ satisfying the identities
\[
	\beta_1(a^*)=\beta_1(a), \quad  \beta_1(ab) = \beta_1(ba),
\] 
which imply that $\beta_1(ab)=\beta_1(a^*b^*)$, i.e., that the involution preserves the bilinear
form $\beta_2$ (as well as all the higher $\beta_n$). $\Xi$-Frobenius algebras can be called
``Frobenius algebras with involution". Variations of these were used by Alexeevski and Natanzon
\cite{alex-natanzon} to construct TQFTs on unoriented surfaces. 

\vskip .2cm
  
\noindent {\bf Paracyclic and $N$-cyclic cases.} Here $\DG=\Lambda_\infty$ or $\Lambda_N$.  Suppose
$A$ is  a unital associative algebra equipped with an automorphism $F$. We then have an invertible
transformation $\tau_n$ on $C_n(A)$ given by
\begin{equation}
	\label{eq:tau-tw-alg}
	\tau_n(a_0, \cdots, a_n) = (F^{-1} (a_n), a_0, a_1, \cdots, a_{n-1}). 
\end{equation}
This makes $C_\bullet(A)$ into a paracyclic vector space. If $F^N=\on{Id}$ for some $N$, then
$C_\bullet(A)$ is an $N$-cyclic vector space. 
  
A $\Lambda_\infty$-trace on $A$ is a linear functional $\beta_1: A\to\k$ such that
\[
	\beta_1(F(a))=\beta_1(a), \quad  \beta_1(a\cdot b) = \beta_1(b\cdot F(a)). 
\]
In particular, the form $\beta_2(a,b)=\beta_1(ab)$ is not required to be symmetric, and $F$
intertwines $\beta_2$ and $\beta_2^t$:
\begin{equation}
	\label{eq:nakayama}
	\beta_2(a,b) = \beta_2(b, F(a)). 
\end{equation}
Note that in the Frobenius case ($A$ finite-dimensional and $\beta_2$, considered as a morphism
$A\to A^*$, is an isomorphism), $F$ is defined by \eqref{eq:nakayama} uniquely:
$F=\beta_2^{-1}\beta_2^t$. In this context $F$ is called the {\em Nakayama automorphism} of
$\beta_2$. We get the following.
  
\begin{prop}
\begin{enumerate}
\item[(a)] A $\Lambda_\infty$-Frobenius algebra is the same as a finite-dimensional unital
	associative algebra $A$ together with a linear functional $\beta_1: A\to\k$ such that:
\begin{enumerate}
\item[(a1)] $\beta_2(a,b)=\beta_1(ab)$ is a non-degenerate, not necessarily symmetric bilinear form
	on $A$.

\item[(a2)] The Nakayama automorphism $F$ of $\beta_2$ is an algebra automorphism of $A$. 
\end{enumerate}

\item[(b)] A $\Lambda_N$-Frobenius algebra is a $\Lambda_\infty$-Frobenius algebra such that $F^N=\Id$. 
\end{enumerate}
\qed
\end{prop}
  
For $N=2$, we obtain the concept which is similar to the Frobenius algebras studied in the work of
Novak and Runkel \cite{novak-runkel} where such algebras are used to construct TQFTs on surfaces
with 2-spin structure. 
  
\vskip .2cm
  
\noindent {\bf Quaternionic case (Dunn):}
Here $\DG=\nabla$. Let $A$ be an associative algebra with an anti-automorphism $J$ of order 4.
Then the transformation
\[
w_n(a_0\otimes \cdots \otimes a_n) \,\,=\,\, J^{-1}(a_0) \otimes J(a_n) \otimes \cdots \otimes J(a_1)
\]
together with $\tau_n$ defined as in \eqref{eq:tau-tw-alg} for $F=J^2$, make $C_\bullet (A)$ into a
quaternionic object in the category of vector spaces. 
 
A $\nabla$-trace on $A$ is a linear functional $\beta: A\to\k$ such  that $\beta_1(Ja) = \beta_1(a)$
and $\beta_2(a,b)=\beta_1(ab)$ satisfies the conditions:
\[
	\beta_2(a,b) = \beta_2(b, J^2 a) = \beta_2(J^{-1} a, Jb).
\]
 
The remaining planar cases can be analyzed in a similar way. 
 
\begin{rems} (a) Let $\DG$ be any planar simplicial group with corresponding planar Lie group $\GG$.
	We expect that the theory of $\DG$-structured graphs developed in \S
	\ref{sec:structuredgraphs} can be used to generalize the results of \cite{kock,
	alex-natanzon, novak-runkel} as follows: A $\DG$-Frobenius algebra gives rise to a numerical
	invariant defined on $\GG$-structured surfaces without boundary and to a TQFT on
	$\GG$-structured 2-dimensional cobordisms  (with marked points  as in  \S
	\ref{subsec:c-infty} allowed in both cases and with additional constraints of commutativity
	removing the dependence on marking points). 

	\vskip .1cm

	(b) Further, one can extend the concept of a $\DG$-Frobenius algebra to the case of
	associative dg (or $A_{\infty}$) algebras, and dg (or $A_\infty$) categories also allowing
	the forms $\beta_1$ and $\beta_2$ to have some degree $d\neq 0$ as in
	\cite{kontsevich-feynman}. The corresponding ``$d$-dimensional $\DG$-Calabi-Yau algebras and
	categories" should then give cohomology classes on moduli spaces of marked $\GG$-structured
	surfaces generalizing \cite{costello-TFT, kontsevich-feynman}.

	We leave both these directions to future work. 
\end{rems}

\newpage
\section{Crossed simplicial groups and generalized orders}\label{sec:order} 

\subsection{$\DG$-structured sets}
   
The simplex category $\Delta$, whose objects are the standard ordinals $[n]$, $n\geq 0$,  can be embedded
into a larger category ${\bf \Delta}$ with objects being all finite nonempty linearly ordered sets.
While this embedding is an equivalence of categories, it is often more convenient to work with the
larger category. In this section, we provide an analogous construction for any crossed simplicial
group $\DG$. We introduce a larger category $\G$ of $\DG$-structured finite sets, fitting into a diagram
of functors with vertical arrows being surjective on objects and horizontal arrows being
equivalences of categories:
\begin{equation}
	\label{eq:Gc}
	\xymatrix{
		\DG \ar[d]_\lambda  \ar[r]^\epsilon& \G \ar[d]^{\lambda_{\G}}\\
		{\bf N}\ar[r]& \FSet
	}
\end{equation}
here $\FSet$ is the category of all nonempty finite sets, $\bf N \subset \FSet$ is the full
subcategory on the standard objects $\{0,1,\dots, n\}$, and $\lambda$ is the functor from
Proposition \ref {prop:forget}. 
 
\subsubsection{Combinatorial model}
\label{sec:combmodel} 
  
If we assume that a category $\Gc$ as in \eqref{eq:Gc} is already constructed, then any object
$I\in\Gc$ is isomorphic to a unique $[n]\in\DG$ and so $\Isom_\Gc([n], I)$ is a right torsor over
$\Gen_n=\Isom_\DG([n], [n])$. This leads to the following two definitions. 
 
\begin{defi}
Let $\DG$ be a crossed simplicial group. A {\em $\DG$-structured set} is a pair $(I, \Oc(I))$, where:
\begin{enumerate}
\item[(1)] $I$ is  a nonempty finite set of some cardinality $n+1$, $n\geq 0$. 
\item[(2)] $\Oc(I)$ is a right $\Gen_n$-torsor together with a map
\[
\rho: \Oc(I) \lra \Isom_{\FSet}\bigl( \{0,1,\dots, n\}, I \bigr)
\]
which is equivariant with respect to the homomorphism $\lambda_n: \Gen_n\to S_{n+1}$
from Proposition \ref {prop:forget}.
\end{enumerate}
\end{defi}

Thus $\rho$ is a {\em reduction of structure group} inducing an isomorphism of $S_{n+1}$-torsors
\[
	\O(I)\times_{\Gen_n} S_{n+1} \overset{\cong}{\lra} \Isom_{\FSet}\bigl( \{0,1,\dots, n\}, I\bigr).
\]
The datum of $\rho$ will be referred to as a $\DG$-{\em structure}, or a $\DG$-{\em order}, on the set $I$.
Elements of $\Oc(I)$ will be called {\em structured frames} of $I$. 
 
\begin{defi}\label{defi:dgstructure}
A morphism of $\DG$-structured sets 
\[
\psi: (I', \Oc(I'))\lra (I, \Oc(I)), \quad |I'|=n'+1, \, |I|=n+1, 
\]
is a datum of morphisms
\[
	\left\{ \psi_{f,f'}\in\Hom_\DG([n'], [n]) \; \big| \; \text{$f'\in \Oc(I')$,
	$f\in\Oc(I)$}\right\},
\]
satisfying, for every $g \in \Gen_{n}$, $g' \in \Gen_{n'}$, the equivariance condition
\[
	\psi_{fg, f'g'} = g^{-1} \circ \psi_{f, f'}\circ g'.
\]
The composition of two morphisms
\[
	(I'', \Oc(I'')) \overset{\psi'}{\lra}  (I', \Oc(I')) \overset{\psi}{\lra} (I, \Oc(I))
\]
is defined by the formula
\[
(\psi\psi')_{f, f''} = \psi_{f,f'} \circ \psi'_{f', f''}
\]
where the right hand is independent on $f'$. 
\end{defi}
 
We denote by $\Gc$ the category of $\DG$-structured sets thus defined.  The following is then
immediate.
 
\begin{prop}
\begin{itemize}
\item[(a)] We have an equivalence of categories $\epsilon:\DG\to\Gc$ sending an object $[n]$
to the $\DG$-structured set $(\{0,1,\dots, n\}, \Gen_n)$.

\item[(b)] For any object $(I, \Oc(I))\in\Gc$, we have a canonical identification of $\Gen_n$-torsors
\[
	\Oc(I) \overset{\cong}{\lra} \Isom_\Gc(\epsilon [n], (I, \Oc(I)))
\]
and of sets
\[
	I \overset{\cong}{\lra} \Hom_\Gc(\epsilon[0], (I, \Oc(I))/\Gen_0. 
\]
\item[(c)]  In particular, the correspondence $(I, \Oc(I))\mapsto I$ gives a functor
	$ \lambda_{\Gc}:\Gc\to  \FSet$ extending $\lambda:\DG\to\bf N$. \qed
\end{itemize}
\end{prop}

In classical combinatorics, it is often important to speak about different orders of a specific type on
a given set, without identifying isomorphic ones. For example, an $N$-element set has $N!$ different
linear orders, all isomorphic to one another. We introduce the corresponding concept for
$\DG$-orders.

\begin{defi}
Let $I$ be a nonempty finite set. The category $\Gc/I$ of $\DG$-orders on $I$ has objects given by all
$\DG$-structures on $I$, and morphisms given by those morphisms $\psi: (I,\Oc)\to (I, \Oc')$
in $\Gc$ which induce the identity map on the set $I$.
\end{defi}

Clearly, $\Gc/I$ is a groupoid. If $|I|=n+1$ and $|\Gen_n|<\infty$, then the number of isomorphism
classes of objects of $\Gc/I$ is
\[
|\pi_0(\Gc/I)| = \frac{(n+1)!}{|\Gen_n|} \cdot |\Gen_n^0|, \quad \text{where
$\Gen_n^0=\Ker(\lambda_n)$.} 
\]
This number can be understood as ``the number of different $\DG$-orders on $I$". Each such order,
represented by an object of $\Gc/I$, has automorphism group isomorphic to $\Gen_n^0$. 

\begin{exa} Consider the crossed simplicial group $\DG = \Xi$ given by the dihedral category.  We
	have $\Gen_0 = \Gen_0^0 = \Zt$ so that, for a singleton set $I$, the category $\Gc/I$ has a
	single isomorphism class of objects and each object has automorphism group $\Zt$.  We have
	$\Gen_1 = \Zt \times \Zt$ and $\Gen_1^0 = \Zt$ so that, on a given set $I$ of cardinality
	$2$, we have one isomorphism class of objects in $\Gc/I$, each object having automorphism
	group $\Zt$. For $|I| = n+1$ with $n \ge 2$, the objects of $\Gc/I$ have trivial
	automorphism group, and there are $\frac{n!}{2}$ isomorphism classes.
\end{exa}

Note that Definition \ref{defi:dgstructure} does not use any special features of $\DG$ except the functor
$\lambda:\DG \to \bf N$. Using the unique factorization property of crossed simplicial groups, we can
give the following reformulation. 

\begin{prop}\label{prop:morphisms-Gc}
The set $\Hom_\Gc((I', \Oc(I')), (I, \Oc(I)))$ is in bijection with the set of pairs
$(\psi_*, \psi^*)$, where $\psi_*: I'\to I$ and $\psi^*: \Oc(I)\to\Oc(I')$ are maps of sets such that:
\begin{itemize}
\item[(i)] For every structured frame $f$ of $I$, the unique map $f^*\psi: \{0,1,\dots, n'\}\to \{0,1,\dots, n\}$ making the
diagram
\[
\xymatrix{
\{0,1,\dots,n'\} 
\ar[d]_{\rho(\psi^*f)}\ar[r]^{f^*\psi }& \{0,1,\dots,n\} 
\ar[d]^{\rho(f)}
\\
I'\ar[r]_{\psi_*}&I
}
\]
commutative is given by a morphism in $\Delta$. 

\item[(ii)] For each frame $f$ of $I$ and $g\in\Gen_n$, we have
\[
\psi^*(f .  g) = \psi^*(f) . (f^*\psi)^*(g), 
\]
where $(f^*\psi)^*: \Gen_n\to\Gen_{n'}$ is the map associated to $f^*\psi$ via the simplicial set
$\Gen$. 
\end{itemize}
\end{prop}

Note that the diagram in (i) can be seen as an analog of the diagram
\eqref{eq:can-square} describing the canonical factorization of morphisms in $\DG$.
Condition (ii) mimics the first identity of crossed simplicial groups in
Proposition \ref{prop:skeleton}(c). 

\begin{proof}[Proof of Proposition \ref{prop:morphisms-Gc}:]
Let $\psi=(\psi_{f,f'})$ be a morphism from $(I', \Oc(I'))$
to $(I, \Oc(I))$ in $\Gc$. The map $\psi_*: I'\to I$ is defined to be the value, on $\psi$,
of the functor $\lambda_{\Gc}$. To define the map $\psi^*: \Oc(I)\to\Oc(I')$ associated to $\psi$,
we fix $f\in\Oc(I)$, choose and arbitrary $f'\in\Oc(I')$ and interpret them as vertical
isomorphisms in the diagram
\[
\xymatrix{
(I', \Oc(I')) \ar[rr]^\psi && (I, \Oc(I))
\\
\epsilon[n']
\ar[u]^{f'} 
 \ar[rr]^{\psi_{f, f'} = f^{-1}\circ\psi\circ f' }&& \epsilon[n]
 \ar[u]_f
 \\\epsilon[n']\ar[u]^{g'} 
 \ar[urr]_\phi
}
\]
Next, we define $\phi$ and $g'$ in the diagram via the canonical factorization
$\psi_{f, f'} = \phi\circ g'$ with $\phi\in\Delta$, $g'\in\Gen_{n'}$, and set $\psi^*(f)=f'\circ g'$. The result is independent
of $f'$ in virtue of the equivariance condition on the datum $\psi=(\psi_{f, f'})$. 
This defines the pair $(\psi_*, \psi^*)$ associated to $\psi$. We leave the remaining verifications
to the reader. 
\end{proof}

\begin{exa}\label{ex:DG-closure} A linear order $\leq$ on a nonempty finite set $I$ defines a
	canonical $\DG$-order on $I$ called the $\DG$-{\em closure} of $\leq$: Note
	that $\leq$ can be identified with the unique monotone bijection $\sigma:
	\{0,1,\dots, n\}\to I$, where $|I|= n+1$. We then define the $\Gen_n$-torsor $\Oc_{\sigma}$ to
	be $\Gen_n$, and the map $\rho$ to be the composition
	\[
	\Gen_n\buildrel\lambda_n\over\lra S_{n+1}= \Isom_{\FSet}\bigl(
	\{0,1,\dots, n\}, \{0,1,\dots, n\}\bigr)\buildrel\sigma_*\over\lra 
	 \Isom_{\FSet}\bigl(
	\{0,1,\dots, n\},  I\bigr). 
	\]
	Denoting by ${\bf\Delta}$ the category of all nonempty finite ordinals, we can
	view the $\DG$-closure as a functor ${\bf i}: {\bf\Delta} \to\Gc$ extending
	$i:\Delta \to \DG$. 
\end{exa}

\begin{prop}\label{prop:model-G/I}
\begin{itemize}
\item[(a)]
For linear orders on $I$ given by bijections $\sigma_1, \sigma_2$, we have
\[
\Hom_{\Gc/I}((I, \Oc_{\sigma_1}), (I, \Oc_{\sigma_2})) \cong \bigl\{ g\in\Gen_n \bigl| \,\, \lambda_n(g)
= \sigma_1^{-1}\sigma_2\bigr\},
\]
and the composition of morphisms corresponds to the multiplication in $\Gen_n$. 

\item[(b)] Every object $(I, \Oc(I))$ of $\Gc/I$ is isomorphic  to an object of the form $(I,\Oc_\sigma)$,
so the category with objects  $(I,\Oc_\sigma)$ and morphisms defined by the right hand side of the identification
in (a), is equivalent to $\Gc/I$. 
\end{itemize}
\end{prop}

\noindent {\sl Proof:} (a) is straightforward from the definitions. To prove (b), note that any $f\in \Oc(I)$
gives rise to an isomorphism $\psi^f: [n]\to (I, \Oc(I))$ in $\Gc$ which gives a bijection
$\sigma=\lambda(\psi^f)$ from $\{0,1,\dots, n\}$ to $I$. We then see easily that $(I, \Oc(I))$
is isomorphic to $(I, \Oc_\sigma)$ by an isomorphism identical on $I$. \qed

\vskip .2cm

Note that a $\Delta$-structured set is a finite nonempty linearly ordered set. Any subset of a
linearly ordered set admits a canonical induced order. The following proposition shows that this
fundamental property generalizes to $\DG$-structured sets. 

\begin{prop}
Let $u: I'\hookrightarrow I$ be an embedding of nonempty finite sets. Then there is a functor
\[
\Res^I_{I'}: \Gc/I \lra\Gc/I'
\]
called {\em restriction of $\DG$-structures from $I$ to $I'$}, with the following properties:
\begin{itemize}
\item[(a)] For each $\DG$-structure $\Oc(I)$ on $I$ the map $u$    extends to a canonical
morphism in $\Gc$
\[
\psi^u: \Res^I_{I'}(I, \Oc(I))\lra (I, \Oc(I)).
\]
\item[(b)] For a triple of composable embeddings $I''\hookrightarrow I'\hookrightarrow I$ we have a canonical
isomorphism
\[
\Res^I_{I''} \,\,\cong \,\, \Res^{I'}_{I''}\circ \Res^I_{I'}. 
\]
\end{itemize}
\end{prop}

\begin{proof} We use the model of $\Gc/I$ from Proposition \ref{prop:model-G/I}. Given a total order $\leq $ on $I$,
it restricts to a total order on $I'$, and so we have a diagram
\be\label{eq:comm-diagr-emb}
\xymatrix{
\{0,1,\dots, n'\} 
\ar[d]_{\sigma'}
\ar[r]^\phi& \{0,1, \dots, n\}\ar[d]^\sigma
\\
I'\ar[r]_u& I,
}
\ee
in which $\psi, u$ are monotone embeddings, and $\sigma,\sigma'$ are monotone bijections. We define on objects
$\Res^I_{I'}(I,\Oc_\sigma)=(I', \Oc_{\sigma'})$. 

Further, let $g\in\Gen_n$ be such that $\lambda_n(g)=\sigma_1^{-1}\sigma_2$, where 
$\sigma_1, \sigma_2: \{0,1, \dots, n\}\to I$ are the bijections corresponding to two total orders
$\leq^1, \leq^2$ on $I$. Let $\sigma'_1, \sigma'_2: \{0,1,\dots, n'\}\to I'$ be the
bijections corresponding to the restrictions of $\leq^1, \leq^2$ to $I'$. We the have two commutative
diagrams as in \eqref{eq:comm-diagr-emb}, whose arrows we denote $u, \sigma'_\nu, \sigma_\nu, \phi_\nu$,
$\nu=1,2$. 

We claim, first of all, that $\phi_2=g^*\phi_1$. Indeed, consider the canonical diagram in $\DG$
\[
\xymatrix{
[n']\ar[d]_{\phi_1^*g}
\ar[r]^{g^*\phi_1}&[n]
\ar[d]^g
\\
[n']\ar[r]_{\phi_1}& [n]
}
\]
defining $g^*\phi_1$ and apply the functor $\lambda$ to it. Because of the assumptions on $g$, the resulting diagram
in $\FSet$ is 
\be\label{eq:2.1.11}
\xymatrix{
\{0,1,\dots, n'\}
\ar[d]_{\lambda_n(\phi_1^*g)}
\ar[r]^{\phi_2}& \{0,1,\dots, n\}
\ar[d]^{\lambda_n(g)=\sigma_1^{-1}\sigma^2}
\\
\{0,1,\dots, n'\}\ar[r]_{\phi_1}& \{0,1,\dots, n\},
}
\ee
whence $\phi_2=g^*\phi_1$.

We further claim that
$\lambda_n(\phi_1^*(g)) = (\sigma'_1)^{-1}\sigma'_2$. Indeed, in virtue of the injectivity of
$\phi_1, \phi_2$ and bijectivity of $\lambda_n(g)$, there can be at most one map making 
\eqref{eq:2.1.11} commutative, and both sides of the proposed equality do.
So we define $\Res^I_{I'}$ on morphisms by sending $g$ to $\phi_1^*(g)$.
To see that $\Res^I_{I'}$ commutes with composition of morphisms, we consider three orders
$\leq^\nu$ on $I$, $\nu=1,2,3$ with three corresponding bijections $\sigma_\nu$ and monotone maps
$\phi_\nu$. If now
$g_1=\sigma_1^{-1}\sigma_2$, $g_2=\sigma_2^{-1}\sigma_3$, then
$\Res^I_{I'}(g_1g_2)  = \phi_1^{-1}(g_1g_2)$, while
\[
\Res^I_{I'}(g_1) \circ\Res^I_{I'}(g_2) \,\,=\,\,\phi_1^*(g_1)\phi_2^*(g_2) \,\,=\,\,
\phi_1^*(g_1) \circ ((g_1^*\phi_1)^*(g_2)),
\]
which equals $\phi^*(g_1g_2)$ by the first equality in Proposition \ref{prop:skeleton}(c). We leave the
remaining verifications to the reader. 
\end{proof}

\begin{rem}
Given a structure torsor 
\[
\Oc(I)\buildrel \rho\over\lra\Isom_{\FSet}(\{0,1,\dots, n\}, I)
\]
for $I$, we can directly construct the ``restricted torsor" $\Oc(I')$ for $\Res^I_{I'}(I, \Oc(I))$ and the canonical
morphism
\[
\psi=(\psi_*=u, \psi^*): (I', \Oc(I'))\to (I, \Oc(I))
\]
as follows.
Choose some $f\in\Oc(I)$, so that $\sigma=\lambda_{\Gc}(f)$ is a bijection defining a total order on $I$.
Let $\phi: [n']\to[n]$ be the monotone map in \eqref{eq:comm-diagr-emb}. To emphasize the dependence of $\phi$ on
$f$, we denote it as $\phi=f^*\psi$ (where $\psi$ stands for the morphism we are constructing). 
For $g,h\in\Gen_n$ put $f.g\sim f.h$ whenever
$(f^*u)^*(g)=(f^*u)^*(h)$ in $\Gen_{n'}$. It is immediate that this is an equivalence relation, independent on
the choice of $f$. We put $\Oc(I')=\Oc(I)/\sim$ and define $\psi^*: \O(I) \to \O(I')$ to be the quotient map. The composite
	\[
		\O(I) \lra \Isom_{\FSet}(\{0,1,\dots, n\}, I) \lra \Isom_
		{\FSet}(\{0,1,\dots, n'\}, I) 
	\]
	factors through $\psi^*$ providing a map $\rho_{I'}: \O(I') \to \Isom(\{0,1,\dots, n'\}, I)$.
	There exists a unique  $\Gen_{n'}$-action on $\O(I')$ such that, for every $f \in
	\O(I)$ and $g \in \Gen_n$, we have 
	\[
			\psi^*(f.g) = \psi^*(f).((f^*\psi)^*g).
	\]
	 It is now readily verified that this $\Gen_m$-action on $\O(I)$ is simply
	transitive,  and the maps $\psi_*=u$ and $\psi^*$ satisfy the conditions (i) and (ii)
	of Proposition \ref{prop:morphisms-Gc}. 
\end{rem}

\subsubsection{Topological model}
\label{sec:topmodel}

Let $\DG$ be a planar crossed simplicial group with corresponding planar Lie group $p: G \to O(2)$.
In this section, we define a topological model for $\DG$ based on the planar
homeomorphism group
\[
p_\Homeo: \Homeo^G(S^1)\lra\Homeo(S^1)
\]
corresponding to $G$ as introduced in \S \ref{subsub:plan-lie}

By a {\em circle} we mean a topological space homeomorphic to the standard circle $S^1$. By a {\em
marked circle} we mean a pair $(C,J)$ where $C$ is a circle and $J\subset C$ is a nonempty closed
subset homeomorphic to a finite disjoint union of copies of the interval $[0,1]$. For two marked
circles $(C,J)$ and $(C',J')$, we denote
\[
\Homeo((C,J), (C',J')) \,\,=\,\,\bigl\{ \phi\in\Homeo(C,C') \bigl| \,\, \phi(J)\subset J'
\bigr\}. 
\]
It has a natural topology as a closed subspace of $\Homeo(C,C')$. 

\begin{prop}\label{prop:homeo-contr}
Each connected component of $\Homeo((C,J), (C',J'))$ is contractible. 
\end{prop}

\begin{proof} Follows from the fact that the group of orientation preserving self-homeomorphisms
of an interval is contractible.
\end{proof}

By a $G$-{\em structured circle}\footnote{A more precise term would be ``$\Homeo^G(S^1)$-structured circle".}
we mean a pair $(C,\rho)$ where $C$ is a circle, and $\rho: F\to\Homeo(S^1, C)$ is a reduction of the structure
group along $p_\Homeo$. Thus, $F$ is a $\Homeo^G(S^1)$-torsor and $\rho$ is $p_\Homeo$-equivariant.
We will often denote a $G$-structured circle simply by $C$, assuming $\rho$ to be given. 

For $G$-structured circles $(C,\rho)$ and $(C',\rho')$, we define $\Homeo^G((C,\rho), (C',\rho'))$
to be the set of pairs $(\phi, \widetilde\phi)$, where $\phi\in\Homeo(C,C')$ and $\widetilde\phi: F\to F'$
is a $\Homeo^G(S^1)$-equivariant homeomorphism such that the diagram
\[
\xymatrix{
F\ar[d]_\rho
\ar[r]^{\widetilde\phi}& F'\ar[d]^{\rho'}
\\
\Homeo(S^1, C) \ar[r]_\phi & \Homeo(S^1, C')
}
\]
is commutative. Note that the set $\Homeo^G((C,\rho), (C',\rho'))$ has a natural topology with respect to which it is
homeomorphic to $\Homeo^G(S^1)$. 

We will consider circles which are both structured and marked, denoting them by $(C,J)$ where
$C=(C,\rho)$ is a $G$-structured circle. For $G$-structured marked circles $(C, J)$ and
$(C',J')$, we denote by $\Homeo^G((C,J), (C',J '))$ the closed subspace in $\Homeo^G(C,C')$ formed by
$(\phi,\widetilde\phi)$ such that $\phi\in \Homeo((C,J), (C',J '))$. 
 
\begin{defi}
We define the category $\CG$ with objects given by $G$-structured marked circles and morphisms 
\[
\Hom_{\CG}((C,J), (C',J')) = \pi_0\, \Homeo^G((C,J), (C',J ')). 
\]
\end{defi}

We have the main result of this section: 
 
\begin{thm}\label{thm:CG=DG}
	Let $p: G \to O(2)$ be a planar Lie group with corresponding crossed simplicial group $\DG$. Then
	the functor
	\[
		\lambda_{\CG}: \CG \lra \FSet, \quad (C,J)\mapsto \pi_0(J)
	\]
	canonically factors as in 
	\[
		\xymatrix{
			\CG \ar[dr]_{\lambda_{\CG}} \ar[rr]^{\pi}_{\simeq} & & \G \ar[dl]^{\lambda_{\G}} \\
		& \FSet & }
	\]
	where $\G$ denotes the category of $\DG$-structured sets and $\pi: \CG \to \G$ is an equivalence
	of categories. In particular, the category $\CG$ is equivalent to $\DG$. 
\end{thm}
\begin{proof} 
	Let $(C,J)$ be a marked $G$-structured circle with reduction $\rho: F \to \Homeo(S^1, C)$ of structure
	group along $p_\Homeo$. We claim that the association
	\[
		(C,J) \mapsto (\pi_0(J), \rho_{\pi_0(J)}),
	\] 
	where $(\pi_0(J), \rho_{\pi_0(J)})$ denotes the $\DG$-structured set constructed in Lemma
	\ref{lem:pi0} below, extends to an equivalence of categories $\pi: \CG \to \DG$.
	Given a morphism $(\varphi, \widetilde{\varphi}): (C,J) \to (C', J')$ of marked
	$G$-structured circles, we have to construct a morphism of the $\DG$-structured sets
	$(\pi_0(J), \rho_{\pi_0(J)})$ and $(\pi_0(J'),\rho_{\pi_0(J')})$. Clearly, we have an
	induced map $\psi: \pi_0(J) \to \pi_0(J')$ of underlying sets. We now define the pullback
	map $\psi^*: \pi_0(F_{J'}) \to \pi_0(F_{J})$ of structured frames. Let $f \in F_{J'}$
	represent a structured frame in $\pi_0(F_{J'})$. Its image under $\rho$ defines a
	homeomorphism  $\rho(f): \Homeo(S^1, C')$ which maps $[n]$ to $J'$. We choose a point in the
	unique component of $C'\setminus J'$ bounded by $\rho(f)(n)$ and $\rho(f)(0)$, and denote
	its preimage in $C$ under $\varphi$ by $p$. Note that $p$ lies in $C \setminus J$.  Consider
	the homeomorphism $\varphi^{-1} \circ \rho(f)$ in $\Homeo(S^1, C)$. Now we choose a path
	$\alpha: [0,1] \to \Homeo(S^1, C)$ starting at $\varphi^{-1} \circ \rho(f)$, ending at a
	homeomorphism which maps $[m]$ to $J$, such that, for all $t \in [0,1]$, $\alpha(t)(0) \ne
	p$. It is easy to see that the space of such paths is contractible. Since $F \to \Homeo(S^1,
	C)$ is a covering, there is a unique lift of $\alpha$ to $F$ satisfying $\alpha(0) =
	\widetilde{\varphi}$. We define $\psi^*(f)$ to be the connected component of $\alpha(1)$ in
	the space $F_J$.  The pair $(\psi_*, \psi^*)$ defines a morphism in $\Gc$ by Proposition
	\ref{prop:morphisms-Gc}. It straightforward to show that this construction is functorial. 
	To see that it provides an equivalence of categories it suffices to verify the unique
	factorization property for the full subcategory of $\CG$ generated by a set of standard
	objects $\{ (S^1, [n]) \}$ which is easily done.
\end{proof}

\begin{lem}\label{lem:pi0} Let $n \ge 0$ and consider the standard circle $S^1$ equipped with the standard subset
	$[n] \subset S^1$ of $(n+1)$st roots of unity.
	\begin{enumerate}
		\item Consider the subgroup $\Homeo(S^1, [n]) \subset \Homeo(S^1)$ of homeomorphisms
			preserving which preserve the standard subset $[n]$ of $(n+1)$st roots of
			unity. Let $H_n$ be the pullback 
			\[ 
				\xymatrix{ H_n\ar[d] \ar[r]&\Homeo^G(S^1) \ar[d]^{p_\Homeo} \\ 
				\Homeo(S^1, [n]) \ar[r]&\Homeo(S^1).  } 
			\] 
			Then each connected component of $H_n$ is contractible
			and there is a canonical group isomorphism $\pi_0(H_n) \cong \Gen_0$. 
		\item Let $(C,J)$ be a $G$-structured marked circle with reduction $\rho: F \to \Homeo(S^1, C)$ of structure
			group along $ p_\Homeo$. Consider the subset $\Homeo( (S^1, [n]), (C, J)) \subset \Homeo(S^1, C)$ 
			consisting of homeomorphisms which map $[n]$ into $J$ and let 
			$\rho_J: F_J \to \Homeo( (S^1, [n]), (C, J))$ denote the restriction of
			$\rho$. Then the set $\pi_0(F_J)$ is a $\pi_0(H_n)$-torsor, equipped with a
			natural equivariant map 
			\[
				\rho_{\pi_0(J)}: \pi_0(F_J) \to \pi_0(J)
			\]
			making $(\pi_0(J), \rho_{\pi_0(J)})$ a $\DG$-structured set.
	\end{enumerate}
\end{lem}
\begin{proof} (1) As argued in Proposition \ref{prop:homeo-contr}, each component of the
 subgroup $\Homeo(S^1, [n])$ is contractible. Therefore, the map $p_\Homeo$ restricts to a covering 
 $H_n \to \Homeo(S^1,[n])$ which is topologically trivial so that each connected component of $H_n$
 must be contracible. We have
 \[
	 \pi_0 \Homeo(S^1, [n]) \cong D_{n+1} \cong \Aut_\Xi([n])
 \]
 where $\Xi$ denotes the dihedral category of Example \ref {ex:dihedral}. The group $\pi_0(H_n)$ can
 be identified with the preimage $p^{-1}(D_{n+1})$ under $p: G\to O(2)$. By Theorem \ref{thm:crossed-groups}(b2), 
 this preimage is canonically identified with $\Gen_n$. 

 \noindent
 (2) Passing to connected components, we obtain a sequence of maps
	\[
			\pi_0(F_J) \overset{\pi_0(\rho_J)}{\lra} \pi_0(\Homeo( (S^1, [n]), (C, J)))
			\lra \Isom_{\FSet}([n],
			\pi_0(J))
	\]
	whose composite is defined to be $\rho_{\pi_0(J)}$. It is apparent that $(\pi_0(J),
	\rho_{\pi_0(J)})$ is a $\DG$-structured set.
\end{proof}

\begin{cor}\label{cor:structured-circle-orders}
Let $C$ be a $G$-structured circle. Then, for any nonempty finite subset $I\subset C$, both $I$ and
$\pi_0(C-I)$ acquire canonical $\DG$-structures. These structures are natural with respect to
homeomorphisms of $G$-structured circles. 
\end{cor}

The following fact generalizes the result of Drinfeld \cite[\S 3]{drinfeld:cyclic}. 

\begin{prop}
	 Let $R:\Set_\DG\to\Top$ be the geometric realization functor: $R(X)=|i^*X|$, where $i: \Delta\to\DG$
	 is the embedding. The group $\Homeo^G(S^1)$ acts on $R$ by natural isomoprhisms. In particular, it
	 acts on each $|i^*X|$ be homeomorphisms. 
\end{prop}
 
 It seems plausible that the group $\Homeo^G(S^1)$ is in fact identified with $\Aut(R)$
 which would provide its natural construction in terms of $\DG$. 
 
 \vskip .2cm
 
 \noindent {\sl Proof:} The argument is similar to that of \cite{drinfeld:cyclic}. 
 That is, let $\bf\Delta$ be the category of all nonempty finite ordinals, and ${\bf i}: {\bf\Delta}\to \Gc$
 be the functor of $\DG$-closure, extending $i$ (Example \ref{ex:DG-closure}). A simplicial set $Y$ can
 be seen as a contravariant functor $Y: {\bf\Delta}\to\Set$ and, as pointed out in Formula (1.1) of
 \cite {drinfeld:cyclic},
 \[
 |Y|\,\,=\,\,\varinjlim\limits_{F\subset [0,1]} \, Y\bigl( \pi_0([0,1]-F)\bigr), 
 \]
 where $F$ runs over finite subsets of $[0,1]$, and $\pi_0([0,1]-F)$ is equipped by the total order
 induced by that on $[0,1]$. 
 
 Now, if $X$ is a $\DG$-set, we can view $X$ as a contravariant functor $\Gc\to\Set$ and identify
 \[
 \begin{gathered}
 |i^*X| \,\,=\,\, \varinjlim\limits_{F\subset [0,1]} \, ({\bf i}^*X)\bigl( \pi_0([0,1]-F)\bigr) \,\,=\,\,
 \varinjlim\limits_{F\subset [0,1]} \, X\bigl({\bf i} ( \pi_0([0,1]-F))\bigr) \,\,= \\
=\,\,   \varinjlim\limits_{1\in F\subset S^1} \, X\bigl({\bf i} ( \pi_0(S^1-F))\bigr) \,\,=\,\,
  \,\,  \varinjlim\limits_{F\subset S^1} \, X\bigl({\bf i} ( \pi_0(S^1-F))\bigr).
  \end{gathered}
 \]
 Here we view $S^1$ as obtained from $[0,1]$ by identifying $0$ and $1$ into one point $1\in S^1$ and as equipped
 with the standard $\Homeo^G(S^1)$-structure. Each set $\pi_0(S^1-F)$ is then equipped with a $\DG$-structure
 by Corollary \ref{cor:structured-circle-orders}. Now, the last colimit is manifestly acted upon by
 $\Homeo^G(S^1)$. \qed

\subsubsection {Interstice duality}
 
The topological model $\CG$ allows for an immediate and unified proof of the following result which is
well-known for cyclic (cf. \cite{connes, drinfeld:cyclic}), dihedral, and quaternionic categories (cf. \cite{dunn}).

\begin{cor}\label{cor:duality}
Any planar crossed simplicial group is self-dual.
\end{cor}
\begin{proof}  We define the  functor $(-)^{\vee}: \CG \to
	\CG^{\op}$  by putting
		\[
		(C,J)^{\vee} := (C, \overline{C \setminus J}), \quad 
 (\varphi, \widetilde{\varphi})^{\vee} := (\varphi^{-1}, \widetilde{\varphi}^{-1})
	\]
	on objects and morphisms, respectively. 
	Since the square of this functor equals the identity functor, it must be an equivalence.
	Further, since $\DG$ is equivalent to $\CG$, we conclude that there exists a duality functor $(-)^\vee_\DG:
	\DG \overset{\simeq}{\to} \DG^\op$ with $[n]^\vee=[n]$. The choice of values for this functor on morphisms, however, is
	not canonical. 
\end{proof}

The connected components of $\overline{C \setminus J}$ can be regarded as ``interstices" positioned
between adjacent connected components of $J$. 

\begin{rem}\label{rems:duality}
After choosing a duality functor $(-)^\vee_\DG$, we obtain identifications of sets
\[
\begin{gathered}
\Gen_n \,\, \buildrel \on{C.F.} \over = \,\, \Hom_\DG([n], [0]) \,\,\simeq \,\, \Hom_\DG([0]^\vee, [n]^\vee)\,\,\simeq \\
\simeq \,\, \Hom_\DG([0], [n]) \,\, \buildrel \on{C.F.} \over = \,\, \Gen_0 \times \Hom_\Delta([0], [n]),
\end{gathered}
\]
where ``C.F." stands for the canonical factorization of morphisms in a crossed simplicial group.
These identifications can be interpreted by saying that $\Gen_n$
``grows linearly with $n$" (with ``coefficient of proportionality" being $\Gen_0$). Cf. the discussion
of three growth types of crossed simplicial groups in \S \ref{sub:weyl}.
We further deduce that the $\Gen_n$-action on $\Hom_\DG([0], [n])$ is simply transitive. 
\end{rem}

\begin{rem} One can also give a purely combinatorial description of duality in terms of the
category $\Gc$ of $\DG$-structured sets. Note that for any object $(I, \Oc(I))\in\DG$ we
have canonical identifications
\[
I \,\,=\,\, \Hom_\Gc([0], I)/\Gen_0, \quad \Oc(I) \,\,=\,\, \Hom_\Gc(I, [0]). 
\]
The dual of $(I, \O(I))$ is the pair
\[
	(I^{\vee}, \O(I^{\vee})) := (\Hom_{\G}(I, [0])/\Gen_0, \Hom_{\G}([0],I))
\]
which can be made into a $\DG$-structured set so that $I \mapsto I^{\vee}$ defines an equivalence
$\G \to \G^{\op}$.
\end{rem}
	
In the subsequent sections, we explicitly analyze the category $\G$ of $\DG$-structured sets for
various examples of planar crossed simplicial groups.

\subsection{Cyclic orders}\label{sec:cycorder}

Let $\DG = \Lambda$ be Connes' cyclic category. A $\Lambda$-structure on a set $J$ can then be
interpreted as a {\em cyclic order}. 

\subsubsection{Finite total cyclic orders}

\begin{defi}\label{def:totalcyclic} A {\em (total) cyclic order} on a finite set $J$ of cardinality $n$ is 
the choice of a simply transitive action of the group $\ZZ/n\ZZ$. Note that a cyclic order on the set $J$ 
is uniquely determined by specifying for every element $j$ its {\em successor} $j + 1$.
\end{defi}

To avoid confusion, we refer to the usual concept of total order as {\em linear order}.

\begin{exas}
	\begin{enumerate} 
		\item Any set of cardinality $\le 2$ has a unique cyclic order. A cyclic order on a
			set $J$ of cardinality $3$ is the same as an orientation of the triangle
			$|\Delta^J|$.
		\item Every subset of a cyclically ordered set $J$ has a canonical cyclic order.
		\item Any finite linearly ordered set $J$ carries a canonical cyclic order where the
			successor of each nonmaximal element of $J$ is its successor with respect to the linear
			order and the minimal element is the successor of the maximal element. We
			say that this cyclic order is the {\em cyclic closure} of the linear order on $J$.
		\item More generally, given a cyclically ordered finite set $K$ and a map of finite sets $f: J \to
			K$ where each fiber of $f$ is equipped with a linear order, we can define
			the {\em lexicographic} cyclic order on $J$ as follows. First note, that the
			set of fibers of $f$ is canonically identified with $K$ and hence carries a
			cyclic order. Therefore, every fiber has a successor fiber. Now the cyclic
			order on $J$ is given as follows: For every element $j$ which is nonmaximal
			in its fiber, we define the successor $j +1$ to be the
			successor with respect to the specified linear order of the fiber. If $j$ is
			maximal in its fiber, we define its successor to be the minimal element of
			the successor fiber. Here the convention is to skip all empty fibers. The
			previous example is then given by the case when $K$ is a point, equipped
			with its unique cyclic order.
	\end{enumerate}
\end{exas}

\begin{defi} A morphism $f: J \to K$ of finite cyclically ordered sets is a map $f: J \to K$ of underlying sets
	equipped with the choice of a linear order on every fiber such that the cyclic order on $J$
	is the induced lexicographic order. The resulting category ${\bf \Lambda}$ 
	given by finite nonempty cyclically ordered sets and their morphisms is called the {\em large cyclic
	category.}
\end{defi}

\begin{rems} (a) Restricting to the standard cyclic ordinals $\cn$, this gives a purely  combinatorial description
of morphisms in Connes' category $\Lambda$. 

(b) We have a natural faithful embedding ${\bf \Delta} \subset {\bf \Lambda}$ given by
	associating to a linearly ordered set $J$ its cyclic closure. This extends the embedding $\Delta\subset\Lambda$. 
\end{rems}

\subsubsection{General cyclic orders}
While the concept of cyclic order introduced above is sufficient for the main purposes of this work, it
is not fully satisfactory in that it does not cover the prototypical example, generating all
the finite cyclic orders:   {\em  the circle $S^1$ itself}.
Indeed, $S^1$ is an infinite set. 
The most flexible definition of a linear order is given by a
binary relation on a set. It is therefore natural to define cyclic orders in a similar style.
The following definition is a slight modification of those
given by Huntington and Nov\'ak \cite{huntington, novak}.

\begin{defi}\label{def:cyclic-order}
Let $J$ be a set. A {\em partial cyclic order} on $J$ is a ternary relation $\lambda \subset J^3$
with the following properties:
\begin{enumerate}[label=(C\arabic{*})]
\item (reflexivity) We have $(a,a,a)\in\lambda$ for any $a\in J$. 

\item (antisymmetry) If  $(a,b,c)\in\lambda$ and $(b,a,c)\in\lambda$, then
$|\{a,b,c\}|\leq 2$. 
 
 \item (transitivity)  Let $a,b,c,d\in J$ be distinct.
 If $(a,b,c)\in\lambda$ and $(a,c,d)\in\lambda$, then $(a,b,d)\in\lambda$
 and $(b,c,d)\in\lambda$: 
  
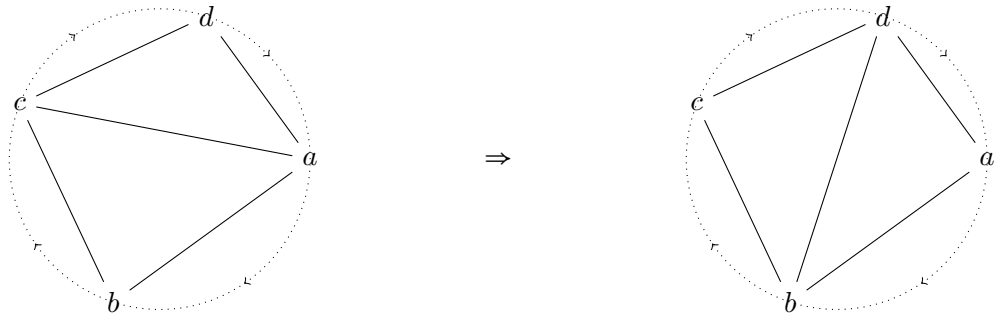
\[
\begin{tikzpicture}

 \node (O) at (0,0) (origin) {}; 
\node (Pa) at (0:2cm)   { $a$}; 
 \node (Pb) at (3.5*72:2cm)  {$b$}; 
 \node (Pc) at (2.2*72:2cm)  {$c$}; 
\node (Pd) at (1*72:2cm)  {$d$};

\draw[ 
        decoration={markings, mark=at position 0.125 with {\arrow{<}},
        mark=at position 0.35 with {\arrow{<}},
        mark=at position 0.6 with {\arrow{<}},
        mark=at position 0.85 with {\arrow{<}}
        },
        postaction={decorate}, dotted
        ]
 (origin) circle (2cm);

 \draw (node cs:name=Pa) --(node cs:name=Pb);
  \draw (node cs:name=Pb) --(node cs:name=Pc);
   \draw (node cs:name=Pa) --(node cs:name=Pc);
 \draw (node cs:name=Pa) --(node cs:name=Pd);
  \draw (node cs:name=Pc) --(node cs:name=Pd);

\node(A) at (4.5,0) (arrow) {$\Rightarrow$} ;

\node (O') at (9,0) (origin') {}; 
\node [xshift=9cm] (Pa') at (0:2cm)  { $a$}; 
 \node  [xshift=9cm] (Pb') at (3.5*72:2cm) {$b$}; 
 \node [xshift=9cm] (Pc') at (2.2*72:2cm) {$c$}; 
\node [xshift=9cm] (Pd') at (1*72:2cm) {$d$};
 
 \draw[ 
        decoration={markings, mark=at position 0.125 with {\arrow{<}},
        mark=at position 0.35 with {\arrow{<}},
        mark=at position 0.6 with {\arrow{<}},
        mark=at position 0.85 with {\arrow{<}}
        },
        postaction={decorate}, dotted
        ]
 (origin') circle (2cm);
 
 \draw (node cs:name=Pa') --(node cs:name=Pb');
  \draw (node cs:name=Pb') --(node cs:name=Pc');
   \draw (node cs:name=Pc') --(node cs:name=Pd');
    \draw (node cs:name=Pa') --(node cs:name=Pd');
 \draw (node cs:name=Pb') --(node cs:name=Pd');

 \end{tikzpicture}
\]
 
 \item (cyclic symmetry) If $(a,b,c)\in\lambda$, then $(b,c,a)\in\lambda$. 
 
 \item (2-cycle axiom) If $(a,b,c)\in\lambda$, then $(a, a, c)\in\lambda$.
 
\end{enumerate}
A {\em total cyclic order} is a partial cyclic order satisfying the additional condition:
\begin{enumerate}[resume,label=(C\arabic{*})]
\item (comparability) For any $a,b,c$ we have either $(a,b,c)\in\lambda$, or $(b,a,c)\in\lambda$. 
\end{enumerate}
\end{defi}

Note the formal similarity of (C3) and  the Pachner move (exchange of two triangulations of a 4-gon).
 Note also that the implication in (C3)
is in fact an equivalence, as we can apply it twice. 
 
\begin{exas}\label{ex:cyclic-orders}
\begin{enumerate}
\item Let $C$ be a circle.
Then $C^3_\neq = C^3-\bigcup_{i\neq j} \{x_i=x_j\}$ has two connected components which are
in bijection with  orientations of $C$. The ``anti-clockwise" component corrresponding to a
choice of orientation contains triples $(x_1, x_2, x_3)$  with $x_i$ following very closely
each other in the orientation direction. The closure $\lambda$ of this component is a total
cyclic order on $C$.  In particular, the standard oriented circle $S^1 =
\{|z|=1\}\subset\CC$ has a canonical cyclic order. 
 
This implies that a total cyclic order on a finite set $J$ can be defined as an  homeomorphism class
of embeddings of $J$ into oriented circles. That is, two such embeddings $u: J\to C$  and $u':J\to
C'$  are identified  iff there is an orientation preserving homeomorphism $\phi: C\to C'$ such that
$\phi u=u'$. It is now easy to verify that a total cyclic order on a finite set as defined in 
Definition \ref{def:cyclic-order} is the same as a total cyclic order in the sense of Definition
\ref{def:totalcyclic}.
 
\item Let $\leq$ be a partial linear order on $S$. Its {\em cyclic closure} is the ternary relation
$\lambda=\lambda_\leq\subset S^3$ obtained as the closure of the set $\{x_1\leq x_2\leq x_3\}$ under
the action of $\ZZ/3$ by cyclic rotations. It is straightforward to see that $\lambda$ is a partial
cyclic order, total if $\leq$ is a total linear order, cf. \cite[Th. 3.5]{novak}.  

This implies that total cyclic order on a finite set $J$ with $|J|=n$ is the same as a class of
bijections $J\to \{1, 2, ..., n\}$  under the action of $\ZZ/n$, each bijection (i.e., a total
order) giving a cyclic order by cyclic closure. 

\item Every set $J$ has the {\em discrete partial cyclic order} consisting only of triples
	$(a,a,a)$, $a\in J$. If $|J|=1$, this is the only partial cyclic order on $J$.  If $|J|=2$,
	then $J$ has two partial cyclic orders: one discrete and the other is the total cyclic
	order, formed by all triples $(a,b,c)\in J^3$.  Note that Nov\'ak's setting \cite{novak}
	allows only one partial cyclic order on a 2-element set, the discrete one. 
	
	\item Generalizing (1), one can see that the set of points of any ``abstract circle"
	in the sense of Moerdijk \cite{moerdijk} has a total cyclic order.
	
\end{enumerate}
\end{exas}
 
\begin{exa} The real projective line $\RR P^1$, being a circle, has a cyclic order. Such an order is
	determined by a choice of orientation of $\RR^2$.  More generally, let
	$F_n=\on{Flag}(1,2,...,n; \RR^n)$ be the space of complete flags in $\RR^n$.  For example,
	$F_2=\RR P^1$.  One of the main results of \cite{fock-goncharov} (theory of positive
	configurations of flags) says, in our terminology, that $F_n$ has a natural partial cyclic
	order $\lambda$, determined by a choice of orientation of $\RR^n$.  Explicitly, a closed
	smooth curve $C\subset \RR P^{n-1}$ is called {\em convex}, if every hyperplane intersects
	$C$ in at most $n-1$ points.  In particular, such $C$ does not lie in a hyperplane. Taking
	the osculating flags to $C$ at all its points, we obtain an embedding $\gamma_C: C\to F_n$.
	The 2-element set of orientations of any convex curve  $C\subset \RR P^{n-1}$ is canonically
	identified with the 2-element set of orientations of $\RR^n$, by a version of the Frenet
	frame construction (this is the old-fashioned definition of the orientation of the
	space in terms of ``left and right screws").
	  Thus choosing an orientation of $\RR^n$, we equip any convex $C$ with a
	total cyclic order $\lambda_C$.  We then say that $({\bf f}_1, {\bf f}_2, {\bf f}_3)\in
	F_n^3$ lies in $\lambda$, if there is a convex curve $C\in \RR P^{n-1}$ and three points
	$(t_1, t_2, t_3)\in\lambda_C$ such that ${\bf f}_i=\gamma_C(t_i)$.  See {\em loc. cit.} for
	more details as well as a generalization to the complete flag variety of any split simple
	algebraic group over $\RR$.
\end{exa}

The concept of lexicographic order introduced above can be defined in the more general context of
partial orders.

\begin{defi}  \label{defi:lexico} 
Let $f: I\to J$ be a map of sets.  Suppose that $J$ is equipped with a partial cyclic order
$\lambda_J$, and each fiber $f^{-1}(j)$ is equipped with a partial linear order $\leq_j$. In this
setting we construct a partial cyclic order $\lambda=\on{Lex}(f, (\leq_j))$ on $I$, called the {\em
lexicographic cyclic order} as follows. 
We say that $(a,b,c)\in\lambda_I$ if $(f(a), f(b), f(c))\in\lambda_J$ and either
\begin{enumerate}[labelindent=1cm]
\item[(i)]  all three elements $f(a), f(b), f(c)$  are distinct, or:

\item[(ii)]  $f(a)=f(b)\neq f(c)$ and $a\leq_{f(a)} b$,  or:

\item[(ii')] $f(a)=f(c)\neq f(b)$ and $c\leq_{f(c)}  a$, or:

\item[(ii'')] $f(a)\neq f(b)=f(c)$ and $b\leq_{f(b)} c$,  or:

\item[(iii)] $f(a)=f(b)=f(c)$ and either
\begin{enumerate}
\item $a\leq_j b\leq_j c$, or
\item $b\leq_j c\leq_j a$, or
\item $c\leq_j a\leq_j b$, \,\, $j=f(a)=f(b)=f(c)$.
\end{enumerate}
\end{enumerate}
If $\lambda_J$ is a total cyclic order, and each $\leq_i$ is a total
linear order, then $\on{Lex}(f, (\leq_j))$ is a total cyclic order. 
\end{defi}

\begin{defi}\label{def:morphisms-cyclic-orders}
 Let $(I,\lambda_I)$ and $(J, \lambda_J)$ be partially cyclically ordered sets. 
\begin{enumerate}[label=(\alph*),itemsep=0cm,topsep=1ex]
	\item Assume that $\lambda_I$ is a total cyclic order. Then we define a {\em morphism} $F: I\to J$
to be a datum consisting of a map $f: I\to J$ and a total linear order $\leq_j$ on
each $f^{-1}(j)$ such that $\on{Lex}(f, (\leq_j))=\lambda_I$.
	\item In general, a morphism $F: I\to J$ is a compatible system of
morphisms $I'\to J$ given for all totally cyclically ordered subsets $I'\subset I$  
(with respect to the induced cyclic order). 
\end{enumerate}
\end{defi}
 
Note that the map of sets $f$ underlying a morphism $F: I\to J$, is always ``cyclically monotone'':
$(a,b,c)\in\lambda_I$ implies $(f(a), f(b), f(c))\in\lambda_J$. If $f$ is injective, then $F$ is
determined by $f$. We call such morphisms {\em injective}. In general, however, a morphism
contains more data than just this map $f$.
The collection of partially cyclically ordered sets with morphisms as defined above forms a category
${\mathbf \Lambda}_{\on{par}}$ which contains the categories $\Lambda$ and ${\bf \Lambda}$ as full
subcategories.

\subsubsection{Cyclic duality}
  Given a cyclically
ordered set $J$ of cardinality $n \ge 1$,  the set $J^{\vee} := \Hom_{\Lambda}(J, \langle 0\rangle)$
 is, by definition, in canonical bijection with the
set of {\em linear refinements} of the cyclic order on $J$, i.e., linear orders whose cyclic closure recovers 
the cyclic order on $J$. Further, the action of $\ZZ/n$ on $J$ induces a $\ZZ/n$-action on
$J^{\vee}$ which is simply transitive. 

\begin{exa} Let $J \subset S^1$ be a subset of the standard circle. Providing $J$ with the
	counter-clockwise cyclic order induced from $S^1$, we may identify the set
	$\Hom_{\Lambda}(J, \langle 0\rangle)$ with the set of  homotopy classes of monotone maps $f: S^1 \to S^1$
	of degree $1$ taking the subset $J$ to $1 \in S^1$. Let us refer to the arcs on $S^1$ which
	connect an element of $J$ to its successor as {\em interstices of $J$}. The map $f$ is
	uniquely determined by specifying which interstice of $J$ maps to a nontrivial loop in
	$S^1$. Therefore, $J^{\vee}$ is canonically identified with the set of interstices of $J
	\subset S^1$. For an abstract cyclically ordered set, the set $J^{\vee}$ should be regarded
	as an intrinsic construction of the set of interstices with its cyclic order, without any
	reference to an embedding into the circle. The role of $J^\vee$ as a dual object was emphasized
	by Drinfeld \cite{drinfeld:cyclic} whose result can be formulated as follows. 
	
\end{exa}

\begin{prop}
\label{prop:cyclic-interstice} The cyclic interstice construction $J \mapsto J^{\vee}$ extends to a functor
	\[
		{\bf \Lambda}^{\op} \lra {\bf \Lambda}
	\]
	which is an equivalence of categories.
\end{prop}
\begin{proof} Although this is a particular case of the general duality for planar crossed
	simplicial groups, let us provide a self-contained proof for convenience of the reader. 

 	Given a morphism of cyclically ordered sets $f: J \to K$, we obtain a map of sets of linear
	refinements $f^{\vee}: K^{\vee} \to J^{\vee}$. To lift this map to a morphism of cyclically ordered sets,
	we have to additionally provide a linear order on each fiber of $f^{\vee}$. Consider linear
	orders $\lambda_1, \lambda_2 \in K^{\vee}$ which map to the same linear order $\lambda \in
	J^{\vee}$. Let $j_{\max}$ be the maximum of $(J, \lambda)$. We say $\lambda_1 \le \lambda_2$ if
	\[
		\big|\{k \in K\; |\; f(j_{\max}) \le_{\lambda_1} k\}\big| \le \big|\{k \in K \;|\; f(j_{\max})
		\le_{\lambda_2} k\}\big|.
	\]
	This construction provides linear orders on the fibers of $f^{\vee}$
	which are compatible with the cyclic order on $K^{\vee}$. The equivalence of categories can
	be verified by noting that the double dual $(J^{\vee})^{\vee}$ can be canonically and
	functorially identified with $J$.
\end{proof}

\subsection{Dihedral orders}\label{sec:dihorder}

Let $\Delta \Gen = \Xi$ be the dihedral category. A $\Xi$-structure on a set $J$ yields a
concept of {\em dihedral order}. 

\subsubsection{Finite total dihedral orders}

Given a finite set $J$ with cyclic order $\lambda$, we can define the {\em opposite} cyclic order
$\lambda^{\op}$ on $J$ by declaring the successor $j+1$ of an element $j$ to be the predecessor 
$j-1$ of $j$ with respect to $\lambda$. The association $(J,\lambda) \mapsto (J, \lambda^{\op})$ extends to a self-equivalence of the
category $\bf \Lambda$ where the opposite $f^{\op}$ of a morphism $f: J \to K$ is defined to be the
same map on underlying sets equipped with the opposite linear order on each fiber.

\begin{defi} A {\em dihedral order} on a finite set $J$ of cardinality $\ge 3$ is defined to be the choice
	of a set $o(J) = \{\lambda_0,\lambda_1\}$ of opposite cyclic orders on $J$. A dihedral
	order on a set $J$ of cardinality $\le 2$ is defined to be the choice of an arbitrary
	two-element set $o(J)$. Given a dihedrally ordered set $(J, o(J))$, we refer to the two-element set $o(J)$ as the {\em
	orientation torsor} of $J$. 
\end{defi}

A morphism $J \to K$ of dihedrally ordered sets is defined to be a
bijection $f: o(J) \to o(K)$ and an unordered pair $\{f_0:(J,\lambda_0) \to  (K, f(\lambda_0)),f_1: (J,
\lambda_1) \to (K, f(\lambda_1))\}$ of opposite morphisms of cyclic ordinals. Here, we agree by
convention that for a dihedrally ordered set $J$ of cardinality $\le 2$, both elements of the
orientation torsor $o(J)$ induce the unique cyclic order on $J$. With this notion of morphism, the
set of finite nonempty dihedrally ordered sets organizes into a category ${\bf \Xi}$ which we call the {\em large
dihedral category}.

\begin{rems} (a) The reason for the introduction of a ``virtual'' orientation torsor for sets of
	cardinality $\le 2$ should now be clear. For example, with this convention, 
	the automorphism groups of the standard dihedrally ordered sets $\{0\}$ and $\{0,1\}$ are
	$D_1 = \ZZ/2$ and $D_2 = \ZZ/2 \times \ZZ/2$, respectively. The nontrivial
	automorphism given by interchanging the elements of the virtual orientation torsor acts
	trivially on the underlying sets. 
	
	\vskip .2cm
	
	(b) The automorphism of $\bf \Lambda$ given by passing to the opposite cyclic order can be
	regarded as an action of the (discrete $2$-) group $\ZZ/2$ on $\bf \Lambda$. Our approach
	amounts to defining $\bf\Xi$ as the quotient of $\bf\Lambda$ by this action.		 
\end{rems}

\begin{exa} \label{exa:dihedral}
	\begin{enumerate}
		\item Every set of cardinality $1 \le n \le 2$ admits a dihedral order which is unique up to
		nonunique isomorphism. This corresponds to the fact that we have surjections $D_n \to
		S_n$ with nontrivial kernel. Every set of cardinality $3$ admits a unique dihedral order 
		corresponding to the fact that we have an isomorphism $D_3 \to S_3$.
		A set of cardinality $n \ge 3$ has $|S_n|/|D_n|$ different dihedral orders.
		
	\item \label{exa:item2} 
		For a set $J$ of cardinality $\ge 3$, a dihedral order on $J$ can be identified with
		an equivalence class of embeddings of $J$ into the standard unit circle $S^1 \subset
		\mathbb C$. Here two embeddings are equivalent if they differ by a (not necessarily
		orientation preserving) homeomorphism of $S^1$.  A set $J$ of cardinality $4$ has
		$3= |S_4|/|D_4|$ dihedral orders.  These orders are in natural bijection with
		unordered partitions of $J$ into two 2-element sets: $J = \{a,c\} \sqcup \{b,d\}$.
		This can be immediately seen by identifying a dihedral order with an equivalence
		class of embeddings into $S^1$: the above partition encodes that the diagonals
		$[a,c]$ and $[b,d]$ intersect inside the unit disk.	
		
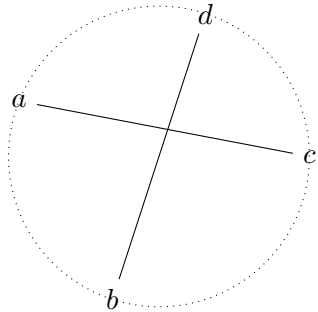
\[
			\begin{tikzpicture}

			\node (O) at (0,0) (origin) {}; 
			\node (Pb) at (0:2cm)   { $c$}; 
			\node (Pc) at (3.5*72:2cm)  {$b$}; 
			\node (Pa) at (2.2*72:2cm)  {$a$}; 
			\node (Pd) at (1*72:2cm)  {$d$};

			\draw[ 
				decoration={markings},
				postaction={decorate}, 
				dotted
				]
			 (origin) circle (2cm);

			 \draw (node cs:name=Pa) --(node cs:name=Pb);
			  \draw (node cs:name=Pc) --(node cs:name=Pd);

			 \end{tikzpicture}
		\]
	Similarly, a dihedral order on a finite set $J$ of cardinality $\geq 4$ is uniquely
	determined by the knowledge of which pairs of diagonals intersect in the unit disk. In
	particular, this means that a dihedral order on $J$ is uniquely recovered from the
	collection of the induced dihedral orders on all 4-element subsets of $J$.  

	\item Every finite cyclically ordered set $(J, \lambda)$ admits a canonical dihedral order by
	setting $o(J) = \{\lambda, \lambda^{\op}\}$ if $|J| \ge 3$ and $o(J) = \ZZ/2$ otherwise.
	We call this dihedral order the {\em dihedral closure} of $\lambda$.
	There is an obvious way to extend this construction to a faithful functor
	\[
		{\bf \Lambda} \lra {\bf \Xi} 
		\]
	We denote the dihedral closure of the standard
	cyclically ordered set $\cn$ by $\dn$. The small dihedral category $\Xi$ is then obtained as the full
	subcategory of ${\bf \Xi}$ spanned by the standard dihedrally ordered sets $\{\dn\}$.
	\end{enumerate}
\end{exa}

\begin{refo}  A more algebro-geometric version of Example \ref{exa:dihedral} \ref{exa:item2} can be obtained
	as follows. Let  $J$ be a finite set with $\geq 3$ elements, and $\overline M_{0,J}$ be the
	moduli space of stable $J$-pointed curves of genus 0. This is a smooth projective  variety
	over $\QQ$ of dimension $|J|-3$ whose open part is the quotient
	\[
	M_{0,J} = \bigl( (\PP^1)^J \setminus (\text{diagonals} )\bigr) \bigl/ PGL_2,
	\]
	see, e.g., \cite{kapranov-veronese} for more details.  Now note that connected components of the
	real locus $M_{0,J}(\RR)$ are in bijection with dihedral orders on $J$. Indeed, $\RR P^1=
	\PP^1(\RR)$ is topologically a circle, and action of $PGL_2(\RR)$ on it contains transformations
	both preserving and reversing the orientation of this circle, so a comparison as in 
	Example \ref{exa:dihedral} \eqref{exa:item2} is immediate. 
	It was proved in \cite{kapranov-permuto} that the closure in $\overline M_{0,J}(\RR)$ of each component of
	$M_{0,J}(\RR)$, is identified with the Stasheff polytope $K_{|J|-1}$. Thus 
	$\overline M_{0,J}(\RR)$ can be obtained by gluing as many Stasheff polytopes
	as there are dihedral orders on $J$. 
\end{refo}
 
\subsubsection{General dihedral orders}
  
According to Example \ref{exa:dihedral} \eqref{exa:item2}, a dihedral order on a finite
set $J$ can be understood as a $4$-ary relation on $J$ in analogy to Definition
\ref{def:cyclic-order}. In this paper we do not attempt a full investigation of this
approach which should naturally lead to a good concept of partial dihedral order on a
possibly infinite set.  
Let us just indicate the most important property of the resulting $4$-ary relation in the case
of a finite dihedrally ordered set $J$. For any $n\geq 1$ let
\[
	J^n_\neq = J^n \setminus \bigcup_{i\neq j} \{x_i=x_j\}
\]
be the set of $n$-tuples of distinct elements of $J$, and let $\gamma \subset J^4_\neq$
consist of $4$-tuples $(a,b,c,d)$ such that dihedral order on $\{a,b,c,d\}$ induced from the
embedding into $J$ coincides with the dihedral closure of the linear order $(a,b,c,d)$.  
For any ${\bf x}= (x_0, \dots , x_4)\in J^5_\neq$ let
\[
	\partial_i {\bf x} = (x_0, \dots, \widehat x_i, \dots, x_4), \quad i=0, \dots, 4, 
\]
be the tuple obtained by removing the $i$th element. 
 
\begin{prop}[The pentagon property for dihedral orders]
Let $J$ be a finite dihedrally ordered set, $|J|\geq 5$, and $\gamma\subset J^4_\neq$
defined as above. For any ${\bf x}=(x_0, \dots, x_4) \in J^5_\neq$ the following are equivalent:
\begin{enumerate}
\item[(i)] 
All $\partial_i {\bf x}$, $i=0, \dots, 4$, belong to $\gamma$.
In order words, the 
total order on $\{x_0, \dots, x_4\}$ is compatible with the induced cyclic order. 

\item[(ii)] $\partial_1{\bf x}$ and $\partial_3 {\bf x}$ belong to $\gamma$.

\item[(iii)] $\partial_0{\bf x}$, $\partial_2{\bf x}$ and $\partial_4 {\bf x}$ belong to $\gamma$. 
\end{enumerate}
\end{prop}
 
\begin{proof} Clearly, (i) implies (ii) and (iii). Let us prove that (ii)$\Rightarrow$(i).  Let
	$J\hookrightarrow C$ be an embedding into a circle defining the dihedral order.   Since
	$\partial_1{\bf x}\in\gamma$, we have that $x_0, x_2, x_3, x_4$ are in the right dihedral
	order.  In other words,  their embedding into a circle $C$ is such that they subdivide $C$
	into the following  (unoriented)  arcs: $A(x_0, x_2)$, $A(x_2, x_3)$, $A(x_3, x_4)$ and
	$A(x_4, x_0)$.  Here $A(x_p, x_q)$ denotes the arc joining $x_p, x_q$ which does not contain
	any other $x_j, j\in \{0,2,3,4\}$.  With respect to this embedding into $C$, the point $x_1$
	must lie on one of these arcs, and (i) holds precisely  when $x_1\in A(x_0, x_2)$. Now, the
	assumption that $\partial_3{\bf x}\in\gamma$ means that the diagonal $[x_0, x_2]$ crosses
	the diagonal $[x_1, x_4]$ which means that $x_1\in A(x_0, x_2)$. 
 
	 The proof that (iii)$\Rightarrow$(i) is similar. In fact, the conditions $\partial_1{\bf
	 x}\in\gamma$ and $\partial_3{\bf x}\in\gamma$ already imply (i), the situation differing from (ii)
	 by a cyclic rotation. 
\end{proof}

\subsubsection{Dihedral duality.}
 
Given any finite nonempty dihedrally ordered set $J$, consider the set $\widetilde{J^{\vee}} :=
\Hom_{{\bf \Xi}}(J, \overline {\langle 0\rangle})$. We can identify the set $\widetilde{J^{\vee}}$ with the set of pairs
$(\lambda, o)$ where $o \in O(J)$ is an orientation and $\lambda$ is a linear refinement of the
cyclic order on $J$ corresponding to $o$. In particular, $\widetilde{J^{\vee}}$ admits a natural
involution given by $(\lambda,o) \mapsto (\lambda^{\op}, o^{\op})$. We define $J^{\vee}$ to be the
orbit set under this involution. We can interpret $O(J)$ as an orientation torsor for $J^{\vee}$
which determines opposite cyclic orders on $J^{\vee}$. Hence $J^{\vee}$ admits a canonical dihedral
order. 

\begin{exa} Let $J \subset S^1$ be a subset of the standard circle considered without chosen
	orientation. Then $J$ inherits a natural dihedral order and the morphism set
	$\Hom_{\bf \Xi}(J, \overline {\langle 0\rangle})$ can be identified with homotopy classes of monotone homotopy
	equivalences $f:S^1 \to S^1$ which map the subset $J$ to the point $1 \in S^1$ where we
	orient the target circle counterclockwise. This map $f$ is uniquely determined by the choice
	of an orientation and the choice of an interstice which maps to a nontrivial loop in $S^1$.
	Therefore,  $\Hom_{\bf \Xi}(J, \overline {\langle 0\rangle})$ is in natural bijection with the set of {\em
	oriented interstices} of $J \subset S^1$. The quotient set $\Hom_{\bf \Xi}(J, \overline {\langle 0\rangle})/(\ZZ/2)$ is
	then identified with the set of {\em unoriented interstices}. The two possible choice of
	coherent orientations of all interstices correspond to the elements of the orientation torsor.
\end{exa}

\begin{prop}\label{prop:dihedral-interstice} The dihedral interstice construction $J \mapsto J^{\vee}$ extends to a functor
	\[
		{\bf \Xi}^{\op} \lra {\bf \Xi}
	\]
	which is an equivalence of categories. \qed
\end{prop}

\subsection{Paracyclic   orders}\label{sec:paraorder}

Let $\Delta \Gen = \Lambda_{\infty}$ be the paracyclic category. A $\Lambda_{\infty}$-structure on a
set $J$ defines the concept of a {\em paracyclic order}. As before, we study it at several
levels of generality. 

\subsubsection{Finite total paracyclic orders}

\begin{defi} 
	A {\em paracyclic order} on a finite nonempty set $J$ consists of 
	\begin{enumerate}
		\item a $\ZZ$-torsor $\widetilde{J}$,
		\item a surjection $\pi: \widetilde{J} \to J$ which exhibits $J$ as the orbit set of
			the bijection \[T:\; \widetilde{J} \to \widetilde{J},\quad \widetilde{j} \mapsto \widetilde{j} +
				|J|.\]
	\end{enumerate}
\end{defi}

\begin{exa}\label{exa:paracyclic-exp}
Let $C$ be an oriented circle. A choice of a base point $c\in C$ defines a canonical 
paracyclic order on any finite nonempty $J\subset C$. More precisely, the choice of $c$
produces a canonical universal covering $\pi: \widetilde C_c\to C$, with $\widetilde C_c$ being
homeomorphic to $\RR$ and oriented, and so equipped with
a canonical total order. The discrete subset $\widetilde j=\pi^{-1}(J)$ has a natural successor operation 
induced by the  order on $\widetilde C_c$,  which makes it into a $\ZZ$-torsor.
 The generator of the group $\ZZ$ of deck transformations
on $\widetilde C_c$ corresponds to the shift by $|J|$ in this torsor structure. 
 
\end{exa}

A morphism $(J, \widetilde{J}) \to (K, \widetilde{K})$ of paracyclically ordered sets is given by a
commutative diagram 
\[
	\xymatrix{
		\widetilde{J} \ar[d]_{\pi_J} \ar[r]^{\widetilde{f}} & \ar[d]^{\pi_K} \widetilde{K}\\
		J \ar[r]^f &  K
	}
\]
such that $\widetilde{f}$ is monotone with respect to the linear orders on $\widetilde{J}$ and $\widetilde{K}$
induced from the $\ZZ$-torsor structures. 
The resulting category of finite nonempty paracyclically ordered
sets is called the {\em large paracyclic category} denoted by ${\bf \Lambda_{\infty}}$.

\subsubsection{General paracyclic orders}
Following \cite[\S 3.1]{HSS1}, call a $\ZZ$-{\em ordered set} as pair $(S,F)$, consisting of
a partially ordered set $(S,\leq)$ and an  morphism of partially ordered set $F: S\to S$.

\begin{defi}
A {\em partial paracyclic order} on a set $J$ is a datum of:
\begin{enumerate}
\item  a $\ZZ$-ordered set $(\widetilde J, F)$
such that $F$ is a fixed point-free bijection. We denote by $F^\ZZ$ the group
of automorphisms of $S$ generated by $F$.

\item A bijection $\widetilde J/F^\ZZ\to J$. 

\end{enumerate}

\end{defi}

 \begin{exa} (Cf. \cite{HSS1}, Ex. 3.1.5) Let $\k$ be a discrete valued field with the ring of integers
 $\mathfrak o$ and uniformizing element $\varpi$. Then the set $B=GL(n, \k)/GL(n, \mathfrak{o})$
 is identified with the set of $\mathfrak o$-lattices $L\subset \k^n$, so it is partially ordered
 by inclusion of lattices. The automorphism $F$ given by multiplication of (matrices or lattices)
 by the scalar matrix $\varpi$, makes $B$ into a $\ZZ$-ordered set. The quotient
 \[
 B/\varpi^\ZZ \,\,=\,\, PGL(n, \k)/PGL(n, \mathfrak{o})
 \]
 has therefore a canonical (partial) paracyclic order. 
  \end{exa}

\subsubsection{Paracyclic duality.}

We denote by $\pt$ the set $\{0\}$ equipped with the paracyclic order given by the trivial
$\ZZ$-torsor. Given a finite nonempty paracyclically ordered set $(J, \widetilde{J})$, 
consider the set $\widetilde{J^{\vee}} := \Hom_{\bf \Lambda_{\infty}}(J, \pt)$ and define $J^{\vee}$
to be the orbit set under the translation action of $\ZZ$ on $\pt$. We can equip
$\widetilde{J^{\vee}}$ with the $\ZZ$-action given by pulling back the translation action on $J$. 
It is straightforward to verify that $\widetilde{J^{\vee}}$ is a $\ZZ$-torsor and the projection map
$\widetilde{J^{\vee}} \to J^{\vee}$ defines a paracyclic order on $J^{\vee}$.

\begin{exa}
	Let $J \subset S^1$ and consider the paracyclic order $\pi^{-1}(J) \to J$ on $J$ induced from the universal
	cover map $\pi: \RR \to S^1$ (corresponding to the base point $1\in S^1$.
	 Then we can identify $\Hom_{\bf \Lambda_{\infty}}(J, \pt)$
	with homotopy classes of monotone maps $\RR \to \RR$ which 
	\begin{enumerate}
		\item map the subset $\widetilde{J} \subset \RR$ to the subset $\ZZ \subset \RR$, 
		\item are equivariant with respect to translation by $|J|$ on $\widetilde{J}$ and
			translation by $1$ on $\ZZ$.
	\end{enumerate}
	Any such homotopy class is uniquely determined by specifying the interstice
	$[\widetilde{j},\widetilde{j}+1]$ of $\widetilde{J} \subset \RR$ which maps
	onto the interstice $[0,1]$ of $\ZZ \subset \RR$. Thus the set $\widetilde{J^{\vee}}$ can
	be naturally identified with the set of interstices of $\widetilde{J} \subset \RR$. The
	quotient $J^{\vee}$ can therefore be canonically identified with the set of interstices of $J \subset
	S^1$.
\end{exa}

\begin{prop} The interstice construction $J \mapsto J^{\vee}$ extends to a functor
	\[
		({\bf \Lambda_{\infty}})^{\op} \lra {\bf \Lambda_{\infty}}
	\]
	which is an equivalence of categories. \qed
\end{prop}

\newpage
\section{Structured surfaces} 

\subsection{Structured $C^\infty$-surfaces}\label{subsec:c-infty}

Let $p:G\to O(2)$ be a planar Lie group, and $\p: \GG=G^\diamond \to GL(2, \RR)$
the connective covering extending $p$. We denote $K=\Ker(p)=\Ker(\p)$. As the kernel of a connective
covering, the group $K$ is abelian. We say that $\GG$ {\em preserves
orientation} if $\p(\GG)\subset GL^+(2,\RR)$. Depending on whether $\GG$ preserves
orientation or not, we have one of the following two short exact sequences of groups
with abelian kernel:
\begin{equation}\label{eq:group-K}
	\begin{gathered}
		0\to K\lra \GG \lra GL^+(2,\RR)\to 1,
		\\
		0\to K\lra \GG \lra GL(2,\RR)\to 1.
	\end{gathered}
\end{equation}
In the first case, $K$ lies in the center of $\GG$. In the second case, the action of $\GG$ on $K$
by conjugation factors through an action of $\pi_0(GL(2,\RR))= \{1, \tau\}$ given by an involution
$\tau:K\to K$. 

\begin{prop}\label{prop:group-K}
\begin{enumerate}
\item[(a)] Let $\DG$ be the crossed simplicial group corresponding to $\GG$. Then $K=\Ker\{\rho: \Gen_0\to \ZZ/2\}$.

\item[(b)] The kernels corresponding to the various planar groups are given by:\\
  \renewcommand{\arraystretch}{2}
  \begin{tabular}{ p{1cm} || p{2cm}| p{2cm}| p{2cm}|p{2cm}|p{2cm}}
	  $G$& $\on{Spin}_N(2)$& $\widetilde{SO(2)}$  & $\on{Pin}^+_N(2)$ & $\widetilde {O(2)}$& $\on{Pin}^-_{2M}(2)$
 \\
 \hline
 $K$ & $\ZZ/N$ & $\ZZ$ & $\ZZ/N$ & $\ZZ$ & $\ZZ/2M$
    \end{tabular}\\
 	If $G$ does not preserve orientation then $\tau$ acts on $K$ by $\tau(k)=-k$. 
\end{enumerate}
\end{prop}
\begin{proof} Follows from Proposition 
\ref{prop:orient-G-0} (a) and (b) respectively.
\end{proof}

By a {\em surface} we will mean a  connected $C^\infty$-surface $S$ possibly with boundary
$\partial S$. Given a surface $S$, let $\Fr_S\to S$ be the principal $GL(2,\RR)$-bundle of frames in
the tangent bundle $T_S$. The sheaf of connected components $\underline{\pi}_0(\Fr_S)$ is a
$\ZZ/2$-torsor over $S$ which we call the {\em orientation torsor} of $S$ and denote by $\orr_S$. Thus
the total space of $\orr_S$ is the orientation cover of $S$, denoted $\widetilde S$. 

Suppose $\GG$ does not preserve orientation. In this case, the $GL(2,\RR)$-action on $K$ from
\eqref {eq:group-K} together with the $GL(2,\RR)$-torsor $\Fr_S\to S$,  gives rise to a local
system of abelian groups $\underline K^\orr$ on $S$, locally isomorphic to $\underline{K}$. If $S$
is orientable, then $\underline K^\orr \cong \underline K$ is a constant sheaf.

\begin{defi}
A $\GG$-structured surface is a datum $(S,F)$, where $S$ is a surface, and $F\buildrel\rho\over\to \Fr_S$ is a reduction
of structure groups along $\p$. 
\end{defi}

Thus $F\to S$ is a principal $\GG$-bundle on $S$, and $\rho$ is $\p$-equivariant. 

\begin{exas} We give examples using the notation of Table \ref{table:corr}.
	\begin{enumerate}[label=(\alph{*})]
		\item  For $G=O(2)$, we have $\GG=GL(2,\RR)$, and a $\GG$-structured surface is just a surface with no additional structure.

		\item For $G=SO(2)$, we have $\GG=GL^+(2,\RR)$, and a $\GG$-structure on $S$ is the
			same as an orientation of $S$.

		\item For $G=\on{Spin}_N(2)$, the group $\GG$ is the $N$-fold connected covering of $GL^+(2,\RR)$, and
		a $\GG$-structure is known as an $N$-{\em spin structure}.  The case $N=2$, being an instance
		of the spinor construction existing in any number of dimensions, has been studied more systematically, cf. \cite{johnson}. 

	\item For the unversal covering $G=\widetilde{SO(2)}$ of $SO(2)$, the group $\GG=\widetilde{GL^+(2,\RR)}$
		is the universal covering of $GL(2,\RR)$. A $\widetilde{GL^+(2,\RR)}$-structure on a
		surface will be called a {\em framing}. Note that, in our definition, a framing on
		$S$ is a discrete datum: not an actual choice of a trivialization of $T_S$ but what
		amounts to an isotopy class of such trivializations. In fact, the same is true in
		all the other cases, and this is the reason for working with $\GG$ instead of $G$.
		For instance, an $O(2)$-structure on $S$ is a Riemannian metric, while a
		$GL(2,\RR)$-structure is no extra structure at all. 
	\end{enumerate}
\end{exas}


A {\em structured diffeomorphism} of $\GG$-structured surfaces $(S,F)\to (S', F')$ is a pair
$(\phi,\widetilde \phi)$ where $\phi: S\to S'$ is a diffeomorphism, and $\widetilde\phi: F\to F'$
is a $\GG$-equivariant diffeomorphism lifting $d\phi: \Fr_S\to\Fr_{S'}$. A structured diffeomorphism
$(S,F) \to (S,F')$ with $\phi=\Id$ will be called an {\em isomorphism of $\GG$-structures} on $S$. 
We denote by $\Gstr(S)$ the groupoid formed by $\GG$-structures on $S$ and their isomorphisms.

Suppose $\GG$ preserves orientation. Then any $\GG$-structure on $S$ gives rise to an orientation.
Given an oriented surface $S$, we denote by $\Gstr^+(S)$ the full subcategory in $\Gstr(S)$ formed
by those $\GG$-structures which induce the given orientation. 

The categories $\Gstr(S), \Gstr^+(S)$ can be understood in terms of standard constructions of
non-abelian cohomology \cite{giraud}. For a Lie group $L$, let $\underline {L}$ be the sheaf of
groups on $S$ formed by local $C^\infty$-maps into $L$. Thus, the pointed set $H^1(S,\underline L)$ is the set of
isomorphism classes of principal $L$-bundles on $S$. 

\begin{prop}\label{prop:G-structures}
\begin{enumerate}
\item[(a)] Suppose $\GG$ preserves orientation, and let $S$ be an oriented surface. Then:
\begin{enumerate}
\item[(a1)] For any object $F\in\Gstr^+(S)$, the group $\Aut(F)$ is identified with $K$.

\item[(a2)] The set of isomorphism classes of objects of $\Gstr^+(S)$ is either empty or a torsor
over $H^1(S,K)$.
\end{enumerate}

\item[(b)] Suppose $\GG$ does not preserve orientation, and let  $S$
be any surface. Then: 

\begin{enumerate}
\item[(b1)] For  any object $F\in\Gstr(S)$, the group $\Aut(F)$ is identified with $H^0(S, \underline K^\orr)$. 

\item[(b2)] The set of isomorphism classes of objects of $\Gstr(S)$ is either empty, or is a torsor
over $H^1(S,\underline K^\orr)$.

\item[(b3)] The group  $H^0(S, \underline K^\orr)$ is identified with $K$, if $S$ is orientable, and with
the 2-torsion subgroup $K_2\subset K$, if $S$ is non-orientable. 
\end{enumerate}
\end{enumerate}
\end{prop}
\begin{proof} Let $X$ be a topological space
and
\[
0\to A\lra\Gamma'\buildrel\pi\over\lra \Gamma\to 1
\]
be an extension of sheaves of groups on $X$ with abelian kernel $A$. Let $P$ be a sheaf of $\Gamma$-torsors
on $X$, with associated class $[P]\in H^1(X, \Gamma)$.
 In this situation, the $\Gamma$-action on $A$ by conjugation gives rise to the twisted form
$A^P$, a sheaf of abelian groups on $X$ locally isomorphic to $A$. Denote by $\on{Lift}(P)$ the category
formed by sheaves $P'$ of $\Gamma'$-torsors on $X$ together with a $\pi$-equivariant 
map $\pi_{P'}: P'\to P$. Morphisms in $\on{Lift}(P)$ are isomorphisms of torsors
commuting with the projection to $P$. 
In the described situation, we have:
\begin{enumerate}
\item[(i)] Each group $\Aut_{\on{Lift}(P)}(P')$ is identified with $H^0(X, A^P)$.

\item[(ii)] The set of isomorphism classes of objects of $\on{Lift}(P)$ is either empty or is
a torsor over $H^1(X, A^P)$.

\item[(iii)] If $A$ is central, then there is a natural coboundary map of pointed sets
$\delta: H^1(X,\Gamma)\to H^2(X, A)$.
In this case $\on{Lift}(P)\neq\emptyset$ if and only if $\delta([P])=0$. 
\end{enumerate}
Now, parts (a1-2) and (b1-2) follow by considering the extensions of sheaves
on $S$ obtained from \eqref{eq:group-K} by passing to sheaves of $C^\infty$-sections. Part (b3) follows from
Proposition \ref{prop:group-K}(b). 
\end{proof}

\begin{exa} Let $S$ be an oriented surface and $G=\on{Spin}_N(2)$. If $S$ is noncompact or has
	nonempty boundary then we have $H^2(S,K) = 0$ so that $\Gstr^+(S)$ is nonempty. If $S$ is
	compact without boundary, then $\Gstr^+(S)$ is nonempty for $N=2$.
	This follows by considering the inclusion of short exact sequences of sheaves of groups
	\begin{equation}
		\label{eq;two-seq}
		\xymatrix{
			0\ar[r]& \underline{\ZZ/2}\ar[r]\ar[d]^{=}&\underline{S}^1\ar[d] \ar[r]^{z
			\mapsto z^2} & \underline{S}^1\ar[d]\ar[r] & 1
			\\
			0\ar[r]& \underline{\ZZ/2}\ar[r]&\underline{\widetilde {GL^+(2,\RR)}_2}\ar[r] & \underline {GL^+(2,\RR)}\ar[r] & 1
		}
	\end{equation}
	where $S^1=SO(2)$. Choosing a Riemannian metric on $S$, we reduce the structure group to $S^1$
	and so realize $[\Fr_S]$ as coming from an element of $H^1(S, \underline{S}^1)$. The coboundary map 
	$\delta_{S^1}: H^1(S, \underline S^1)\to H^2(S, \ZZ/2) \cong \ZZ/2$ is given by the degree
	modulo $2$. Since the tangent bundle of a surface of genus $g$ has degree $2g -2$, we have
	$\delta_{S^1}([\Fr_S]) = 0$.
\end{exa}

\begin{defi}
	 A {\em $\GG$-structured marked surface} is a datum consisting of:
	 \begin{enumerate}
	 \item[(1)] A compact surface $S$ (possibly with boundary). We denote $S^\circ= S-\partial S$ the interior of $S$.
	 
	 \item[(2)] A finite subset $M\subset S$.
	 
	 \item[(3)] A $\GG$-structure $F$ on $S-(M\cap S^\circ)$. 
	 \end{enumerate}

\end{defi}	 

We will usually denote $\GG$-structured   marked surfaces by  $(S,M)$, keeping the $\GG$-structure
	 implicit.
	 
	 \begin{defi} Let $(S,M)$ and $(S', M')$ be two $\GG$-structured   marked surfaces, with $\GG$-structured
	 torsors $F$ and $F'$ respectively. A {\em structured diffeomorphism} $(S,M)\to(S', M')$ is a datum
	 consisting of:
	 \begin{enumerate}
	 \item[(1)] A diffeomoprhism $\phi: S\to S'$ taking $M$ to $M'$.
	 \item[(2)] A structured diffeomorphism $\widetilde\phi: S-(M\cap S^\circ)\to S'- (M'\cap S'{}^\circ)$
	 lifting the restriction of $\phi$. 
	 \end{enumerate}
	  \end{defi}
	  
	  Note that $\phi$ in (1) is  uniquely determined by  $\widetilde\phi$ in (2).
	  
	  \vskip .2cm
	 
	  For a $\GG$-structured marked surface
	 $(S,M)$, we denote by $\Diff^\GG(S,M)$ the group of $\GG$-structured self-diffeomorphisms of $(S,M)$,
	 and by   $\Diff^\GG_\pure(S,M)$ the subgroup of structured self-diffeomorphisms preserving $M$ pointwise. 
	  We consider
	 these groups as subgroups in the group of diffeomorphisms of the structure torsor $F$ and
	 equip them with the subspace topology.

For $\GG=GL(2,\RR)$, we get the ordinary (pure) diffeomorphism group of the  marked surface $(S,M)$,
which we denote $\Diff(S,M)$ and $\Diff_\pure(S,M)$. Forgetting the $\GG$-structured data gives
morphisms of topological groups
\[
\begin{gathered}
	\pi: \Diff^\GG(S,M) \lra\Diff(S,M), 
	\\
	\pi_\pure: \Diff_\pure^\GG(S,M) \lra\Diff_\pure(S,M).  
\end{gathered}
\]

\begin{prop}\label{prop:str-non-str-diff}
The group $\Ker(\pi)=\Ker(\pi_\pure)$ is a discrete abelian group, so that $\pi$ and $\pi_\pure$ are
unramified coverings. Further, the groups $\on{Im}(\pi)$ and $\on{Im}(\pi_\pure)$ are 
unions of connected components of $\Diff(S,M)$. 
\end{prop}

\begin{proof} The group $\Ker(\pi)=\Ker(\pi_\pure)$ is identified with $\Aut(F)$, the group
of automorphisms of the given $\GG$-structure $F$ on $S-(M\cap S^\circ)$. This group was described
in Proposition \ref {prop:G-structures} (a1)(b1), in particular it is abelian. Further, $\phi\in\Diff(S,M)$ lies
in $\on{Im}(\pi)$ if and only if the $\GG$-structure $\phi^*F$  on $S-(M\cap S^\circ)$  is isomorphic to $F$. So 
$\on{Im}(\pi)$ is the stabilizer of the action of $\Diff(S,M)$ on the  set of
isomorphism classes of objects of $\Gstr(S-(M\cap S^\circ))$. This set was described in Proposition 
 \ref {prop:G-structures}  (a2)(b2), in particular, it is discrete. So $\on{Im}(\pi)$ is open and
 closed, and similarly for $\on{Im}(\pi_\pure)$. 
\end{proof}
 
\begin{defi}
 Let $(S,M)$ be a $\GG$-structued marked surface. We define the $\GG$-structured mapping class
 groups as
 \[
 	\Mod^\GG(S,M) = \pi_0 \Diff^\GG(S,M), \quad  \Mod^\GG_\pure(S,M) =\pi_0 \Diff_\pure^\GG(S,M). 
 \]
\end{defi}

The following classical construction will be used frequently.

\begin{defi}
Let $S$ be a compact $C^\infty$-surface with boundary $\partial S$, and $S^\circ = S-\partial S$. 
The {\em Schottky double} of $S$ is the compact oriented $C^\infty$ surface $S^\#$
without boundary obtained by compatifying $\widetilde{S^\circ}$, the orientation cover of $S^\circ$,
by adding one copy of $\partial S$. Thus we have a 2-sheeted covering $\pi: S^\#\to S$, ramified
along $\partial S$, and the deck involution $\sigma: S^\#\to S^\#$ of this covering, reversing the
orientation and identical on  $\partial S$. If $(S,M)$ is a marked surface, we equip $S^\#$ with the
subset $M^\#=\pi^{-1}(M)$. 
\end{defi}

Note that $S^\#$ is not the same as $\widetilde S$, if $\partial S\neq\emptyset$. 

\begin{defi} \cite{seppala} A surface $S$ is called {\em classical}, if $S$ is orientable and $\partial S=\emptyset$,
and {\em non-classical} otherwise. 
\end{defi}

Thus $S^\#$ is connected, if and only if $S$ is non-classical. 

\begin{defi} A marked surface $(S,M)$ is called {\em stable}, if:
\begin{enumerate}
\item[(1)]  $M\neq\emptyset$ and moreover, $M$ meets each component of $\partial S$.

\item[(2)] $\chi (S^\# - M^\#) < 0$. 
\end{enumerate}

\end{defi}

For orientable surfaces this is equivalent to Def. 3.11 of \cite {HSS-triangulated}. 
For a stable marked surface $(S,M)$ the points of $M$ lying in the interior of $S$, will be called
{\em punctures}. By blowing up each puncture to a small circle, one can obtain an alternative
point of view on stable marked surfaces, in which all marked points are supposed to lie on
the boundary, but it is allowed to have boundary components without a marked point.
See \cite[Rem. 3.3.2]{HSS-triangulated}. 
     
\begin{prop}\label{prop:diff-contract}
If $(S,M)$ is a stable $\GG$-structured marked surface, then each component of
$\Diff^\GG(S,M)$ is contractible.
\end{prop}
     
\begin{proof} By Proposition \ref{prop:str-non-str-diff}, the statement reduces to
the particular case of $\GG=GL(2,\RR)$ where it translates to the contractibility of each component
or, equivalently, of the identity component of $\Diff(S,M)$. This statement is shown in
\cite{gramain} 
\end{proof}

\subsection{Conformal structured surfaces}

\subsubsection {The setup}
As in \S \ref{subsec:c-infty}, we fix a planar Lie group $G$ and the corresponding connective covering
$\GG\buildrel\p\over\to GL_2(\RR)$. Let
\begin{equation}
	\label{eq:emb-conf}
	\Conf(2) = (\ZZ/2)\ltimes \CC^* \subset GL(2, \RR)
\end{equation}
be the subgroup of conformal transformations. Since the embedding \eqref{eq:emb-conf}
is a homotopy equivalence, the categories of connective coverings of $GL(2,\RR)$ and $\Conf(2)$
are equivalent by Proposition \ref {prop:summary-covers}(c). We denote
\[
p_\conf: G_\conf\lra \Conf(2)
\]
the connective covering obtained by restricting $\p$. 

Similarly to  \S \ref{subsec:c-infty}, by a
$G_\conf$-{\em structured surface} we mean a surface $S$ together with a restriction of the structure
group of $\Fr_S$ to $G_\conf$. As before, such a structure is a datum of a principal $G_\conf$-bundle
$F_\conf$ together with a compatible $P_\conf$-equivariant map $F_\conf\to\Fr_S$. 
The concepts of a structured diffeomorphism between two $G_\conf$-structured surfaces and of an isomorphism
between two $G_\conf$-structures on a given surface $S$ are defined entirely similar to 
\S \ref{subsec:c-infty}. In particular, we have a groupoid $\Gcstr(S)$ of $G_\conf$-structures on $S$
and their isomorphisms. If $\GG$ preserves orientation and $S$ is oriented, we also have the full
subcategory $\Gcstr^+(S)$ of $G_\conf$-structures compatible with the orientation. 

\begin{exas}
	\begin{enumerate}[label=(\alph{*})]
	\item An $SO(2)_\conf$-structured surface is a {\em Riemann surface}, i.e., a
		1-dimensional complex manifold $X$ possibly with boundary. More precisely, an
		isomorphism class of objects of $SO(2)_\conf \on{Str}(S)$ is the same as an almost
		complex structure $J: T_S\to T_S$, $J^2=-1$, and the automorphism group of any
		object of  $SO(2)_\conf \on{Str}(S)$ is trivial. As well known, all almost complex
		structures in complex dimension $1$ are integrable. For a Riemann surface $X=(S,J)$
		we denote by $\overline X = (S, -J)$ the {\em conjugate} Riemann surface.

	\item  A $\on{Spin}_N(2)_\conf$-structured surface is a triple $(X, \Lc, \phi)$
		where $X$ is a Riemann surface, $\Lc$ is a line bundle on $X$, and $\phi:
		\Lc^{\otimes N}\to\omega_X$ is an isomorphism. Here $\omega_X$ is the canonical line
		bundle of $X$.  For $N=2$, such line bundles $\Lc$ are traditionally called {\em
		theta-characteristics}, and a choice of $\Lc$ is often referred to as a ``spin
		structure" on $X$, cf.  \cite{atiyah}. For ``higher spin structures" on Riemann
		surfaces ($N>2$), see \cite{jarvis}. 
  
	\item An $O(2)_\conf$-structured surface is a {\em Klein surface}, i.e., a
		surface together with an atlas whose transition maps are conformal but do not
		necessarily preserve orientation.  Such maps are called {\em dianalytic} (i.e.,
		analytic or anti-analytic). Thus a Klein surface does not have to be orientable.
		See \cite{alling-greenleaf, braun} for more background. 
	\end{enumerate}
\end{exas}

\begin{prop}\label{prop:prod-of-grpds}
\begin{enumerate}
\item[(a)] For any surface $S$ we have a natural identification of groupoids
\[
\Gcstr(S) \simeq \Gstr(S)  \times  O(2)_\conf \on{Str}(S). 
\]
\item[(b)] The automorphism group of any object of $O(2)_\conf \on{Str}(S)$ is trivial.
\end{enumerate}
\end{prop}

\begin{proof} (a) By definition, we have
\[
G_\conf = \GG \times_{GL(2,\RR)} O(2)_{\conf}
\]
where $O(2)_{\conf}=\Conf(2)$. So given a $\GG$-structure $F\to \Fr_S$ and an $O(2)_\conf$-structure $F^{O(2)}_\conf \to\Fr_S$,
we obtain a $G_\conf$-structure by forming the fiber product
\[
F_\conf \,\,=\,\, F\times_{\Fr_S} F^{O(2)}_\conf
\]
which is a principal $G_\conf$-bundle. We leave the remaining details to the reader.

(b) Follows  because $O(2)_{\conf}$ is a subgroup of $GL(2,\RR)$. 
\end{proof}

\subsubsection {Klein surfaces and algebraic curves over $\RR$}

Let $X$ be a Klein surface, and $X^\circ=X-\partial X$. For each $x\in X^\circ$, there are exactly
two complex structures on $T_xX$ compatible with the given conformal structure. The set of all these
structures for all $x\in X^\circ$ is canonically identified with the orientation cover $\widetilde{X^\circ}$.
Thus $\widetilde{X^\circ}$ has a canonical complex structure, making in into a (possibly non-compact)
Riemann surface. Further, the complex structure on $\widetilde{X^\circ}$ extends naturally
to the Schottky double $X^\#= \widetilde{X^\circ}\sqcup \partial X$. The local structure of $X^\#$
near the boundary is clear from the following example.

\begin{exa}
Let $X$ be a Riemann surface. Then $X^\# = X\cup_{\partial X} \overline X$. 
\end{exa}

Thus $X^\#$, being a compact Riemann surface without boundary, can be considered as a smooth
projective algebraic curve over $\CC$. Further, the canonical involution $\sigma: X^\#\to X^\#$ is
anti-holomorphic and makes $X^\#$ defined over $\RR$, so that the real locus of $X^\#$ is
$X^\#(\RR)=\partial S$. The following is classical \cite{alling-greenleaf, natanzon, braun}.

\begin{prop}
The Schottky double construction establishes an equivalence between the following categories:
\begin{itemize}
\item[(i)] Non-classical Klein surfaces $X$.

\item[(ii)] Pairs $(Y,\sigma)$ where $Y$ is a compact Riemann surface without boundary and
$\sigma: Y\to Y$ is an anti-holomorphic involution.

\item[(iii)] Smooth projective algebraic curves over $\RR$. \qed

\end{itemize}

\end{prop}

Pairs $(Y,\sigma)$ as in (ii) are known as {\em symmetric Riemann surfaces}. The Klein surface
corresponding to a symmetric Riemann surface  $(Y,\sigma)$ is the orbit space $Y/\sigma$.

\subsubsection{Teichm\"uller spaces of marked Klein surfaces}

Let $(S,M)$ be an oriented marked surface without boundary. We denote by $\Riem^+(S)$ the space of
all complex structures on $S$ compatible with the orientation.
Let $\Diff(S,M)_e$ be the identity component of $\Diff(S,M)$. Diffeomorphisms in this component
preserve the orientation, so $\Diff(S,M)_e$ acts on $\Riem^+(S)$. The classical
{\em Teichm\"uller space} of $(S,M)$ is defined as the quotient
\[
\Teich^+(S,M) \,\,=\,\, \Diff(S,M)_e \backslash \Riem^+(S). 
\]
The following is well known. 

\begin{thm}
Let $g$ be the genus of $S$ and suppose that $(S,M)$ is stable, i.e., $2g-2+|M|>0$. Then:
\begin{enumerate}
\item[(a)] $\Teich^+(S,M)$ has a natural structure of a complex manifold of dimension $3g-3+|M|$. As
a $C^\infty$-manifold, it is diffeomorphic to Euclidean space. 

\item[(b)] Let $\sigma:S\to S$ be an orientation reversing involution preserving $M$ as a set. Conjugation
with $\sigma$ defines an anti-holomorphic involution
\[
\sigma_{\Teich}: \Teich^+(S,M) \lra \Teich^+(S,M). \qed
\]

\end{enumerate}
\end{thm}

Let now $(S,M)$ be an arbitrary stable marked surface. Denote by $\Klein(S)$ the space of all
possible conformal structures on $S$. The {\em Teichm\"uller space}
of $(S,M)$ is defined by
\[
\Teich(S,M) = \Diff(S,M)\backslash \Klein(S). 
\]

\begin{thm}\label{thm:klein-teich}
\begin{enumerate}
\item[(a)] If $S$ is classical (i.e., orientable without boundary), then the space $\Teich(S,M)$ is the disjoint union of
two connected components, corresponding to the two orientations of $S$. Each component is canonically
identified with the corresponding classical Teichm\"uller space $\Teich^+(S,M)$. 

\item[(b1)] Suppose $S$ is non-classical, let $\pi: S^\#\to S$ be its Schottky double with involution
$\sigma$, and let $M^\#=\pi^{-1}(M)$. Let $g^\#$ be the genus of $S^\#$. Then
\[
\Teich(S,M) =  (\Teich^+(S^\#, M^\#))^{\sigma_{\Teich}}
\]
is the fixed point locus of the anti-holomorphic involution $\sigma_{\Teich}$. 
In particular, it is a real analytic manifold of dimension $3g^\#-3+|M^\#|$. 

\item[(b2)] Suppose $S$ is non-orientable, let $\varpi: \widetilde S\to S$ be its
orientation cover with the orientation reversing deck involution $\tau$, and let $\widetilde M=\varpi^{-1}(M)$.
Then
\[
	\Teich(S,M) = (\Teich^+(\widetilde S, \widetilde M))^{\tau_{\Teich}}
\]
is the fixed point locus of the involution $\tau_{\Teich}$. 

\item[(c)] In the situation (b1), the manifold $\Teich(S,M)$ is diffeomorphic to Euclidean space.

\end{enumerate}
\end{thm}

\begin{proof} Parts (a), (b1-2) are obvious. Part (c) was proved by Natanzon \cite{natanzon}
for $\partial S = M=\emptyset$ and Sepp\"al\"a \cite{seppala} for $M=\emptyset$. The extension 
to the case $M\neq\emptyset$ is can be done by the same methods.
\end{proof}

\subsubsection {Moduli spaces of marked $G_\conf$-structured surfaces}

\begin{defi}
A $G_\conf$-{\em structured marked surface} is a datum consisting of:
\begin{enumerate}
\item[(1)] A conformal surface $S$.

\item[(2)] A finite subset $M\subset S$.

\item[(3)] A $G_\conf$-structure on $S-(M\cap S^\circ)$ compatible with the conformal
structure restricted from $S$. 
\end{enumerate}
\end{defi} 

 Given a $C^\infty$ marked surface $(S,M)$, we have a  topological groupoid   $G_\conf\on{Str}(S,M)$
 formed by pairs consisting of a conformal structure on $S$ and a compatible $G_\conf$-structure on
 $S-(M\cap S^\circ)$. By Proposition \ref {prop:prod-of-grpds} we have
 \be\label{eq:prod-str}
 G_\conf\on{Str}(S,M)\,\,\simeq \,\, \GG\on{Str}(S-(M\cap S^\circ))\times O(2)_\conf\on{Str}(S). 
 \ee

Let $(S, M)$ be a stable marked surface. The {\em moduli spaces} of $G_\conf$-structures on
$(S,M)$ are defined as the groupoid quotients
 \be
\begin{gathered}
\Mc^{G_\conf}(S,M) \,\,=\,\, \Diff(S,M) \Bbs \Gcstr(S,M)),  \\
\Mc^{G_\conf}_\pure (S,M) \,\,=\,\, \Diff_\pure(S,M) \Bbs \Gcstr(S,M). 
\end{gathered}
\ee
Thus they are topological groupoids, and $\Mc^{G_\conf}(S,M)$ is a further quotient of
$\Mc^{G_\conf}_\pure(S,M)$ by a finite subgroup of permutations of $M$.

\begin{rem}
Equivalently, $\Mc^{G_\conf}(S,M)$ can be seen as the groupoid formed by all $G_\conf$-structured marked
surfaces of topological type $(S,M)$ and by their structured diffeomorphisms. 
\end{rem}

\begin{exa}
Let $M$ be a finite set and suppose that $2g-2+|M|>0$. Let $\Mc(g,M)$ be the Deligne-Mumford moduli stack
formed by smooth projective algebraic curves $X$ of genus $g$ together with an embedding $M\hookrightarrow X$. 
Let $(S,M)$ be a classical marked surface of genus $g$. Any choice of an orientation of $S$ identifies
$\Mc^{SO(2)_\conf}_{\pure}(S,M)$ with the groupoid (orbifold) of $\CC$-points $\Mc(g,M)(\CC)$. Change of orientation corresponds to the
action of $\on{Gal}(\CC/\RR)$ on $\Mc(g,M)(\CC)$ which comes from the fact that $\Mc(g,M)$ is defined
over $\RR$ (in fact, over $\QQ$). The  orbifold $\Mc^{SO(2)_\conf}$ is therefore the quotient of $\Mc(g,M)(\CC)$
by the symmetric group of $M$. As is well known,
\[
\Mc(g,M)(\CC) = \Mod^+_\pure(S,M) \Bbs \Teich^+(S,M)
\]
is the quotient of the topological cell $\Teich^+(S,M)$ by the mapping class group
of orientation preserving diffeomorphisms fixing all points of $M$. In what follows, we realize a general
$\Mc^{G_\conf}(S,M)$ in terms of quotients of cells by discrete groups.
\end{exa}

Let $(S,M)$ be stable, and define
\begin{equation}
	\label{eq:top-moduli}
	\Mc^\GG(S,M) = \Diff(S,M)\Bbs \Gstr(S-(M\cap S^\circ)) \simeq
	\Mod(S,M) \Bbs \Gstr(S-(M\cap S^\circ)) 
\end{equation}
as the ``topological" analog of the moduli space, in which we consider $\GG$ instead of $G_\conf$ as
a structure group. Here the last equivalence is a homotopy equivalence of topological
groupoids coming from contractibility of $\Diff(S,M)_e$
(Proposition \ref {prop:diff-contract}). So we can and will consider $\Mc^\GG(S,M)$
as a set-theoretic, non-topological, groupoid given by the right-hand side of \eqref{eq:top-moduli}. 
Similarly, we define $\Mc^\GG_\pure(S,M)$, using $\Mod_\pure(S,M)$. 

\begin{prop}
We have a homotopy equivalence of topological groupoids
\[
\Mc^{G_\conf}(S,M) \simeq \coprod_{ \{F\} \in \Mc^\GG(S,M)} \Mod^\GG(S_F, M) \Bbs \Teich(S,M)
\]
and a similar identification of $\Mc^{G_\conf}_\pure(S,M)$. Here $\{ F\}$ runs over
isomorphism classes of objects of $\Mc^\GG(S,M)$, and $S_F$ is the $\GG$-structured
surface corresponding to $F$. 
\end{prop}

\begin{proof} Using  \eqref{eq:prod-str}, we have
\begin{align*}
\Mc^{G_\conf}(S,M) & = \Diff(S,M) \Bbs \bigl( \Gstr(S-(M\cap S^\circ)) \times O(2)_\conf\on{Str}(S)\bigr)\\
& = \Diff(S,M) \biggl\backslash \hskip -.25cm \biggl\backslash \coprod_{ \{F\}\in \Gstr(S-(M\cap S^\circ))} O(2)_\conf\on{Str}(S)\\
& = \coprod _{ \{F\}\in \Diff(S,M) \bbs \Gstr(S-(M\cap S^\circ))} \biggl( \on{Stab}(F) \Bbs O(2)_\conf\on{Str}(S)\biggr).
\end{align*}
Note that at the  bottom of last coproduct the quotient by $\Diff(S,M)$ is $\Mc^\GG(S,M)$. At the same time,
$\on{Stab}(F)$ is the union of connected components of $\Diff(S,M)$, in particular,
$\on{Stab}(F)_e = \Diff(S,M)_e$. 
Therefore the above coproduct can be identified with
\[
\begin{gathered}
\coprod_{ \{F\}\in\Mc^\GG(S,M)} \pi_0 (\on{Stab}(F)) 
\biggl\backslash \hskip -.25cm \biggl\backslash \biggl(  \on{Stab}(F)_e \Bbs O(2)_\conf\on{Str}(S)\biggr)\,\,=
\\ =\,\,
\coprod_{ \{F\} \in\Mc^\GG(S,M)} \Mod^\GG(S_F, M) \Bbs \Teich(S,M)
\end{gathered}
\]
as claimed. In the last identification, we identified the topological groupoid $O(2)_\conf\on{Str}(S)$
with the topological space $\Klein(S)$ of its isomorphism classes. This is possible because
all automorphism groups of objects in this groupoid are trivial by Proposition 
\ref{prop:prod-of-grpds}(b).\end{proof}

\newpage
\section{Structured graphs} \label{sec:structuredgraphs}
\subsection{Structured graphs and structured surfaces}\label{subsec:structuredgraphs}
   
\begin{defi}
A {\em graph} $\Gamma$ is a pair of sets $(H,V)$ equipped with
\begin{itemize}
	\item an involution $\tau: H \to H$,
	\item a map $s: H \to V$.
\end{itemize}
The elements of $H$ are called {\em halfedges} of $\Gamma$, the fixed points of $\tau$ are called
{\em external halfedges}, and the nonfixed points of $\tau$ are called {\em internal halfedges}. 
The $2$-element orbits of $\tau$, which are hence comprised of a pair of internal halfedges, are called {\em edges}. 
The elements of $V$ are called the {\em vertices} of $\Gamma$.  Given a vertex $v \in V$, we define the set $H(v) =
s^{-1}(v)$ of {\em halfedges incident to $v$}. The cardinality $|H(v)|$ is called the valency of
$v$.
\end{defi}

\begin{defi}
Let $\Gamma$ be a graph. The {\em incidence category} $I(\Gamma)$ of $\Gamma$ has, as objects,
all the vertices and edges of $\Gamma$ with a non-identity morphism $s(h) \to
\{h,\tau(h)\}$ for every internal halfedge $h$. The {\em incidence diagram of $\Gamma$} is the functor
\[
	I_{\Gamma}: I(\Gamma) \lra \Set
\]
given on objects by assigning to a vertex $v$ the set $H(v)$ and to an edge $e$ with half-edges $h,h'$,  the
$2$-element set  $\{h,h'\}$. To a morphism $s(h) \to \{h,h'\}$ we associate the map of sets
$H(s(h)) \to \{h,h'\}$ given by mapping $h$ to $h'$ and collapsing the remaining halfedges in
$H(s(h))$  (denoted by A in Fig. \ref{fig:F1}) to $h$.  
\end{defi}

We say a graph is {\em compact} if all its halfedges are internal. Unless stated otherwise, we will
assume all graphs to be compact.

\begin{figure}[h!]
\centering
 \begin{tikzpicture}[scale =0.5]
    
    \node (me)  at (0,0){$\star$}; 
\node (v') at (7,0){$\bullet$}; 
\draw (me) -- (v'); 
\node (v) at (-7,0){$\bullet$};
\draw (v) -- (me); 

\draw (v) -- (-11,0); 
\draw (v) -- (-10, 3); 
\draw (v) -- (-10, -3); 

\node at (-3.5, .5) {$h$}; 
\node at (3.5, .5) {$h'$}; 
\node at (0, .7) {$m(e)$}; 

\node at (-6.5,0.5) {$v$}; 
\node at (7.5,0.5) {$v'$}; 

\draw [decorate,decoration={brace,amplitude=10pt},xshift=-12pt,yshift=0pt]
(-12,-3.5) -- (-12,3.5) node [black,midway,xshift=-0.6cm] 
{ $A$};

\end{tikzpicture}
\caption{The incidence map: $m(e)$ is the midpoint of the edge $e=\{h,h'\}$. }
\label{fig:F1}
\end{figure}
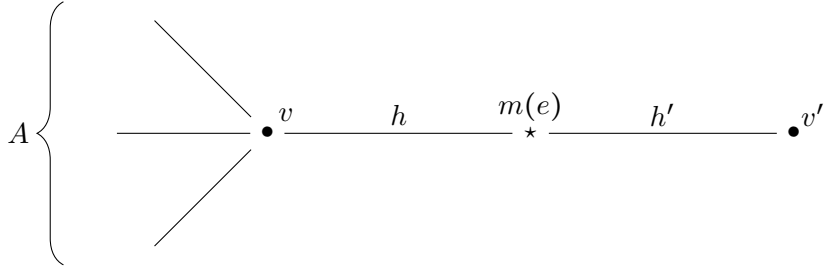

\begin{defi}
	Given a graph $\Gamma$, we define its {\em realization} $|\Gamma| = |\N(I(\Gamma)|$ to be
	the geometric realization of its incidence category. We denote by $\partial |\Gamma|\subset|\Gamma|$ 
	the set of 1-valent vertices and by $|\Gamma|^\circ$ the complement $|\Gamma| - \partial|\Gamma|$. 
\end{defi}

\begin{rem}\label{rem:midpoint}
 	Note that the realization $|\Gamma|$ comes equipped with two kinds
	of $0$-cells. On the one hand, we have a $0$-cell for each vertex $v$ of $\Gamma$, on the
	other hand, we have a $0$-cell $m(e)$ for each edge $e$ of $\Gamma$ which can be regarded as a
	chosen midpoint of the geometric edge which connects the two vertices of $\Gamma$ incident
	to $e$, see Figure \ref{fig:F1}. 
\end{rem}

\begin{defi} \label{defi:ggraph} Let $\Gamma$ be a graph and let $\G$ be a small category equipped with a functor $\G
	\to \Set$. A {\em $\G$-structure on $\Gamma$} is a lift 
	\[
	\xymatrix{
			& \G \ar[d] \\
			I(\Gamma) \ar[r]^{I_{\Gamma}} \ar[ur]^{\widetilde{I_{\Gamma}}} & \Set
	}
	\]
	of the incidence diagram of $\Gamma$ to $\G$. 
\end{defi}

We now fix a planar Lie group $\rho: G \to O(2)$, the corresponding crossed
simplicial group $\DG$, and the larger (but equivalent) category $\G$ of $\DG$-structured sets. 
Note that the interpretation of the objects of $\G$ as sets with
extra structure gives a forgetful functor $\G \to \Set$. In this context, we will slightly abuse
notation, and refer to $\G$-structured graphs as $\DG$-structured graphs.
   
\begin{exas}\label{defi:exggraphs}  
	\begin{enumerate}
		\item A graph with a $\Lambda$-structure is called a {\em ribbon graph}.
		Explicitly, a $\Lambda$-structure on a graph $\Gamma$ is the datum of a cyclic order
		on the $H(v)$ for each vertex $v\in V$. Each  set $\{h,h'\}$ of half-edges of any edge $e$
		has a trivial cyclic order, since it has cardinality 2. 
		Thus, our definition reduces to the usual one
		\cite{penner:book}. 
		
		\item A graph with a $\Xi$-structure is called a {\em M\"obius graph}.
		Explicitly, a $\Xi$-structure on a graph $\Gamma$  consists,
		first, of a dihedral order on each $H(v)$ and, second, of identification
		of the orientation torsors 
		\[
		O(H(v))\buildrel\sim\over \lla O(e)\buildrel\sim\over \lra O(H(v'))
		\]
		for each edge $e$ with vertices
		$v$ and $v'$ (the case $v=v'$ is allowed).
		 This structure  
		is equivalent to the (somewhat more cumbersome)  concept of a M\"obius graph defined
		 in terms of ``ribbon graphs with  $\ZZ/2$-grading
		on edges" as in \cite{braun, mulase-waldron, mulase-yu}.  
								
		\item A graph with a $\Lambda_{\infty}$-structure is called a {\em framed graph}.
		\item A graph with a $\Lambda_N$-structure is called an {\em $N$-spin graph}.
		Very recently, surfaces with 2-spin structure were studied, from a combinatorial point of
		view by Novak and Runkel \cite{novak-runkel}. They introduced a concept of
		a combinatorial spin structure on a triangulation $T$ of a surface $S$. Such a structure
		equips the dual graph of $T$ with a $\Lambda_2$-structure in our sense. 
	\end{enumerate}
\end{exas}

For a 2-dimensional $\RR$-vector space $V$, we denote by
\[
C(V) \,\,=\,\, (V-\{0\})/\RR^*_{>0}
\]
the circle of directions of $V$. In particular, for any surface $S$ and any $x\in S$ we have the circle
$C(T_xS)$ of tangent directions at $x$.

\begin{defi} \label{exa:embedding} Let $S$ be a surface. By an {\em embedding}
	of a graph  $\Gamma$ into $S$ we mean an injective, continous map $\gamma: |\Gamma| \to S$ such
	that
	\begin{itemize}
		\item $\gamma(|\Gamma|^\circ)$ is contained in the interior of $S$, and $\gamma(\partial|\Gamma|)\subset\partial S$.
		\item $\gamma$ is smooth along every edge of $|\Gamma|$.
		\item For every vertex $x$ of $|\Gamma|$ (corresponding to a vertex or
			and edge of $\Gamma$), the tangent directions of the edge germs on
		$S$ which leave $\gamma(x)$, are distinct.
	\end{itemize}
\end{defi}
	
\begin{prop}\label{prop:emb-graph-str}
	Let $S$ be a surface with a $\GG$-structure. An embedding of a graph $\Gamma$ into $S$ endows $\Gamma$
	with a $\DG$-structure. 
\end{prop}

\begin{proof} 
	Let $h$ be a half-edge of $\Gamma$ incident to the vertex $v$ and with corresponding edge
	$e=\{h,h'\}$. We denote by $\gamma(|h|)$ the corresponding path in $S$ which runs from
	$\gamma(v)$ to the midpoint $m(e)= \gamma(e)$.

	A $\GG$-structure on $S$ gives, via the homotopy equivalence $\GG\to\Homeo^G(S^1)$, a
	$\Homeo^G(S^1)$-structure on each circle $C(T_xS)$. Therefore, by Corollary
	\ref{cor:structured-circle-orders}, each finite subset of $C(T_xS)$ becomes, canonically, a
	$\DG$-structured set. In particular, for each vertex $v \in V$, we have a canonical
	embedding of sets $H(v)\hookrightarrow C(T_{\gamma(v)} S)$ given by association to a
	half-edge $h$ the direction given by the germ of the path $\gamma(|h|)$ leaving $v$.  By
	the above, this provides a canonical $\DG$-structure on the set $H(v)$.  Further, let $e$ be
	an edge of $\Gamma$ given by a pair of half-edges $h, h'$  as above. 
	The two tangent directions to $|\Gamma|$ at $\gamma(m(e))$ correspond to the two half-edges
	$h$ and $h'$ of $e$. As before, the embedding $|\Gamma|\to S$ gives a $\DG$-structure on the
	2-element set of these directions, i.e., on $\{h,h'\}$. This defines the values of the
	functor $\widetilde I_\Gamma: I(\Gamma)\to \Gc$ on objects.

	To define $\widetilde I_\Gamma$ on morphisms, consider one of the half-edges of $e$, say
	$h$, and let $v=s(h)$ be the corresponding vertex. Consider the $\Homeo^G(S^1)$-structured
	circles
	\[
	C_v = C(T_{\gamma(v)}S), \quad C_e = C(T_{\gamma(m(e))} S). 
	\]
	Let $J_c\subset C_v$ be the union of small non-intersecting closed arcs centered around
	images of elements of $H(v)$.  We write $J_v=A\sqcup B$, where $B$ is the arc corresponding
	to $h\in H(v)$, and $A$ is the union of all the other arcs, see Figure \ref{fig:F2}.

	 \def\centerarc[#1] (#2) (#3:#4:#5)
	{ \draw[#1] (#2) ++(#3:#5) arc (#3:#4:#5);
	}
	
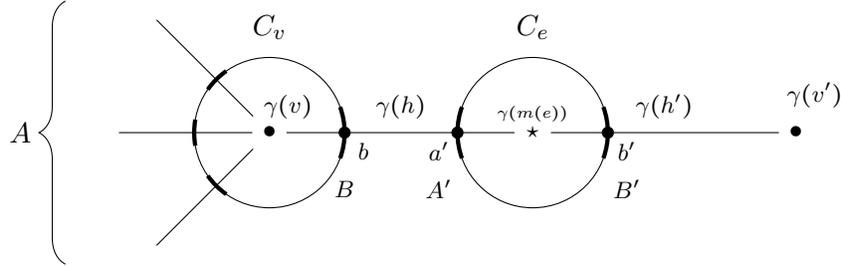
\begin{figure}[h!]
\centering
 \begin{tikzpicture}[scale =0.5]
    
    \node (me)  at (0,0){$\star$}; 
    \node at (0,.5){\tiny$\gamma(m(e))$}; 
\node (v') at (7,0){$\bullet$}; 
\draw (me) -- (v'); 
\node (v) at (-7,0){$\bullet$};
\draw (v) -- (me); 

\draw (v) -- (-11,0); 
\draw (v) -- (-10, 3); 
\draw (v) -- (-10, -3); 

\node at (-3.5, .7) {\footnotesize $\gamma(h)$}; 
\node at (3.5, .7) {\footnotesize$\gamma(h')$}; 
\node at (0, .7) {}; 

\node at (-6.5,0.7) {\footnotesize$\gamma(v)$}; 
\node at (7.5,1) {\footnotesize$\gamma(v')$}; 

\draw [decorate,decoration={brace,amplitude=10pt},xshift=-12pt,yshift=0pt]
(-12,-3.5) -- (-12,3.5) node [black,midway,xshift=-0.6cm] 
{ $A$};

\draw (0,0) circle (2cm); 

\draw (-7,0) circle (2cm);

 \centerarc[ultra thick]  (0,0) (-20:20:2cm);
  \centerarc[ultra thick]  (0,0) (160:200:2cm);
   \centerarc[ultra thick]  (-7,0) (-20:20:2cm);
    \centerarc[ultra thick]  (-7,0) (170:190:2cm);
     \centerarc[ultra thick]  (-7,0) (125:145:2cm);
      \centerarc[ultra thick]  (-7,0) (215:235:2cm);
      
 \node at (-2.5,-1.5) {\footnotesize$A'$}; 
 \node at (2.5,-1.5) {\footnotesize$B'$}; 
  \node at (-5,-1.5) {\footnotesize$B$}; 
  
  \centerarc [ultra thick]  (2,0) (0:360:3pt);
   \centerarc [ultra thick]  (-2,0) (0:360:3pt);
    \centerarc [ultra thick]  (-5,0) (0:360:3pt);
    
    \node at (2.5, -.5) {\footnotesize$b'$};
      \node at (-2.5, -.5) {\footnotesize$a'$};
        \node at (-4.5, -.5) {\footnotesize$b$};
        
        \node at (0,2.8){$C_e$}; 
          \node at (-7,2.8){$C_v$}; 

\end{tikzpicture}
\caption{  $\DG$-structure on an embedded graph. }
\label{fig:F2}
\end{figure}

	Similarly, let $J_e=A'\sqcup B'\subset C_e$ where $A'$ is a small closed arc around the
	image of $h$ in $C_e$ (denoted $a'$) and $B'$ is a small closed arc around the image of $h'$
	in $C_e$ (denoted $b'$). Note that the orientation torsors of $C_v$ and $C_e$ are identified
	by parallel transport along $\gamma(|h|)$. We now consider the space $\Phi$ formed by all
	orientation preserving homeomorphisms
	\[
	\phi: (C_v, J_v) \lra (C_e, J_e)
	\quad \text{s.t. } \phi(A)\subset A', \, \phi(B)\subset B'.
	\]
	These requirements define $\phi$ uniquely up to an isotopy (i.e., $\Phi$ is connected).
	Further, by Proposition \ref{prop:homeo-contr}, $\Phi$ is contractible. Consider now the
	unramified covering
	\[
	\pi: \Homeo^G(C_v, C_e) \lra \Homeo(C_v, C_e).
	\]
	As $\Phi$ is contractible, $\pi^{-1}(\Phi)$ is a disjoint union of components homeomorphic
	to $\Phi$.  We now show that the $\GG$-structure on $S$ defines a canonical choice of a
	component $\widetilde\Phi\subset\pi^{-1}(\Phi)$. By Theorem \ref{thm:CG=DG},
	$\widetilde\Phi$ will define a structure morphism $H(v)\to \{e,e'\}$ in $\Gc$. 

	Choose a Riemannian metric on $S$ (a contractible choice). This reduces the structure group
	of $S$ from $\GG$ to $G$. Let $T: C_v\to C_e$ be the Riemannian parallel transport along
	$\gamma(|h|)$.  It lifts canonically to a morphism of $\Homeo^G(S^1)$-structured circles.
	Denote this structured morphism by $\widetilde T$.  Let $\Pi$ be the space of paths
	$(T_s)_{s\in[0,1]}$ in $\Homeo(C_v, C_e)$ starting at $T$ and ending somewhere in $\Phi$.
	Any path $(T_s)\in\Pi$ defines a lift $\widetilde\Phi$ by transporting $\widetilde T$ along
	this path. Clearly, this $\widetilde \Phi$ depends only on the image of $(T_s)$ in
	$\pi_0(\Pi)$.  Note that $\pi_0(\Pi)$ is a torsor over $\ZZ = \pi_1 \,\Homeo^+(C_v, C_e)$.
	Note that $T$ sends the midpoint $b\in B$ (i.e, the image of $h$ in $C_v$) into the midpoint
	$b'\in B'$ (i.e., the image of $h'$ in $C_e$). Therefore we have a distinguished component
	$\Pi_0$ of $\Pi$ containing paths $(T_s)$ such that $T_s(B)$ meets $B'$ for all $s$. The
	transport of $\widetilde T$ along any path from $\Pi_0$ gives the component
	$\widetilde\Phi$. This concludes the argument. 
\end{proof}

\begin{exas}
\begin{enumerate}
	\item Suppose $\DG=\Lambda$. If $\Gamma$ be embedded into an oriented surface $S$, then it is classical that $\Gamma$
		is canonically a ribbon graph. Explicitly, each $H(v)$ has a cyclic order from the embedding
		into the oriented circle $C(T_{\gamma(v)} S)$, while the set of half-edges of any edge
		has a unique cyclic order since it has cardinality $2$. 
			
	\item Suppose $\DG=\Xi$.
		 Let $\Gamma$ be embedded into an unoriented surface $S$. Then $\Gamma$
		naturally admits the structure of a M\"obius graph.  The set 		$H(v)$ of half-edges incident to $v$
		inherits a dihedral order from the embedding into the circle $C(T_{\gamma(v)} S)$.
		The set $\{h,h'\}$ of half-edges of an edge $e$ is made into a dihedral set by looking at
		the orientation cover $\varpi: \widetilde \to S$: we define the orientation torsor $O(\{h,h'\})$
		of this set to be the set of sections of $\pi$ over the image, under $\gamma$, of the interior of $|e|$. 
		 Finally, we have to provide a lift of 
		the incidence map $i: H(s(h)) \to \{h,h'\}$ to a morphism of dihedrally ordered
		sets. Note that, by the path lifting property of the orientation cover $\pi$, 
		the orientation torsors of $H(s(h))$ and $\{h,h'\}$ are canonically
		identified. For every choice of orientation, there is now a unique linear
		order on $i^{-1}(h')$ compatible with the corresponding cyclic order of
		$H(s(h))$. These two linear orders are opposite providing $i$ with the
		structure of a morphism in ${\Xi}$.
			
	\item  Let $\DG={\Lambda}_\infty$. 
		Assume that $\Gamma$ is embedded into a framed surface $S$. This means, in
		particular, that $S$ is oriented so that we have a principal $GL^+(2,\RR)$-bundle $\Fr^+_S$ of positive frames on $S$. 
		Further, we are given a principal $\widetilde{GL^+(2,\RR)}$-bundle $\widetilde\Fr^+_S$ covering $\Fr^+_S$. 
		Thus, each fiber of $\widetilde\Fr^+_S$ is a universal cover of the corresponding fiber of $\Fr^+_S$.
		Let
		\[
		B^+ \,\,=\,\, \biggl\{
		\begin{pmatrix}
		a&b\\0&c
		\end{pmatrix}\,\in GL^+(2,\RR)\bigl\| \,\, a,c >0\biggr\}. 
		\]
		Then $B^+\backslash GL^+(2,\RR)=S^1$ is the circle of directions of $\RR^2$. Therefore 
		$B^+\backslash \Fr^+_S=C(T_S)$ is the circle bundle of directions of $T_S$. Further, $B^+$, being
		contractible, lifts canonically to a subgroup $\widetilde B^+\subset\widetilde{GL^+(2,\RR)}$.
		The quotient $\widetilde B^+\backslash \widetilde{GL^+(2,\RR)}\simeq\RR$ is the universal covering of
		$S^1$. So each fiber $\widetilde C(T_xS)= \widetilde B^+\backslash \widetilde\Fr^+_{S,x}$ is an oriented
		real line, equipped with the projection $\pi_x: \widetilde C(T_xS)\to C(T_xS)$
		exhibiting it as a universal covering of $C(T_xS)$. Therefore, for each vertex $v\in\Gamma$
		the preimage of $H(v)\subset C(T_{\gamma(v)}S)$ is naturally a $\ZZ$-torsor, thus giving a paracyclic
		order on $H(v)$. Similarly for each edge $\{h,h'\}$. We leave remaining details to the reader. 
		
	\item  Assume that $\Gamma$ is embedded into a surface $S$ which has an $N$-spin
		structure. Then, similarly to (3) , we consider an $N$-fold covering 
		$\pi_x: \widetilde C_N(T_xS)\to C(T_xS)$ for each $x\in S$. Taking $\pi_{\gamma(v)}^{-1}(H(v))$,
		we make each $H(v)$ into an $N$-cyclic set, and similarly for each set $\{h,h'\}$. 
\end{enumerate}
	
\end{exas}
	
%

\subsection {Structured graphs and mapping class groups of structured surfaces}
 
Let $\DG$ be a planar crossed simplicial group, $G$ its corresponding planar Lie group, $\GG$ its
thick variant, and let $\Gc$ denote the category of $\DG$-structured sets. 
 
\begin{defi} Let $\Gamma$, $\Gamma'$ be graphs. A {\em weak equivalence} between $\Gamma$
	and $\Gamma'$ is a functor $\varphi: I(\Gamma) \to I(\Gamma')$ of incidence
	categories such that
	\begin{enumerate}
		\item $\varphi$ induces a bijection of the sets of vertices of valency $1$,
		\item the induced map $|\varphi|: |I(\Gamma)| \to |I(\Gamma')|$ is homotopy
			equivalence.
	\end{enumerate}
\end{defi}

Given a weak equivalence $\varphi: \Gamma \to \Gamma'$ of graphs we obtain an induced natural
transformation $\varphi^* I_{\Gamma'} \to I_{\Gamma}$ of incidence diagrams as follows: Given an
object $x$ of $I(\Gamma)$ corresponding to a vertex of $\Gamma$, we may represent each halfedge $e$ 
incident to $\varphi(x)$ by a path $\alpha$ in $|I(\Gamma')|$ which starts at $\varphi(x)$ and 
parametrizes the edge corresponding to $e$. The germ of the pullback of $\alpha$ at $x$ determines a
unique halfedge at $x$. This construction canonically determines a map $I_{\Gamma'}(\varphi(x)) \to
I_{\Gamma}(x)$. By a similar construction, one obtains maps $I_{\Gamma'}(\varphi(y)) \to
I_{\Gamma}(y)$ for every object $y$ corresponding to an edge of $\Gamma$ which assemble to the
desired natural transformation.

\begin{defi}\label{def:weak-eq-str-graphs}
	Let $\Gamma$, $\Gamma'$ be $\DG$-structured graphs.
	A {\em weak equivalence} $\varphi: \Gamma \to \Gamma'$ of $\DG$-structured graphs is a weak
	equivalence of underlying graphs together with a lift of the pullback morphism 
	$\varphi^* I_{\Gamma'} \to I_{\Gamma}$ to a morphism of $\G$-diagrams $\varphi^*
	\widetilde{I_{\Gamma'}} \to \widetilde{I_{\Gamma}}$.
	We define $\DGgraph$ to be the category with objects given by (compact) $\DG$-structured
	graphs without vertices of valency $2$ and morphisms given by weak equivalences.
\end{defi}

Note that this definition is analogous to that of a morphism of ringed spaces. 
Sometimes we will refer to weak equivalences of structured graphs as {\em contractions}. 

\begin{thm}\label{thm:g-graph} 
	The topological realization $|N(\DGgraph)|$ is a classifying space of $\GG$-structured surfaces so
	that we have a homotopy equivalence
	\[
		|\DGgraph| \simeq \coprod_{(S,M)} B\, \Mod^\GG(S,M),
	\]
	where the coproduct is taken over all topological types of stable marked $\GG$-structured surfaces $(S,M)$.
\end{thm}

The proof will be given in \S \ref{sub sub:decom-str}. 

\begin{defi}
Let $(S,M)$ be a stable marked surface. A {\em spanning graph} for $(S,M)$
is a graph $\Gamma$ together with an embedding $\gamma: |\Gamma|\to S$ such that:
\begin{enumerate}
\item $\gamma$ is a homotopy equivalence.
\item The restriction of $\gamma$ defines a homotopy equivalence between $\partial |\Gamma|$
(a finite set of points) and $\partial S - M$ (a finite set of open intervals). 
\end{enumerate}
We define the category $P(S,M)$ whose objects are isotopy classes of spanning graphs for $(S,M)$ and
with a unique morphism $[\Gamma] \to [\Gamma']$ if $[\Gamma']$ is obtained from $[\Gamma]$
by collapsing a forest. Note that $P(S,M)$ is in fact a poset.
\end{defi}
 
\begin{thm}\label{thm:graph-contr}
   For any stable marked surface $(S,M)$ the poset $P(S,M)$ is contractible. 
\end{thm}

The proof will be given in \S \ref{subsub:harer}. 

Assume that $(S,M)$ carries a $\GG$-structure. We then have a functor 
\begin{equation}\label{eq:rho}
	\rho: P(S,M) \lra \DGgraph
\end{equation}
which chooses for a graph $(\Gamma,\gamma)$ in every isotopy class and associates to it the
$\G$-structure on $\Gamma$ induced from the embedding $\gamma$ (Proposition \ref
{prop:emb-graph-str}). 
Further, the structured mapping class group $\Mod^{\GG}(S,M)$ acts on $P(S,M)$ and we can form the
semidirect product
\[
	\Mod^{\GG}(S,M) \ltimes P(S,M)
\]
as defined in Definition \ref{defi:semidirect}. The proof of Theorem \ref{thm:g-graph} amounts to
the statement that the functor $\rho$ extends to an equivalence of categories
\[
	\widetilde{\rho}: \Mod^{\GG}(S,M) \ltimes P(S,M) \lra \DGgraph,
\]
see Theorem \ref{thm:decomp}.

\subsection{Augmented structured trees and operads}

The definition of a compact graph given in \S \ref{subsec:structuredgraphs} includes $0$-valent
vertices which can be interpreted as endpoints of external halfedges. This concept is suitable for studying
a fixed structured marked surface with boundary. However, when interpreting such structured surfaces as bordisms which
can be composed, we need to include combinatorial analogues of collar neighborhoods near the
boundary. We realize this by using noncompact structured graphs equipped with an augmentation map
for every external halfedge.

\begin{defi}
An {\em augmented $\DG$-structured graph} is a possibly noncompact $\DG$-structured graph $\Gamma$ equipped with
\begin{itemize}
	\item for every external halfedge $e \in H(v)$, an {\em augmentation map} $\varphi_e: H(v)
		\to [1]$ of $\DG$-structured sets, satisfying $\varphi_e^{-1}(\varphi(e)) = \{e\}$,
\end{itemize}
such that every vertex of $\Gamma$ has valency $\ge 2$.
\end{defi}

Given an augmented $\DG$-structured graph $\Gamma$, we obtain a partition of the set of external
halfedges into {\em incoming} ($\varphi_e(e) = 1$) and {\em outgoing} ($\varphi_e(e) = 0$).
We typically enlarge the incidence diagram $I(\Gamma) \to \G$ of $\Gamma$, also including, for every
external halfedge $e$, the corresponding morphism $\varphi_e$. A {\em contraction} of augmented
$\DG$-structured graphs is a contraction of the underlying $\DG$-structured graphs which commutes with
the augmentations.
 
\subsubsection{Augmented structured intervals} 
\label{subsec:interval}

A noncompact {\em interval} is a connected graph with two external halfedges and all vertices of valency $2$:   
\[
	\xymatrix{
		\ar@{-}[r] & \bullet \ar@{-}[r] & \bullet \ar@{-}[r]  & \cdots \ar@{-}[r] & \bullet
		\ar@{-}[r]  & 
	}
\]
Consider an augmented $\DG$-structured interval $\Gamma$ with one vertex $v$, $H(v) = \{e,f\}$ and incidence diagram given by
\begin{equation}\label{eq:interval}
    [1] \overset{\varphi_e}{\lla} H(v) \overset{\varphi_f}{\lra} [1]
\end{equation}
where $e$ is incoming and $f$ is outgoing. Given another augmented structured interval 
\[
  [1] \overset{\varphi_{e'}}{\lla} H(v') \overset{\varphi_{f'}}{\lra} [1]
\]
with $e'$ incoming and $f'$ outgoing, we can concatenate with \eqref{eq:interval} to obtain the
augmented structured interval
\[
  [1] \overset{\varphi_e}{\lla} H(v) \overset{\varphi_f}{\lra} [1] \overset{\varphi_{e'}}{\lla} H(v') \overset{\varphi_{f'}}{\lra} [1]
\]
with one internal edge formed by $\{f,e'\} \cong [1]$. Further, there is a contraction from this
interval to the interval 
\[
  [1] \overset{\varphi_{e}}{\lla} H(v) \overset{\varphi_{f'} \varphi_{e'}^{-1} \varphi_{f}}{\lra} [1].
\]

We define $\Tr_{\DG}(1)$ to be the category of augmented $\DG$-structured intervals with external
halfedges labelled by $0$ (outgoing) and $1$ (incoming) and with contractions as morphisms. We denote 
by $\pi_0\Tr_{\DG}(1)$ the set of isomorphism classes in the groupoid completion of $\Tr_{\DG}(1)$ so 
that two objects $\Gamma$,$\Gamma'$ are in the same class if and only if there exists a chain of zigzag 
contractions 
\[
	\Gamma \lla \Gamma_1 \lra \Gamma_2 \lla \dots \lra \Gamma'.
\]
The above concatenation operation endows the set $\pi_0 \Tr_{\DG}(1)$ with the structure of a
monoid.

\begin{prop} The monoid $\pi_0 \Tr_{\DG}(1)$ can be canonically identified with the group $\Gen_1^0$.  
\end{prop}
\begin{proof} Given an augmented $\DG$-structured interval as in \eqref{eq:interval}, note that
	$\varphi_e$ and $\varphi_f$ are isomorphisms in $\G$. Using $\varphi_e$ to identify $H(v)$
	with $[1]$, we obtain that the augmented $\DG$-strutured interval is isomorphic to \[ [1]
	\overset{\id}{\lla} [1] \overset{g}{\lra} [1] \] where $g$ is an automorphism of $[1]$ which
	induces the identity on the underlying set $\{0,1\}$. Therefore, we have $g \in \Gen_1^0$
	where $\Gen_1^0$ denotes the kernel of the map $\lambda: \Gen_1 \to S_2$ from Proposition
	\ref{prop:forget}. It is immediate to verify that this identitication is compatible with the
	group law on $\Gen_1^0$.
\end{proof}

Since, for a planar crossed simplicial groups $\DG$, we have $\Gen_1^0 \cong \Gen_0$, we obtain the
following result.

\begin{cor} For a planar crossed simplicial group, the monoid $\pi_0 \Tr_{\DG}(1)$ 
	is canonically isomorphic to the group $\Gen_0$.
\end{cor}
 
%
%
   
\subsubsection{Augmented structured trees}

The relation between twisted $\Gen_0$-actions on categories or algebras on one side, and
$\DG$-structured nerves or Hochschild complexes on the other, can be given a more conceptual
explanation. We assume familiarity with the language of operads, referring to \cite{loday-vallette}
for general background. 

Let $\Tr_{\DG}(n)$ denote the category of augmented $\DG$-structured trees with
external halfedges labelled by $0,1,\dots,n$ such that the halfedge $0$ is outgoing and all
remaining halfedges are incoming. The morphisms in $\Tr_{\DG}(n)$ are given by contractions of
augmented $\DG$-structured graphs. We define $P_{\DG}(n)$ to be the set $\pi_0 \Tr_{\DG}(n)$ of
isomorphism classes in the groupoid completion of $\Tr_{\DG}(n)$. Note that $P_{\DG}(1)$ is the
monoid of $\DG$-structured intervals from \S \ref{subsec:interval}. Given augmented $\DG$-structured trees
$\Gamma$, $\Gamma'$, an outgoing halfedge $e \in H(v)$ of $\Gamma$, and an incoming halfedge $e' \in
H(v')$ of $\Gamma'$, we obtain a canonical diagram of $\DG$-structured sets
\[
	H(v) \overset{\varphi_e}{\lra} [1] = \{e,e'\} = [1] \overset{\varphi_{e'}}{\lla} H(v')
\]
so that we can concatenate $\Gamma$ and $\Gamma'$ forming an augmented $\DG$-structured tree with
internal edge $\{e,e'\}$. Similarly, given an augmented $\DG$-structured tree $\Gamma$, an outgoing (resp.
incoming) halfedge $e$, and $g \in P_{\DG}(1) = \Gen_1^0$, we can postcompose the augmentation
$\varphi_e$ with $g$ (resp. $g^{-1}$). This equips the family of sets $\{P_\DG(n)\}_{n\geq 1}$
with operations
\[
 m_{a_1, \dots, a_n}:  P_\DG(n) \times P_\DG(a_1)\times \cdots \times P_\DG(a_n) \lra 
 P_\DG(a_1+\cdots + a_n), \quad n, a_1, \dots, a_n \geq 1. 
\]

\begin{prop} \label{prop:Gc-operad} Let $\DG$ be a crossed simplicial group.
\begin{enumerate}
	\item The action of $S_n$ on $P_\Gc(n)$ by relabelling the incoming external halfedges and
		the maps $\{ m_{a_1, \dots, a_n} \}$ make $P_\DG= \{P_\DG(n)\}_{n\geq 1}$ an operad in the
		category of sets. 
	\item Assume that $\DG$ is planar. Then algebras over $P_\DG$ are precisely monoids with a
		twisted action of $\Gen_0$ in the sense of \S \ref{section:semiconstant}.
\end{enumerate}
\end{prop}
\begin{proof}
	Part (1) is clear. We show (2). Given an augmented $\DG$-structured tree $\Gamma$ in
	$\Tr_{\DG}(n)$ with augmented incidence diagram $I(\Gamma) \to \G$, we first argue that we
	can contract all internal edges to obtain a corolla. To this end, we note that, given an
	internal edge $\{e,e'\}$ incident to vertices $v$ and $v'$, the corresponding diagram
	\[
		H(v) \lra \{e,e'\} \lla H(v') 
	\]
	has a limit in $\G$ whose underlying set is given by $(H(v)\setminus \{e\}) \cup (H(v')
	\setminus \{e'\})$. This holds for all planar crossed simplicial groups and can be 
	verified using the topological model of $\DG$ from \S \ref{sec:topmodel}.
	We can use the cone diagram of this limit in $\G$ to produce a morphism of
	$\DG$-structured graphs $\Gamma \to \Gamma'$ which collapses the edge $\{e,e'\}$. After
	collapsing all internal edges of $\Gamma$, we obtain an augmented $\DG$-structured 
	$n$-corolla $\Gamma_0$. Using the labels we can identify the set $H(v)$ of halfedges
	incident to the single vertex of $\Gamma_0$ with $\{0,1,\dots,n\}$. We can find an
	isomorphism of $\DG$-structured sets $f: (\{0,1,\dots,n\}, \O) \cong [n]$ and since, for planar
	crossed simplicial groups, $\Gen_n$ acts (simply) transitively on $\Hom_{\DG}([0],[n])$ we
	may assume that $f$ maps $0$ to $0$. Therefore, after permuting the labels $1,2, \dots, n$,
	we may assume that $H(v) = [n]$. Every morphism $[n] \to [1]$ such that $\varphi^{-1}(0) =
	\{0\}$ (corresponding to the outgoing halfedge labeled $0$) can be written as a composite 
	$\varphi g$ where $\varphi: [n] \to [1]$ is the morphism in $\Delta$ which maps $0$ to $0$
	and all remaining elements to $1$, and $g \in \on{Stab}(0)
	\subset \Gen_n$. Therefore, we can assume that the augmentation $H(v) = [n] \to [1]$ of the
	single outgoing halfedge $0$ is given by $\varphi$. 
	For every $1 \le i \le n$, the subset of $\Hom_{\DG}([n],[1])$ given by those morphisms such
	that $\varphi^{-1}(1) = \{i\}$ forms a torsor under the action of $\Gen_1^0 \cong \Gen_0$. 
	Note that this action of $\Gen_0$ is precisely the operadic action of the monoid
	$P_{\DG}(1)$ on the incoming halfedges. Therefore, the collection of isomorphism classes of
	labelled augmented $\DG$-structured $n$-corollas forms a torsor under the action of 
	\[
		\underbrace{P_{\DG}(1) \times P_{\DG}(1) \times \dots \times P_{\DG}(1)}_n \times S_{n}.
	\]
	We will now trivialize this torsor, for every $n$, by choosing a specific $\DG$-structured
	$n$-corolla in such a way that these chosen $n$-corollas are closed under operadic
	composition (for $n>1$). To this end, choose a duality functor $D: \DG \to \DG^{\op}$ and consider the
	diagram
	\[
		\xymatrix{ [1] \ar[drr]_{\{0,1\}}& [1] \ar[dr]^{\{1,2\}} & \dots & [1] \ar[dl]^{\{n-1,n\}}\\
			& & [n] & \\
			& & [1]\ar[u]^{\{0,n\}} & 
		}
	\]
	in $\Delta \subset \DG$ and apply $D$ to obtain the incidence diagram of a labeled augmented
	$\DG$-structured $n$-corolla which we define to be our chosen trivialization. It is easy to
	verify that, due to the functoriality of $D$, the collection of corollas thus obtained is
	closed under operadic composition. This exhibits a copy of the associative operad $\on{Ass}$
	in $P$ such that every element of $P$ can be uniquely expressed as an element of $\on{Ass}$
	precomposed with elements of the monoid $P_{\DG}(1)$ acting on the incoming halfedges. An
	explicit computation of the action of $P_{\DG}(1)$ on the outgoing halfedges of $\on{Ass}$
	in terms of the actions on the incoming halfedges implies the claim that $P$ parametrizes
	algebras with twisted $\Gen_0$-action. 
\end{proof}

\begin{rem} Proposition \ref{prop:Gc-operad} can be enhanced by considering the geometric
	realization of the categories $\Tr_\DG(n)$ instead of the set $P_{\DG}(n)$ of its connected
	components. In this way, we obtain an operad $\N\Tr_\DG$ in the category of topological
	spaces which describes $A_\infty$-monoids with a twisted coherent action of $\Gen_0$.  
\end{rem}

\begin{rem} The results of this section suggest that augmented $\DG$-structured {\em graphs} can be
	interpreted in modular operadic terms generalizing the well-known approaches to
	$2$-dimensional topological field theories in the cyclic and dihedral cases (cf. \cite{costello-TFT, braun}). 
	We leave the systematic treatment of these generalizations to future work.
\end{rem}

\def\Tess{\on{Tess}}
\def\Bar{{\on{Bar}}}
\def\bar{{\on{bar}}}
 
\subsection{Structured graphs and structured moduli spaces}
  
  \def\Cell{{\on{Cell}}}
  
  \subsubsection{Orbicell decompositions}
  
  By a $d$-{\em cell} we mean a topological space $\sigma$ homeomorphic to
  an open ball  $B^d \subset \RR^d$. By a $d$-{\em orbicell} we mean a topological groupoid equivalent
  to $D\bbs\sigma$ where $\sigma\simeq B^d$ is a $d$-cell and $D$ is a discrete group acting on $\sigma$ via a
  homomorphism $\phi: D\to\Homeo(\sigma)$ whose image is identified with  a finite group  
   of linear transformations of $B^d$. Note that we do not require $\Ker(\phi)$ to be finite.
   
   \begin{defi}
   Let $X$ be a topological groupoid. An {\em orbicell decomposition} of $X$ is a finite filtration
   \[
   \Xc = \bigl( \Xc_0\subset\Xc_1 \subset\cdots\subset \Xc_n = X)
   \]
   of $X$ by closed topological subgroupoids $\Xc_d$ such that:
   \begin{enumerate}
   \item[(1)] Each $\Xc_d$ is a full subcategory in $X$.
   
   \item[(2)] Each component of $\Xc_d-\Xc_{d-1}$ is a disjoint union of $d$-orbicells.
   \end{enumerate}
   
   \end{defi}
  
  \begin{exa}\label{ex:regular-cell}
  Let $Y$ be a topological space with a cell decomposition $\Yc$. Let $D$ be a discrete
  group acting on $Y$, preserving $\Yc$ and such that for each cell $\sigma$ of $\Yc$ the stabilizer $\on{Stab}(\sigma)$
 acting on $\sigma$, makes $\on{Stab}(\sigma)\bbs\sigma$ into an orbicell. Then the groupoid
 $D\bbs Y$ has an orbicell decomposition into orbicells $\on{Stab}(\sigma)\bbs\sigma$ where $\sigma$ runs
 over orbits of $D$ on the set of cells of $\Yc$. 
  \end{exa}
  
  Recall that for ordinary spaces (not groupoids) $Y$ there is a concept of a {\em regular cell decomposition}
  which is a cell decomposition $\Yc$ such that the closure of each $d$-cell $\sigma$ of $\Yc$ is
  homeomorphic to a closed ball in $\RR^d$ (whose boundary is therefore a $(d-1)$-sphere with
  further cell decomposition induced by $\Yc$). We were not able to find in the literature an
  intrinsic analog of this concept for orbicell decompositions. The following concept will be sufficient
  for our purpose.
  
  \begin{defi}
  An orbicell decomposition $\Xc$ of a topological groupoid $X$ is called {\em quotient-regular},
  if there is an equivalence of $X$ with a quotient groupoid $D\bbs Y$ and a regular $D$-invariant
  cell decomposition $\Yc$ of $Y$ as in Example \ref{ex:regular-cell}, which induces $\Xc$. 
  \end{defi}
  
  For a quotient-regular cell decomposition $\Xc$ we have a category $\Cell(\Xc)$ whose objects
  are cells of $\Yc$ and morphisms are induced by action of $D$ and inclusion of cells.
  In other words, $\Cell(\Xc)$ is the semi-direct product of $D$ and the poset of cells of $\Yc$.
  Note that isomorphism classes of objects of $\Cell(\Xc)$ are in bijection with $D$-orbits on cells
  of $Y$, i.e., with distinct orbicells in $\bigsqcup \Xc_d- \Xc_{d-1}$. The following is then straightforward
  and follows from Proposition \ref{prop:nerve-classical}. 
  
  \begin{prop}\label{prop:orbicell-cell}
  If $\Xc$ is a quotient-regular orbicell decomposition of $X$, then the classifying space
  of the (non-topological) category $\Cell(\Xc)$ is homotopy equivalent to the classifying space of $X$ as a topological
  category. \qed
  \end{prop}

  Note that such a statement cannot be true without some regularity assumptions already in the case of
  ordinary (not orbi) cell decompositions as can be seen by considering $S^2$ decomposed into $\RR^2$ and $\infty$. 
  
\subsubsection{Decompositions of the structured moduli spaces}\label{sub sub:decom-str}

We now prove the following fact which, in virtue of Proposition 
\ref {prop:orbicell-cell}, implies Theorem \ref{thm:g-graph}. 

\begin{thm} \label{thm:decomp}
Consider the orbifold
\[
\Mc^{G_\conf} \,\,=\,\,\coprod_{(S,M)} \Mc^{G_\conf}(S,M)
\]
where $(S,M)$ runs over all possible topological types of stable marked surfaces. Then 
$\Mc^{G_\conf}$ has a quotient-regular orbicell decomposition $\Xc$ such that we have an equivalence
of categories $\Cell(\Xc)\simeq \DGgraph$. 
\end{thm}

\begin{proof} \underbar{(a) Oriented case } $G=SO(2)$. In this case, the statement (that the usual
	moduli space of stable marked {\em Riemann surfaces} has an orbicell decomposition labelled
	by ribbon graphs) is well known, see, e.g., \cite{arbarello, penner:book} and references
	therein. We indicate here the main steps which we later analyze to deduce the general case.
 
	First, each component of $\Mc^{SO(2)_\conf}$ is a quotient groupoid 
	\[
	 \Mc^{SO(2)_\conf}(S,M) \,\,=\,\, \Mod^+(S,M)\Bbs\Teich^+(S,M).
	\]
	Second, $\Teich^+(S,M)$  has a $\Mod^+(S,M)$-invariant regular cell decomposition
	(triangulation)  constructed by Harer \cite{harer} and described in Appendix \ref
	{subsub:harer}.  Denote this cell decomposition by $\Yc$.  The simplices of $\Yc$ correspond
	to tesselations of $(S,M)$ or, dually, to isotopy classes of spanning graphs for $(S,M)$.
	Inclusions of simplices correspond to coarsenings of tesselations or, dually, to contractions of
	spanning graphs.  Any spanning graph for $(S,M)$ carries an induced ribbon structure.
	Therefore, orbicells in the quotient orbicell decomposition, which we denote $\Xc_{S,M}$,
	form a full subcategory in $\Lambda-\on{Graph}$ with objects those graphs that appear as
	spanning graphs for $(S,M)$.  Finally, each ribbon graph $\Gamma$ gives rise to an oriented
	marked surface $(S,M)$ obtained by ``thickening" $\Gamma$, so that $\Gamma$ becomes a
	spanning graph for $(S,M)$, see, e.g., \cite{HSS-triangulated} \S 3.3.4. This shows that for
	$\Xc=\coprod_{(S,M)} \Xc_{S,M}$, we have $\Cell(\Xc)\simeq {\Lambda}-\on{Graph}$. 

	Note that Costello \cite{costello-dual} introduced another topological groupoid $\Nc$ with
	an orbicell decomposition labelled by ${\Lambda}-\on{Graph}$ which is homotopy equivalent,
	but not homeomorphic to $\Mc^{SO(2)_\conf}$. 

	\vskip .3cm

	\underbar{(b) Unoriented case} $G=O(2)$. It is treated similarly to the oriented case, using
	the identification \[ \Mc^{O(2)_\conf}(S,M) \,\,=\,\,\Mod(S,M) \Bbs \Teich(S,M) \] and the
	triangulation $\Yc$ of $\Teich(S,M)$ described in \S \ref{subsub:harer}. As before, the
	simplices of $\Yc$ correspond to triangulations of $(S,M)$ and give M\"obius graphs as
	$\Mod(S,M)$-orbits. Further, any M\"obius graph $\Gamma$ can be ``thickened" producing an
	unoriented surface $(S,M)$ in which it is spanning. So, for the resulting quotient orbicell
	decomposition $\Xc$ of $\Mc^{O(2)_\conf}$ we have $\Cell(\Xc)\simeq {\Xi}-\on{Graph}$. 

	Note that an analog of Costello's approach \cite{costello-dual} for Klein surfaces was developed
	by Braun \cite{braun}. 

	\vskip .3cm

	\underbar{(c) General case:} $G$ an arbitrary planar Lie group. Depending on whether $G$
	preserves orientation or not, we consider one of the two projections
	\[
	\pi: \Mc^{G_\conf} \lra \Mc^{SO(2)_\conf}, \quad  \pi: \Mc^{G_\conf} \lra \Mc^{O(2)_\conf}
	\]
	obtained by forgetting the extra structure. In each case, the projection $\pi$ is an
	unramified covering of orbifolds, so the preimage of any orbicell in the target is a
	disjoint union of orbicells in the source.  We thus obtain an orbicell decomposition of
	$\Mc^{G_\conf}$ which we now analyze. We sketch the case when $G$ does not preserve the
	orientation, the other case being similar but easier. 
	Denote by $\Xc$ the orbicell decomposition of $\Mc^{O(2)_\conf}$ and by $\Xc^G$ the preimage
	decomposition of $\Mc^{G_\conf}$. Since any $G_\conf$-structured marked surface is
	$\GG$-structured, any tessellation $\Pc$ of the surface makes the dual graph of $\Pc$ into
	a $\DG$-structured graph. We obtain therefore a commutative diagram of categories
	\[
	\xymatrix{
	\Cell(\Xc^G)
	\ar[d]_{q_G} \ar[r]^{\pi} & \Cell(\Xc)
	\ar[d]^q
	\\
	\DGgraph \ar[r]_\varpi & {\Xi}-\on{Graph}
	}
	\]
	where $q$ is an equivalence by the case (b) above. We prove that $q_G$ is an equivalence as well.
	For this, we note that any morphism $\Gamma \to \Gamma'$ in $\DGgraph$ can be expressed
	uniquely as the composition of a contraction of a set of $I$ of edges of $\Gamma$ (which
	results in a canonical contraction morphism $\Gamma \to \Gamma_I$) followed by an isomorphism 
	$\Gamma_I \to \Gamma'$. This implies that, assuming $\Hom_{\DGgraph}(\Gamma,\Gamma') \ne
	\emptyset$, any
	fiber of the map 
	\[
		\Hom_{\DGgraph}(\Gamma,\Gamma') \lra
		\Hom_{{\Xi}-\on{Graph}}(\varpi(\Gamma),\varpi(\Gamma'))
	\]
	is acted upon simply transitively by the group $\Aut(\Gamma'/\varpi(\Gamma'))$ of automorphisms of $\Gamma'$ which induce
	the identity on $\varpi(\Gamma')$. A similar analysis of the morphisms in $\Cell(\Xc^G)$
	implies that the fibers of the map 
	\[
		\Hom_{\Cell(\Xc^G)}(\sigma,\sigma') \lra
		\Hom_{\Cell(\Xc)}(\pi(\sigma),\pi(\sigma'))
	\]
	are torsors for the group $\Aut(\sigma'/\pi(\sigma'))$ of automorphisms of $\sigma'$
	inducing the identity on $\pi(\sigma')$.
	Since we already know that the map $q$ is fully faithful, to show that $q^G$ is fully
	faithful, it therefore suffices to show that the induced map 
	\[
		\Aut(\sigma'/\pi(\sigma')) \lra \Aut(q^G(\sigma')/q^G(\pi(\sigma')))
	\]
	is bijective. Showing this statement and the essential surjectivity of $q^G$ amounts to  
	showing that $q^G$ induces an equivalence of all strict fiber categories of $\pi$ and
	$\varpi$. This is shown in the following lemma.
\end{proof}
 
\begin{lem}
Let $(S,M)$ be a stable marked surface and $\Gamma$ be a spanning graph for $(S,M)$.
Then the functor
\[
	q_\Gamma: \Gstr(S-(M\cap S^\circ) )\to \DG\on{Str}(\Gamma)
\]
is an equivalence of categories. 
\end{lem}
 
\begin{proof} We first show that $q_\Gamma$ is essentially surjective. Let $F_\Gamma$ be a
	$\DG$-structure on $\Gamma$.  By Theorem \ref{thm:CG=DG}, $F_\Gamma$ can be seen as a datum,
	for each $x\in\Gamma\subset S'$, of a $\Homeo^G(S^1)$-structure on the circle of directions
	$C_x(S)$. Moreover, these structures are compatible as $x$ moves in $\Gamma$, thus providing
	a $\Homeo^G(S^1)$-structure on $C(T_S)|_\Gamma$, the restriction to $\Gamma$ of the circle
	bundle $C(T_S)$. Since $\Gamma$ is homotopy equivalent to $S-M$ and hence to $S-(M\cap
	S^\circ)$, we get a unique, up to isomorphism, $\Homeo^G(S^1)$-structure on $C(T_{S-(M\cap
	S^\circ)})$. Finally, since $\GG\to\Homeo^G(S^1)$ is a homotopy equivalence, we obtain from
	this a $\GG$-structure on $S-(M\cap S^\circ)$. Thus $q_\Gamma$ is essentially surjective. 
 
	 Second, we show that $q_\Gamma$ induces bijections on Hom-sets. Because $\Gstr(S)$ is a
	 groupoid, it is enough to show that for each $\GG$-structure $F$ on $S-(M\cap S^\circ)$ the
	 induced map of automorphism groups
	 \[
		 q_*: \Aut(F)\lra\Aut(F_\Gamma)
	 \]
	 is an isomorphism. Note that $q_*$ actually comes (by passing to global sections) from a
	 homomorphism of sheaves of groups on $\Gamma$ formed by local isomorphisms:
	 \[
		  \underline q_*: \underline \Aut(F)\lra\underline \Aut(F_\Gamma). 
	 \]
	 So it is enough to prove that $  \underline q_*:$ is an isomorphism of sheaves. By
	 Proposition \ref{prop:G-structures}(b1), $\underline\Aut(F)=\underline K^\orr$ is the local
	 system obtained by the orientation twist of $K=\Ker\{\GG\to GL(2,\RR)\}$. Now note that $K$
	 can also be identified with $\Ker\{\Gen_n\to D_{n+1}\}$ for any $n$ by Theorem
	 \ref{thm:crossed-groups}(b2). An automorphism of the $\DG$-structure $F_\Gamma$ (identical
	 on $\Gamma$) is therefore a section of a local system $\Lc$ on $\Gamma$ whose stalk at each
	 point is isomorphic to $K$. The final identification of $\Lc$ with $\underline K^\orr$ and
	 of $q_*$ with the identity of  $\underline K^\orr$ is left to the reader. 
\end{proof}


\newpage
\section{2-Segal $\Delta\Gen$-objects and invariants of $\GG$-structured surfaces}
   
\subsection{The 1- and 2-Segal conditions}\label{subset:simplicial-cyclic}

Let $\Delta$ be the {\em simplex category} of finite nonempty linearly ordered sets of the form $[n] = \{0,1,\dots,n\}$ with
morphisms given by monotone maps. Given categories $I$ and $\Cb$, we introduce the notation
\begin{align*}
	\Cb^I & = \Fun(I, \Cb), & \Cb_I & = \Fun(I^{\op}, \Cb)
\end{align*}
for the categories of covariant and contravariant functors.
The objects of $\Cb_{\Delta}$ are called {\em simplicial objects} in $\Cb$. 
If $\Cb$ admits limits, then we can use the formalism of Kan extension, to enlarge the domain of
definition of a simplicial object from $\Delta$ to $\Set_{\Delta}$. More precisely, let 
\[
	\Upsilon: \Delta^{\op} \to (\Set_{\Delta})^{\op}
\]
be the opposite of the Yoneda embedding. Then there exists an adjunction
\[
	\Upsilon^*: \Cb_{\Set_\Delta} \longleftrightarrow \Cb_{\Delta} : \Upsilon_*
\]
where, for a simplicial set $K$ and a simplicial object $X$, we have the explicit formula
\[
	\Upsilon_* X (K) = \pro_{\{\Delta^n \to K\}}^{\Cb} X_n
\]
given by a limit in $\Cb$ over the category of simplices of $K$. We introduce the notation
$(K,X) := \Upsilon_* X (K)$ and refer to it as the {\em object of $K$-membranes in $X$}.
Assume further, that $\Cb$ is equipped with a combinatorial model structure. Then we can derive the
above adjunction to obtain
\[
	L\Upsilon^*: \Ho(\Cb_{\Set_\Delta}) \longleftrightarrow \Ho(\Cb_{\Delta}): R \Upsilon_*
\]
where, as above, we have an explicit formula 
\[
	R \Upsilon_* X(K) \cong \hopro_{\{\Delta^n \to K\}}^{\Cb} X_n
\]
in terms of a homotopy limit in $\Cb$. We call $(K,X)_R := R \Upsilon_* X(K)$ the {\em derived
object of $K$-membranes in $X$}.

We will be particularly interested in membranes parametrized by the simplicial sets introduced in the following
examples.
\begin{exa} 
	 For $n \ge 1$, we denote by $I[n] \subset \Delta^n$ the simplicial subset corresponding to the subdivided
	 interval
	\[
	\begin{tikzpicture}[>=stealth,scale=1.5]

	\begin{scope}[ >=stealth]

	\coordinate (A0) at (0,0);
	\coordinate (A1) at (1,0);
	\coordinate (B1) at (2,0);
	\coordinate (DOTS) at (2.5,0);
	\coordinate (A2) at (3,0);
	\coordinate (A3) at (4,0);

	\begin{scope}[decoration={
	    markings,
	    mark=at position 0.55 with {\arrow{>}}}
	    ] 
	\draw[postaction={decorate}] (A0) -- (A1);
	\draw[postaction={decorate}] (A1) -- (B1);
	\draw[postaction={decorate}] (A2) -- (A3);
	\end{scope}

	{\scriptsize
	\draw (A0) node[anchor=south] {$0$};
	\draw (A1) node[anchor=south] {$1$};
	\draw (B1) node[anchor=south] {$2$};
	\draw (DOTS) node {$\cdots$};
	\draw (A2) node[anchor=south] {$n-1$};
	\draw (A3) node[anchor=south] {$n$};
	}
	\fill (A0) circle[radius=1pt];
	\fill (A1) circle[radius=1pt];
	\fill (B1) circle[radius=1pt];
	\fill (A2) circle[radius=1pt];
	\fill (A3) circle[radius=1pt];

	\end{scope}

	\end{tikzpicture}
	\]
	which can be more formally described as the pushout of simplicial sets
	\[
		\textstyle
		\Delta^1 \coprod_{\Delta^0} \Delta^1 \coprod_{\Delta^0} \dots \coprod_{\Delta^0}
		\Delta^1.
	\]
\end{exa}
\begin{exa}
	 Let $P_{n+1}$ be a convex plane $(n+1)$-gon with the set of vertices labelled
	 counterclockwise by $M = \{0,1,\cdots , n\}$. Let $\Tc$ be any triangulation of $P_{n+1}$ 
	 with vertex set $M$. A geometric triangle $\sigma$ of $\Tc$ with vertices $i,j,k$ gives
	 rise to a unique full simplicial subset $\Delta^2 \subset \Delta^n$ with vertices $\{i\},
	 \{j\},\{k\}$. The union, or more formally the pushout, of the simplicial subsets of $\Delta^n$ corresponding to all
	 triangles of $\Tc$ defines a simplicial subset $\Delta^\Tc \subset \Delta^n$ whose
	 geometric realization is homeomorphic to $P_{n+1}$. For example, the two simplicial subsets
	 of $\Delta^3$ corresponding to the two triangulations of a planar square can be visualized
	 as follows:
	 \[
		\tdplotsetmaincoords{100}{170}
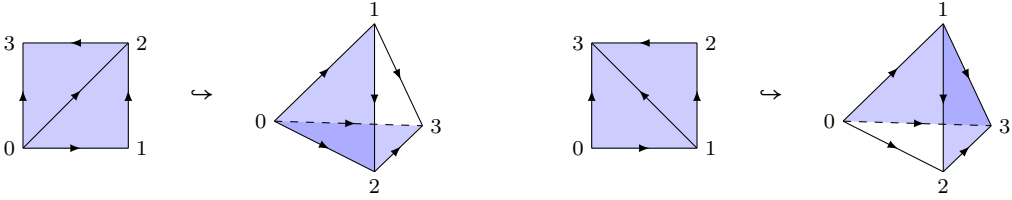
		\begin{tikzpicture}[>=latex,scale=1.0, baseline=(current  bounding  box.center)]
		\begin{scope}[scale=1.4]

		\coordinate (A0) at (0,0);
		\coordinate (A1) at (0,1);
		\coordinate (A2) at (1,1);
		\coordinate (A3) at (1,0);

		\path[fill opacity=0.4, fill=blue!50] (A0) -- (A3) -- (A2) -- (A1) -- cycle;

		\begin{scope}[decoration={
		    markings,
		    mark=at position 0.55 with {\arrow{>}}}
		    ] 
		\draw[postaction={decorate}] (A0) -- (A3);
		\draw[postaction={decorate}] (A3) -- (A2);
		\draw[postaction={decorate}] (A0) -- (A1);
		\draw[postaction={decorate}] (A2) -- (A1);
		\draw[postaction={decorate}] (A0) -- (A2);
		\end{scope}

		{\scriptsize
		\draw (A0) node[anchor=east] {$0$};
		\draw (A1) node[anchor=east] {$3$};
		\draw (A2) node[anchor=west] {$2$};
		\draw (A3) node[anchor=west] {$1$};
		}
		\draw (1.7,0.5) node {$\hookrightarrow$};
		\end{scope}

		\begin{scope}[tdplot_main_coords, >=latex, xshift=4.5cm, yshift=0.5cm]

		\coordinate (A0) at (1,-1,0);
		\coordinate (A1) at (0,1,1);
		\coordinate (A2) at (0,1,-1);
		\coordinate (A3) at (-1,-1,0);


		\path[fill opacity=0.4, fill=blue!50] (A0) -- (A2) -- (A3) -- cycle;
		\path[fill opacity=0.4, fill=blue!50] (A0) -- (A2) -- (A1) -- cycle;

		\begin{scope}[decoration={
		    markings,
		    mark=at position 0.55 with {\arrow{>}}}
		    ] 
		\draw[postaction={decorate}] (A0) -- (A1);
		\draw[postaction={decorate}] (A0) -- (A2);
		\draw[postaction={decorate}] (A1) -- (A2);
		\draw[postaction={decorate}] (A1) -- (A3);
		\draw[postaction={decorate}] (A2) -- (A3);
		\draw[postaction={decorate},dashed] (A0) -- (A3);
		\end{scope}

		{\scriptsize
		\draw (A0) node[anchor=east] {$0$};
		\draw (A1) node[anchor=south] {$1$};
		\draw (A2) node[anchor=north] {$2$};
		\draw (A3) node[anchor=west] {$3$};
		}
		\end{scope}
		\end{tikzpicture}
		\quad\quad\quad\quad
		\begin{tikzpicture}[>=latex,scale=1.0, baseline=(current  bounding  box.center)]
		\begin{scope}[scale=1.4]

		\coordinate (A0) at (0,0);
		\coordinate (A1) at (0,1);
		\coordinate (A2) at (1,1);
		\coordinate (A3) at (1,0);

		\path[fill opacity=0.4, fill=blue!50] (A0) -- (A3) -- (A2) -- (A1) -- cycle;

		\begin{scope}[decoration={
		    markings,
		    mark=at position 0.55 with {\arrow{>}}}
		    ] 
		\draw[postaction={decorate}] (A0) -- (A3);
		\draw[postaction={decorate}] (A3) -- (A2);
		\draw[postaction={decorate}] (A0) -- (A1);
		\draw[postaction={decorate}] (A2) -- (A1);
		\draw[postaction={decorate}] (A3) -- (A1);
		\end{scope}

		{\scriptsize
		\draw (A0) node[anchor=east] {$0$};
		\draw (A1) node[anchor=east] {$3$};
		\draw (A2) node[anchor=west] {$2$};
		\draw (A3) node[anchor=west] {$1$};
		}
		\draw (1.7,0.5) node {$\hookrightarrow$};
		\end{scope}

		\begin{scope}[tdplot_main_coords, >=latex, xshift=4.5cm, yshift=0.5cm]

		\coordinate (A0) at (1,-1,0);
		\coordinate (A1) at (0,1,1);
		\coordinate (A2) at (0,1,-1);
		\coordinate (A3) at (-1,-1,0);


		\path[fill opacity=0.4, fill=blue!50] (A1) -- (A2) -- (A3) -- cycle;
		\path[fill opacity=0.4, fill=blue!50] (A0) -- (A1) -- (A3) -- cycle;

		\begin{scope}[decoration={
		    markings,
		    mark=at position 0.55 with {\arrow{>}}}
		    ] 
		\draw[postaction={decorate}] (A0) -- (A1);
		\draw[postaction={decorate}] (A0) -- (A2);
		\draw[postaction={decorate}] (A1) -- (A2);
		\draw[postaction={decorate}] (A1) -- (A3);
		\draw[postaction={decorate}] (A2) -- (A3);
		\draw[postaction={decorate},dashed] (A0) -- (A3);
		\end{scope}

		{\scriptsize
		\draw (A0) node[anchor=east] {$0$};
		\draw (A1) node[anchor=south] {$1$};
		\draw (A2) node[anchor=north] {$2$};
		\draw (A3) node[anchor=west] {$3$};
		}
		\end{scope}
		\end{tikzpicture}
	\]
\end{exa}

We recall the following definitions from \cite{HSS1}, the first one being a variant of Rezk's
Segal condition \cite{rezk}.

\begin{defi}\label{def:segal} 
Let $\Cb$ be a combinatorial model category, and let $X \in \Cb_\Delta$ be a simplicial object.
\begin{enumerate}
	\item We say that $X$ is {\em 1-Segal} if, for every $n\geq 1$, the morphism
	\[
	f_n: X_n \lra (I[n], X)_R = X_1\times^R_{X_0} X_1 \times^R_{X_0} \cdots \times^R_{X_0} X_1,
	\]
	induced by the embedding $I[n] \hookrightarrow \Delta^n$, is a weak equivalence in $\Cb$. 

	\item We say that $X$ is {\em 2-Segal} if, for every $n\geq 2$ and every triangulation $\Tc$ of $P_{n+1}$,
	the morphism
	\[
	f_\Tc: X_n \lra (\Delta^\Tc, X)_R,
	\]
	induced by the embedding $\Delta^\Tc\hookrightarrow \Delta^n$, is a weak equivalence in $\Cb$. 

\end{enumerate}
\end{defi}

\begin{rem}
	Let $\Cb$ be an ordinary category with limits and colimits. We can equip $\Cb$ with the
	trivial model structure such that the weak equivalences are given by the class of isomorphisms.
	The above Segal conditions then involve underived membrane spaces require the corresponding
	morphisms $f_n$, respectively $f_{\Tc}$, to be isomorphisms in $\Cb$.
\end{rem}

\subsection{Categorified structured state sums}

Let $\DG$ be a planar crossed simplicial group. Then the category $\G$ of $\DG$-structured sets
comes equipped with an interstice duality functor $D: \G \to \G^{\op}$.
Let $\Gamma$ be a $\DG$-structured graph with $\DG$-structure $I(\Gamma) \to \G$. Postcomposing with $D$
and the opposite of the Yoneda embedding 
\[
  \Upsilon: \G^{\op} \lra (\Set_{\G})^{\op} \overset{|\DG}{\lra} (\Set_{\DG})^{\op},
\]
we obtain a diagram $I(\Gamma) \to \Set_{\DG}^{\op}$. 

\begin{defi} Given a $\DG$-structured graph $\Gamma$, the limit of the associated diagram $I(\Gamma) \to
	(\Set_{\DG})^{\op}$, denoted by $\DG^{\Gamma}$, is called the {\em $\DG$-set realization of $\Gamma$}. 
\end{defi}

\begin{prop} The $\DG$-set realization extends to a functor
	\[
		\DGgraph \lra (\Set_{\DG})^{\op}.
	\]
\end{prop}
\begin{proof} A morphism $\varphi$ between $\DG$-structured graphs $\Gamma$ and $\Gamma'$ consists of a functor
	$\varphi: I(\Gamma) \to I(\Gamma')$ of incidence categories together with 
	a natural transformation $\eta: \widetilde{I}_{\Gamma'} \circ \varphi \to
	\widetilde{I}_\Gamma$. We obtain a canonical sequence of morphisms
	\[
		\lim \Upsilon D \widetilde{I}_{\Gamma'} \lra 
		\lim \Upsilon D \widetilde{I}_{\Gamma'} \varphi \overset{\eta}{\lra}
		\lim \Upsilon D \widetilde{I}_{\Gamma}
	\]
	whose composite defines the desired morphism $\DG^{\Gamma'} \to \DG^{\Gamma}$. Here the first
	morphism is the natural morphism $\lim F \to \lim F \varphi$ which exists for any
	functor $F: I(\Gamma') \to \Cb$ provided that both limits exist. It is straightforward to
	verify the functoriality of this association.
\end{proof}

Let $\Cb$ be a combinatorial model category, and let $X: \DG^{\op} \to \Cb$ a $\G$-structured
$2$-Segal object in $\Cb$. We form the right homotopy Kan extension of $X$ along the Yoneda embedding 
$\Upsilon: \DG^{\op} \to (\Set_{\DG})^{\op}$ and take the composite with the above functors 
to obtain a canonical functor
\begin{equation}\label{eq:RX}
	RX: P(S,M) \lra \Ho(\Cb),\; \Gamma \mapsto RX({\Gamma})
\end{equation}
where $RX(\Gamma) := R \Upsilon_*X (\DG^{\Gamma})$ and $\Gamma$ is equipped with its natural
$\DG$-structured from Proposition \ref{prop:emb-graph-str}. Finally, we define
\[
	X(S,M) := \ind RX.
\]

\begin{thm}\label{thm:state-sum} The object $X(S,M)$ comes equipped with a canonical isomorphism
	\begin{equation}\label{eq:statesum}
			RX({\Gamma}) \overset{\cong}{\lra} X(S,M), 
	\end{equation}
	for every object $\Gamma$ in $P(S,M)$.
\end{thm}
\begin{proof}
	The claimed isomorphisms are simply the maps which consitute the colimit cone under the
	diagram $RX$. These maps are isomorphisms since, by Theorem \ref{thm:contractible}, the
	diagram $RX$ is indexed by a simply connected index category and, due to the $2$-Segal
	property satisfied by $X$, maps all morphisms to isomorphisms in $\Ho(\Cb)$ (see
	\cite{HSS-triangulated}).
\end{proof}

\begin{rem} Theorem \ref{thm:state-sum} shows that the structured surface invariant $X(S,M)$ can be explicitly 
	computed via the formula 
	\[
		X(S,M) \cong \hopro_{\{\DG^n \to \DG^{\Gamma}\}} X_n
	\]
	where $\Gamma$ is any chosen spanning graph embedded in $S$. This formula can be regarded as
	a categorified variant of the state sum formulas for surface invariants in the context of
	topological field theory.
\end{rem}

Consider the functor
\[
  \DGgraph \lra \Cb, \; \Gamma \mapsto RX(\Gamma).
\]
Due to the $2$-Segal property of $X$, it maps weak equivalences of $\DG$-structured graphs to weak
equivalence in $\Cb$. We can pass to classifying spaces to obtain a map of topological spaces.
\[
  |\DGgraph| \lra |W_{\Cb}|
\]
where $W_{\Cb}$ denotes the category of weak equivalences in $\Cb$. By Theorem \ref{thm:g-graph}, we
have
\[
	|\DGgraph| \simeq \coprod_{(S,M)} B\, \Mod^\GG(S,M)
\]
so that we obtain, for every stable $\GG$-structured surface $(S,M)$, a map
\[
	\rho: B\, \Mod^\GG(S,M) \lra |W_{\Cb}|.
\]
The topological space $|W_{\Cb}|$ is a model for the maximal $\infty$-groupoid in the simplicial
localization of $\Cb$ with respect to $W$ (cf. \cite{dwyer-kan}). Therefore, the map $\rho$ realizes
a coherent action of the mapping class group $\Mod^\GG(S,M)$ on $\rho(*) \cong X(S,M)$. 

\begin{thm}\label{thm:mod-act} The structured mapping class group $\Mod^\GG(S,M)$ acts coherently on
  the object $X(S,M)$.
\end{thm}

\subsection{Examples}

We sketch some explicit examples of structured $2$-Segal objects and describe the corresponding
categorified state sum invariants. The examples arise more naturally as cosimplicial objects
\[
	\F: \Delta \lra \Cb
\]
with additional structure. We say that $\F$ is $2$-coSegal if $\F^{\op}$ is $2$-Segal and translate
all other concepts in a similar way.

\subsubsection{Fundamental groupoids}
\label{sec:fundamental}

Let $\Groupoids$ denote the category of small groupoids. Given a marked unoriented surface $(S,M)$,
we define the fundamental groupoid $\Pi_1(S,M)$ to be the groupoid with set of objects $M$ and
morphisms given by homotopy classes of paths in $S$ between points in $M$. The category $\Groupoids$
is equipped with the trivial model structure so that weak equivalences are given by isomorphisms
(as opposed to equivalences) of groupoids.

Consider the codihedral groupoid
\[
	\F: \Xi \lra \Groupoids, \; [n] \mapsto \Pi_1(D, \{0,1,\dots,n\})
\]
where $\Pi_1(D,\{0,1,\dots,n\})$ denotes the fundamental groupoid of the unit disk $D$ in $\CC$, relative
to the subset $\{0,1,\dots,n\} \subset D$ of $(n+1)$st roots of unity. The dihedral functoriality is
most naturally seen using a topological model for $\Xi$: We represent a morphism $[m] \to [n]$ by a homotopy equivalence $S^1
\to S^1$ which maps $\{0,1,\dots,m\}$ to $\{0,1,\dots,n\}$. We attach disks to obtain a map of pairs
$(D,\{0,1,\dots,m\}) \to (D,\{0,1,\dots,n\})$ and pass to fundamental groupoids. 

\begin{prop}\label{prop:kampen} The codihedral groupoid $\F$ is $2$-coSegal. The value $\F(S,M)$ of $\F$ on an
	unoriented marked surface $(S,M)$ is the fundamental groupoid $\Pi_1(S,M)$ of $S$ relative
	to $M$.
\end{prop}
\begin{proof}
To verify the $2$-Segal property, we have to show that, for any subdivision of the $n+1$-gon $P_n$
into two polygons $P_m$ with vertices $\{0,1,\dots, i, j, j+1, \dots, n\}$ and $P_l$ with vertices
$\{i,i+1,\dots, j\}$, the corresponding functor
\[
	f: \Pi_1(D, \{0,1,\dots, i, j, j+1, \dots, n\}) \coprod_{\Pi_1(D, \{i,j\})} \Pi_1(D, \{i,i+1,\dots, j\}) \lra \Pi_1(D,\{0,1,\dots,n\})
\]
is an isomorphism of groupoids. Up to isomorphism of groupoids we can alternatively describe $f$
as the natural functor
\[
	f': \Pi_1(P_m) \coprod_{\Pi_1(P_1)} \Pi_1(P_l) \lra \Pi_1(P_n)
\]
induced from expressing $P_n$ as the union of $P_m$ and $P_l$ with intersection $P_m \cap P_l =
P_1$. Here, all fundamental groupoids are taken with respect to the vertices of the respective
polygon. The fact that $f'$ is an isomorphism of groupoids is an easy instance of van Kampen's theorem
(see, e.g., \cite[6.7.2]{brown}).
Similarly, to compute $\F(S,M)$, note that we can pass to the dual of a spanning graph in $S
\setminus M$ to express $(S,M)$ as a union of polygons. Interpreting formula \eqref{eq:statesum} in
these terms, another application of van Kampen's theorem shows that $\F(S,M)$ is the fundamental
groupoid $\Pi_1(S,M)$ computed as a colimit parametrized by the above polygonal subdivision of
$(S,M)$.
\end{proof}

\subsubsection{Relative homology}

Let $\Ab$ be the category of abelian groups. For $n \ge 0$, we define the abelian group $\F_n$ to be
the kernel of the homomorphism
\[
	\ZZ^{n+1} \to \ZZ,\; (z_0,z_1,\dots,z_n) \mapsto \sum z_i.
\]
These abelian groups organize into a codihedral abelian group
\[
	\F: \Xi \lra \Ab, \; [n] \mapsto \F_n.
\]
This can be seen by forming a larger codihedral abelian group given as the composite
\[
	\Xi \overset{\lambda}{\lra} \Set \lra \Ab
\]
and verifying that $\F$ is a codihedral subgroup. Here, $\lambda$ denotes the functor from
Proposition \ref{prop:forget} and $\Set \to \Ab$ is the functor of taking the free abelian group
on a set.

\begin{prop} The codihedral abelian group $\F$ is $2$-coSegal. The value $\F(S,M)$ of $\F$ on an
	unoriented marked surface $(S,M)$ is the relative homology $H_1(S,M)$ of $S$ relative
	to $M$.
\end{prop}
\begin{proof} Note that $\F_n$ can be intepreted as the relative homology of the pair
	$(D,\{0,1,\dots,\})$ from \S \ref{sec:fundamental}. The argument is now very similar to the
	proof of Proposition \ref{prop:kampen}, using the Mayer-Vietoris sequence instead of van
	Kampen's theorem.
\end{proof}

\subsubsection{Topological Fukaya categories}

Let $\k$ be a field. In \cite{HSS-triangulated}, we constructed a cocyclic differential 
$\Zt$-graded category
\[
	\F: \Lambda \lra \dgcatt_{\k}, \; [n] \mapsto \MF^{\ZZ/(n+1)}(\k[z], z^{n+1})
\]
and verified that $\F$ is $2$-coSegal with respect to the Morita model structure on $\dgcatt$. Here,
$\MF^{\ZZ/(n+1)}(\k[z], z^{n+1})$ denotes the category of $\ZZ/(n+1)$-graded matrix factorizations
of the polynomial $z^{n+1}$.
The value of $\F$ on an oriented marked surface $(S,M)$ is called the topological coFukaya category
of $(S,M)$. See \cite{HSS-triangulated} for some explicit calculations of $\F(S,M)$.

As a slight variation, there is a coparacyclic differential $\ZZ$-graded category
\[
	\F_{\infty}: \Lambda_{\infty} \lra \dgcat_{\k}, \;,  [n] \mapsto \MF^{\ZZ}(\k[z], z^{n+1})
\]
given by $\ZZ$-graded matrix factorizations.
The value of $\F_{\infty}$ on a framed marked surface $(S,M)$ provides a $\ZZ$-graded variant of the
topological coFukaya category.

Interpolating between these examples, we can define an $N$-cocyclic differential $\ZZ/2N$-graded category
\[
	\F_{N}: \Lambda_{N} \lra \dgcat_{\k}^{(2N)}, \;,  [n] \mapsto \MF^{\ZZ/N(n+1)}(\k[z], z^{n+1})
\]
so that $\F_N(S,M)$ provides a $\ZZ/2N$-graded variant of the topological coFukaya category
associated to any marked $N$-spin surface $(S,M)$.

\subsubsection{Structured nerves}

Let $\DG$ be a planar crossed simplicial group. Consider the functor
\[
	\F: \DG \lra \Gen_0-\Cat, \; [n] \mapsto \FC(\DG^n|_{\Delta})
\]
from Proposition \ref{prop:cosimplex}. The category $\Gen_0-\Cat$ can be equipped with a model
structure so that weak equivalences are morphisms which, on the underlying categories, induce
equivalences. An argument similar to \cite[7.5]{HSS1} can be used to prove the following statement.

\begin{prop} The $\DG$-object $\F$ is $2$-coSegal.
\end{prop}

The value of $\F$ on a $\GG$-structured surface can be explicitly computed in various examples.  The
question of finding a general intrinsic interpretation of $\F(S,M)$ in terms of the $\GG$-structured
surface is open.
 
\appendix
 
 \renewcommand{\thesection}{\Alph{section}}
 
 \newpage
\section{The tessellation complex and the Teichm\"uller space}
\subsection{The Stasheff polytopes and the tessellation complex}
 
 Let $P$ be a convex polygon in the plane, with the set of vertices $M$. 
 A {\em tessellation} of $P$ is a (possibly empty) set of non-intersecting diagonals in $P$.
 The set $\Tess(P)$ of tesselations of $P$ is partially ordered by inclusions of
 sets of diagonals. A tessellation $\Pc\in\Tess(P)$ subdivides $P$ into several
 sub-polygons $P'$. We write (allowing a slight abuse of notation) $P'\in\Pc$. Thus maximal elements
 of $\Tess(P)$ are triangulations: tesselations into triangles. 
 
 As well known, there is a convex polytope $\TT(P)$ canonically associated to $P$
 and called the {\em Stasheff polytope} of $P$. The poset of faces of $\TT(P)$ is
 anti-isomorphic with $\Tess(P)$. Thus, vertices of $\TT(P)$ correspond to triangulations 
 and edges to tesselations into all triangles and one 4-gon. The face of $\TT(P)$ corresponding
 to a tessellation $\Pc$, is identified with
 \[
 \TT(\Pc) \,\, := \,\, \prod_{P'\in\Pc}\TT(P'). 
 \]
 The combinatorial structure of $\TT(P)$ depends only on the combinatorial structure of $P$,
 i.e., on the number of vertices. Taking $P$ to be the regular polygon $P_{n+1}$ in $\CC=\RR^2$ with
 the set of vertices 
 \[
 M \,\,=\,\, \cn \,\, = \,\, \bigl\{ \exp(2\pi i k /(n+1)), \,\, k=0, 1, \cdots, n\bigr\}, 
 \]
 we see that the dihedral group $D_{n+1}$ acts on $\TT(P)$ by affine automorphisms.
 This means that we can associate a Stasheff polytope $\TT_M$ to any finite dihedral ordinal
 $M\in\bf\Xi$, $|M|\geq 3$. 
 
 More generally, by a {\em (curvilinear) polygon} we will mean a stable marked
 surface $(S,M)$ such that $S$ is, topologically,  a disk and $M\subset\partial S$. In  this
 case $M$ has a canonical dihedral order and we associate to $(S,M)$ the
 Stasheff polytope $\TT(S,M):= \TT_M$. 
 
 Let now $(S,M)$ be an arbitrary stable marked surface. Following \cite{harer, fomin-shapiro-thurston},
 we call an {\em arc} on $(S,M)$ an isotopy class of unoriented paths $\gamma$ beginning
 and ending on $M$ (possibly at the same point) such that:
 \begin{enumerate}
 \item[(1)] $\gamma$ does not meet any  elements of $M$ other than containing them as endpoints.
 
 \item[(2)] $\gamma$ does not have any self-intersection points  except possibly endpoints which
 coincide. 
 
 \item[(3)] $\gamma$ is not homotopic (rel. endpoints) to a point, nor to a boundary segment
 between two adjacent marked points. 
 
 \end{enumerate}
 
 We denote by $\Ac(S,M)$ the set of arcs for $(S,M)$. Thus, in the case when $(S,M)$ is a polygon, $\Ac(S,M)$
 is identified with the set of diagonals.
 
 A {\em sub-tesselation} (or an {\em arc system}) of $(S,M)$ is a finite subset $\Pc\subset \Ac(S,M)$ which can
 be represented by a system of pairwise non-intersecting paths. Note that these paths must also be 
 pairwise non-homotopic, since they represent different elements of $\Ac(S,M)$. A system of paths
 representing a sub-tesselation $\Pc$
 cuts $S$ into open pieces $S_i^\circ$. By the standard ``bordification" procedure we complete
 each $S_i^\circ$ to a compact surface with boundary $S'_i$ and equip it with a marking
 $M'_i$ formed by points of $M$ lying inside or on the boundary of $S_i^\circ$. Here we use the standard 
 conventions about degenerate cases:
 \begin{enumerate}
 \item[(1)] If a component $S_i^\circ$ approaches an arc $\gamma$ on both sides, $\gamma$ contributes two
 intervals to $\partial S'_i$.
 
 \item[(2)] If, moreover. an endpoint of $\gamma$ is not fully encircled by $S^\circ_i$, it contributes two
 elements to $M'_i$. See Figure \ref{fig:F3}. 
 \end{enumerate}
 
 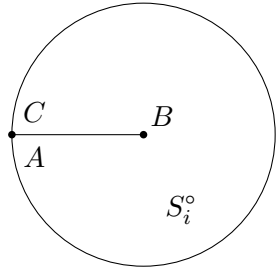
\begin{figure}[h!]
 \centering
 \begin{tikzpicture}[scale =0.5]
 
 \node (B2) at (20,0){}; 
\fill (B2) circle (0.1);
\draw (B2) circle (3.5);

\node(A2) at (16.5,0){};
\fill (A2) circle (0.1); 
\draw (16.5,0) -- (20,0); 

\node at (20.5,0.5){$B$}; 

\node at (17.1 , 0.6){$C$}; 
\node at (17.1 , -0.6){$A$}; 

\node at (21, -2){$S^\circ_i$}; 
 \end{tikzpicture}
 \caption{$\partial S'_i$ consists of three segments $[A,B], [B,C], [C,A]$ and $M_i'=\{A,B,C\}$. 
 \label{fig:F3}
 }
 \end{figure}
 Note that each $(S'_i, M'_i)$ is a stable marked surface. By a slight abuse of notation we write
 $(S'_i, M'_i)\in\Pc$. 
 
 \begin{defi}
A {\em tessellation} of $(S,M)$ is a sub-tesselation $\Pc$ such that each $(S'_i, M'_i)\in\Pc$ is a polygon. 
 For a tesselation $\Pc$ we define its Stasheff polytope as the product
 \[
 \TT(\Pc) \,\,=\,\, \prod_{(S', M')\in\Pc} \TT(S', M').
 \]
 \end{defi}
 
 The set $\Tess(S,M)$ of tesselations of $(S,M)$ is partially ordered by inclusion of subsets in
 $\Ac(S,M)$. Maximal elements are {\em triangulations} of $(S,M)$: tesselations for which
 each $(S',M')$ is a triangle, i.e., $|M'|=3$, $M'\subset\partial S'$. If $\Pc\let \Pc'$, then
 $\TT(\Pc')$ is canonically identified with a face of $\TT(\Pc)$.
 
 \begin{defi}
 The {\em tessellation complex} $\TT(S,M)$ is a polyhedral complex obtained by gluing the polytopes $\TT(\Pc)$ for
 all $\Pc\in\Tess(S,M)$ using the face embeddings $\TT(\Pc')\hookrightarrow \TT(\Pc)$ for $\Pc\leq\Pc'$> 
 \end{defi}
 
 Note that unlike the above references as well as \cite{HSS-triangulated, fock-goncharov}, we do not assume 
 that $S$ is orientable. 
 
 \vskip .2cm
 
 Given a tessellation $\Pc\in\Tess(S,M)$, we associate to it its {\em dual graph} $\Gamma_\Pc$
 by putting one vertex $v_{S',M'}$ inside  each polygon $(S',M')\in\Pc$ and connected the
 adjacent $v_{(S',M')}$ be edges, similarly to \cite[\S 3.3.3]{HSS-triangulated}. When $\Pc$ varies,
 the $\Gamma_\Pc$ run over all isotopy classes of spanning graphs for $(S,M)$.
 This implies the following. 
 
 \begin{prop}\label{prop:tess-graph}
 $\TT(S,M)$ is homotopy equivalent to the nerve of the category $P(S,M)$. 
 \qed
 \end{prop}
 
 We now explain  how to prove the following fact which has been widely used in the oriented case, e.g., in
 \cite {HSS-triangulated, fock-goncharov}. 
 
 \begin{thm}\label{thm:contractible}
	For any stable marked surface $(S,M)$ (orientable or not) the tessellation complex $\TT(S,M)$ is contractible. 
\end{thm}
  
  Similarly to the oriented case, the proof can be extracted from the results of Harer \cite{harer},
  suitably adjusted,  as we now explain.

 \subsection {The dual cell complex of a triangulated manifold}
 
 By a {\em (regular) simplicial complex}  we will mean a datum $A=(V,\Sigma)$
 where $V$ is a set whose elements are called {\em vertices} of $A$, and $\Sigma\subset 2^V$
 is a family of finite subsets called {\em simplices} of $A$. It is required that:
 \begin{itemize}
 \item All 1-element subsets of $V$ are in $\Sigma$.
 
 \item If $I\in \Sigma$ and $I'\subset I$, then $I'\in\Sigma$. 
 \end{itemize}
 
 The {\em realization} $|A|$ of a simplicial complex $X$ is obtained in a standard way by
 gluing the geometric simplices $\Delta^I$ associated to $I\in\Sigma$. Explicitly,
 \[
 \Delta^I \,\,=\,\,\left\{ (p_a)_{a\in I}\in\RR^I \bigl| \,\, p_a\geq 0, \,\,\sum_{a\in I} p_a =1\right\}. 
 \]
 A {\em simplicial subcomplex} of $A=(V,\Sigma)$ is a simplicial complex $A'=(V',\Sigma')$
 s.t. $V'\subset V$ and $\Sigma'\subset\Sigma$. In this case $|A'|$ is a closed subset of $|A|$. 
 
 By a {\em space} we will always mean a topological space homeomorphic to $|A|-|A'|$ where
 $A$ is a simplicial complex and $A'\subset A$ is a subcomplex.
 
 \begin{exas}

 (a) Let $X$ be a space and $\Uc = (U_a)_{a\in V}$ be a locally finite covering of $X$
 by closed subspaces $U_a$. The {\em classical nerve}  of $\Uc$ is the simplicial complex
 $\N\Uc=(V, \Sigma)$ where $\Sigma$ consists of finite $I\subset V$ s.t. $\bigcap_{v\in I} U_v\neq\emptyset$.

  \vskip .2cm
  
  (b) In the situation (a), assume that the covering $\Uc$ is {\em saturated}, i.e., for any $U,U'\in \Uc$
  the intersection $U\cap U'$ is either empty or belongs to $\Uc$. The {\em monotone nerve} of $\Uc$
  is the simplicial complex
  $\N^\leq\Uc = (V.\Sigma^\leq)$, where $\Sigma^\leq$ consists of $I=\{v_0, \cdots, v_p\}\subset V$
  such that $U_{v_{s(0)}}\subset U_{v_{s(1)}} \subset \cdots\subset U_{v_{s(p)}}$ for some permutation
  $s$ of $\{0,1,\cdots, p\}$. Thus $|\N^\leq(\Uc)|$ is the geometric realization of the simplicial
  set, given by the nerve of the poset of nonempty $U\in\Uc$. 
  \end{exas}
 
 The following is classical.

 \begin{prop}\label{prop:nerve-classical}
 \begin{enumerate}
 \item[(a)] Assume that any finite intersection $U_{v_0}\cap\cdots\cap U_{v_p}$ is either
 empty or contractible. Then $|\N\Uc|$ is homotopy equivalent to $X$.
 
 \item[(b)] If, moreover, $\Uc$ is saturated, then $|\N^\leq \Uc|$ is homotopy
 equivalent to $X$. \qed

 \end{enumerate}
 
 \end{prop}
 
 The {\em barycentric subdivision} of a simplicial complex $A=(V,\Sigma)$ is the simplicial complex
 $\on{Bar}(A)=(V_\Bar,\Sigma_\Bar)$, where $V_\Bar=\Sigma$ and $\Sigma_\Bar$ consists of finite subsets
 $\{I_0, \cdots, I_p\}$ such that $I_{s(0)} \subset I_{s(1)} \subset \cdots\subset I_{s(p)}$ 
for some permutation
  $s$ of $\{0,1,\cdots, p\}$. As well known, $|\Bar(A)|$ is canonically homeomorphic to $|A|$.
  Notationally, it will be convenient for us to think that vertices of $\Bar(A)$ are formal
  symbols  $\bar(I)$ for $I\in\Sigma$. 
  
  For each simplex $I\in\Sigma$ we have the sub complex $A_I\subset\Bar(A)$ called the
  {\em link} of $I$. By definition, $A_I=(V_I, \Sigma_I)$, where
  \[
  V_I \,\,=\,\,\bigl\{ \bar(J)\bigl| \,\, J\supset I,\,\, J\in\Sigma \bigr\}, \quad 
  \Sigma_I \,\,=\,\, \bigl\{ \{ \bar(J_0), \cdots, \bar(J_p)\}, \,\,\, I\subset J_0\subset\cdots \subset J_p\bigr\}. 
  \]
  If $A'$ is a subcomplex of $A$ as above, we define the {\em dual subcomplex} to $(A,A')$
  as
  \[
  (A,A')^\vee \,\,=\,\,\bigcup_{I\in\Sigma-\Sigma'} A_I \,\,\,\subset \,\,\, \Bar(A). 
  \]
  
  \begin{prop}\label{prop:dual-CW}
  The realization $|(A,A')^\vee|$ is homotopy equivalent to $|A|-|A'|$. 
  
  \end{prop}
  
  \noindent {\sl Proof:} We consider the covering $\Uc$ of  $|A|-|A'|$ by closed subsets $\Delta^I_0 =
  \Delta^I \cap (|A|-|A'|)$ for $I\in\Sigma-\Sigma'$. This covering is saturated, and $(A,A')^\vee=\N^\leq\Uc$.
  As each $\Delta^I_0$ is contractible, our statement follows from Proposition 
  \ref{prop:nerve-classical}. \qed
  
  \vskip .2cm
  
  By a {\em triangulated manifold} we mean a pair $(A, A')$ consisting of a simplicial complex $A$ and
  a subcomplex $A'$ such that $|A|-|A'|$ is a topological manifold. In this case each $|A_I|$, $I\in\Sigma-\Sigma'$,
  is a topological cell which we call the {\em dual cell} to $I$. Thus $|(A,A')^\vee|$ is equipped with a cell
  decomposition with cells $|A_I|$. 
  
  \subsection{Harer's triangulation of the Teichm\"uller space}
  \label{subsub:harer} 
  
  Let $(S,M)$ be a stable marked surface. The {\em arc complex} $A=A(S,M)$ is the
  simplicial complex with the set of vertices $\Ac(S,M)$ and simplices being sub-tesselations.
  It has a subcomplex $A_\infty=A_\infty(S,M)$ formed by sub-tesselations which are not tesselations.
  
  \begin{prop}\label{prop:tess-A}
  The tessellation complex $\TT(S,M)$ is homeomorphic to $|(A, A_\infty)^\vee|$ so that each
  cell $\TT(\Pc)$ is identified with $A_\Pc$, the link of $\Pc$.
  \end{prop}
  
  \noindent {\sl Proof:} Clear by comparing the definitions. \qed
  
  \vskip .2cm
  
  The following result was stated and proved by Harer \cite{harer} in the oriented case.
  
 \begin{thm}\label{thm:harer-triang}
 There is a homeomorphism $\Psi=\Psi_{S,M}: |A|-|A_\infty| \to \Teich(S,M)$,
 equivariant with respect to the mapping class group $\Mod(S,M)$. 
 \end{thm} 
 
 \noindent {\sl Proof: }
  The construction and the argument of \cite{harer}  extend to the general case as follows.
  A point $p\in |A|-|A_\infty|$ lies, by definition, strictly inside a simplex $\Delta^\Pc$ for some
  tessellation $\Pc$ and thus is represented by the collection of its barycentric coordinates
  \[
  (p_\gamma)_{\gamma\in\Pc}, \,\, p_\gamma > 0,\,\, \sum p_\gamma = 1,
  \]
  one for each arc $\gamma\in\Pc$. 
  
  Consider a polygon $(S'_i, M'_i)$ of $\Pc$. To each edge $e$ of $S'_i$ we associate the ``length"
  $p_e:=p_\gamma$ where $\gamma$ is the arc of $\Pc$ that gave rise to $e$. Let $l_e=\sum_{e\subset\partial S'_i} p_e$
  be the total ``length" of $\partial S'_i$. 
  
  Let $S''_i$ be a disk in $\CC$ with center $0$ and circumference $l_e$. We can map $M'_i$ bijectively
  to a subset $M''_i\subset \partial S''_i$ so that the distances (on $\partial S''_i$) between neighboring
  elements of $M''_i$ match the corresponding numbers $p_e$ for the edges of $S'_i$. In this way we
  get
  a stable marked Klein surface $(S''_i, M''_i)$. We now glue the $(S''_i, M''_i)$ together in the same
  fashion as the $(S'_i, M'_i)$ are glued together to form $(S, M)$. For this, we use the identifications
  of the boundary arcs of different $S''_i$ which are affine linear in the standard angle (or arc length)
  coordinates on the circles $\partial S''_i$. In this way we get a Klein surface $(S'', M'')$, identified with
  $(S,M)$ by a diffeomorphism unique up to isotopy. In other words, we get a point of $\Teich(S,M)$.
  This defines the map $\Psi_{S,M}$. Its $\Mod(S,M)$-equivariance is clear.
  
  The fact that $\Psi_{S,M}$ is a homeomorphism, was proved in \cite{harer} in the oriented case.
  Suppose $S$ is non-orientable, and let $\varpi: (\widetilde S, \widetilde M)\to (S,M)$
  be its orientation cover, with the deck involution $\tau$ preserving $\widetilde M$.
  By equivariance, the homeomorphism
  \[
  \Psi_{\widetilde S, \widetilde M}: \,\, |A(\widetilde S, \widetilde M)|-|A_\infty(\widetilde S, \widetilde M)| 
  \lra \Teich (\widetilde S, \widetilde M)
  \]
  sends  the involution $\tau_A: |A(\widetilde S, \widetilde M)|\to |A(\widetilde S, \widetilde M)|$
  induced by $\tau$,  to $\tau_{\Teich}$.
  Because of Proposition \ref{thm:klein-teich}(b2), we are reduced to the following.
  
  \begin{lem}
  We have
  \[
  \bigl( |A(\widetilde S, \widetilde M)|-|A_\infty(\widetilde S, \widetilde M)|\bigr)^{\tau_A} \,\,=\,\,
  |A(S,M)| - |A_\infty(S,M)|. 
  \]
  \end{lem}
  
  \noindent {\sl Proof of the lemma:} A point $p\in |A(\widetilde S, \widetilde M)|^{\tau_A}$ lies
  in a $\tau$-invariant simplex $\Delta^{\widetilde\Pc}$ corresponding to a $\tau$-invariant
  sub-tesselation $\widetilde\Pc$ of $(\widetilde S, \widetilde M)$. Let $\Pc$ be the set of
  arcs $\varpi(\widetilde\gamma)$, $\widetilde\gamma\in\widetilde\Pc$. 
  We claim that $\Pc$ is a sub-tesselation for $(S,M)$. Indeed, the fact that $\widetilde \gamma$ does not
  meet $\tau(\widetilde \gamma)$ (which both belong to $\widetilde\Pc$) means that $\varpi(\widetilde\gamma)$
  does not intersect itself. Further, a homotopy  between $\varpi(\widetilde\gamma_1)$ and 
  $\varpi(\widetilde\gamma_2)$ implies that that $\widetilde \gamma_2$ is homotopic to either
  $\widetilde \gamma_1$ or to $\tau(\widetilde \gamma_1)$ (because $\varpi$ is an unramified covering).
  So $\Pc$ is a sub-tesselation. 
  
  Further, suppose that $\widetilde\Pc$ is a tessellation, Then any polygon $(\widetilde S'_i, \widetilde M'_i)$
  of $\widetilde\Pc$ must be an unramified covering of its image which therefore must also be a polygon. 
  Therefore $\Pc$ is a tessellation. 
  
  We conclude that 
  \[
  p\,\,\in \,\,  \bigl( |A(\widetilde S, \widetilde M)|-|A_\infty(\widetilde S, \widetilde M)|\bigr)^{\tau_A}
  \]
  must lie in a simplex $\Delta^{\widetilde\Pc}$ where $\widetilde\Pc=\varpi^{-1}(\Pc)$
  is the preimage of a tessellation $\Pc$ of $(S,M)$ and be given by a collection of barycentric coordinates
  $(p_{\widetilde\gamma})_{{\widetilde\gamma}\in\widetilde\Pc}$ which is $\tau$-invariant. Such collections
  of coordinates are in bijection (homeomorphism) with collections of barycentric coordinates
  $(p_\gamma)_{\gamma\in\Pc}$, $\sum p_\gamma=1$, i.e., with points of the
  simplex $\Delta^\Pc$ in $|A(S,M)| - |A_\infty(S,M)|$. This proves the lemma and Theorem 
  \ref{thm:harer-triang}. 
  
  Now, Propositions \ref{prop:tess-graph},  \ref{prop:tess-A} and \ref{prop:dual-CW} together with the fact that
  $\Teich(S,M)$ is homeomorphic to a Euclidean space (Theorem \ref{thm:klein-teich}(c))
  imply Theorem \ref{thm:graph-contr}. 
 
 \vfill\eject

\end{document}